\documentclass[11pt, reqno]{amsart}
\usepackage{amsfonts,latexsym,enumerate}
\usepackage{amsmath}
\usepackage{amscd} 
\usepackage{float,amsmath,amssymb,mathrsfs,bm,multirow,graphics}
\usepackage[dvips]{graphicx}
\usepackage[percent]{overpic}
\usepackage{amsaddr}
\usepackage[numbers,sort&compress]{natbib}
\usepackage{xcolor}
\usepackage{todonotes}

\newcommand{\R}{{\Bbb R}}

\newcommand{\C}{{\Bbb C}}
\newcommand{\D}{{\Bbb D}}

\newcommand{\IV}{\text{\upshape IV}}

\newcommand{\res}{\text{\upshape Res\,}}
\newcommand{\diag}{\text{\upshape diag\,}}
\newcommand{\re}{\text{\upshape Re\,}}
\newcommand{\im}{\text{\upshape Im\,}}

\newcommand{\sol}{\text{\upshape sol}}

\allowdisplaybreaks

\newtheorem{theorem}{Theorem}[section]

\newtheorem{lemma}[theorem]{Lemma}

\newtheorem{RHproblem}[theorem]{RH problem}
\newtheorem{figuretext}{Figure}

\numberwithin{equation}{section}

\usepackage[colorlinks=true]{hyperref}
\hypersetup{urlcolor=blue, citecolor=red, linkcolor=blue}

\usepackage{tikz}
\usepackage{tkz-fct}
\usepackage{pict2e}
\usetikzlibrary{arrows,calc,chains, positioning, shapes.geometric,shapes.symbols,decorations.markings,arrows.meta}
\usetikzlibrary{patterns,angles,quotes}
\usetikzlibrary{decorations.markings}
\tikzset{middlearrow/.style={
			decoration={markings,
				mark= at position 0.6 with {\arrow{#1}} ,
			},
			postaction={decorate}
		}
	}
\tikzset{->-/.style={decoration={
				markings,
				mark=at position #1 with {\arrow{latex}}},postaction={decorate}}}
	
\tikzset{-<-/.style={decoration={
				markings,
				mark=at position #1 with {\arrowreversed{latex}}},postaction={decorate}}}
				
				\tikzset{
	master/.style={
		execute at end picture={
			\coordinate (lower right) at (current bounding box.south east);
			\coordinate (upper left) at (current bounding box.north west);
		}
	},
	slave/.style={
		execute at end picture={
			\pgfresetboundingbox
			\path (upper left) rectangle (lower right);
		}
	}
}
\tikzset{cross/.style={cross out, draw, 
         minimum size=2*(#1-\pgflinewidth), 
         inner sep=0pt, outer sep=0pt}}

\def\Xint#1{\mathchoice
{\XXint\displaystyle\textstyle{#1}}%
{\XXint\textstyle\scriptstyle{#1}}%
{\XXint\scriptstyle\scriptscriptstyle{#1}}%
{\XXint\scriptscriptstyle\scriptscriptstyle{#1}}%
\!\int}
\def\XXint#1#2#3{{\setbox0=\hbox{$#1{#2#3}{\int}$ }
\vcenter{\hbox{$#2#3$ }}\kern-.59\wd0}}

\def\dashint{\Xint-}
\title[Boussinesq's equation for water waves: Sector IV]
{Boussinesq's equation for water waves: \\ the soliton resolution conjecture for Sector IV}

\author{C. Charlier$^{1}$ and J. Lenells$^{2}$}

\address{$^{1}$Centre for Mathematical Sciences, Lund University,
22100 Lund, Sweden. \\
$^{2}$Department of Mathematics, KTH Royal Institute of Technology, \\
10044 Stockholm, Sweden.}
\email{christophe.charlier@math.lu.se}
\email{jlenells@kth.se}

\begin{document}

\begin{abstract}
We consider the Boussinesq equation on the line for a broad class of Schwartz initial data relevant for water waves. In a recent work, we identified ten main sectors describing the asymptotic behavior of the solution, and for each of these sectors we gave an exact expression for the leading asymptotic term in the case when no solitons are present. In this paper, we derive an asymptotic formula in Sector IV, characterized by $\frac{x}{t}\in (\frac{1}{\sqrt{3}},1)$, in the case when solitons are present. In particular, our results provide an exact expression for the soliton-radiation interaction to leading order and a verification of the soliton resolution conjecture for the Boussinesq equation in Sector IV.
\end{abstract}

\maketitle

\noindent
{\small{\sc AMS Subject Classification (2020)}: 35C08, 35G25, 35Q15, 37K40, 76B15.}
%37K15

\noindent
{\small{\sc Keywords}: Boussinesq equation, long-time asymptotics, Riemann-Hilbert problem.}

%\setcounter{tocdepth}{1}
%\tableofcontents

\section{Introduction}

The Boussinesq equation \cite{B1872}
\begin{align}\label{boussinesq}
u_{tt} = u_{xx} + (u^2)_{xx} + u_{xxxx},
\end{align}
where $u(x,t)$ is a real-valued function and subscripts denote partial derivatives, models small-amplitude dispersive waves in shallow water propagating in both the right and left directions, see e.g. \cite{J1997}. Equation (\ref{boussinesq}) supports solitons \cite{H1973, B1976, BZ2002} and admits a Lax pair \cite{Z1974}. 
Until recently \cite{CLmain}, the determination of the long-time asymptotics for the solution of \eqref{boussinesq} was still an open problem, see e.g. Deift's list of open problems in \cite{D2008}.

In \cite{CLmain}, we studied the direct and inverse scattering problems for \eqref{boussinesq}. Among other things, we established that for a wide class of initial data relevant for water waves, the solution is unique and exists globally. We identified ten main asymptotic sectors and for each of these sectors, we computed the leading asymptotics of the solution. The proofs of some of these asymptotic formulas were omitted in \cite{CLmain} for conciseness. Moreover, only solitonless solutions were considered.

The purpose of this paper is to derive an asymptotic formula for the solution in the sector $\frac{x}{t} \in (\frac{1}{\sqrt{3}}, 1)$ in the case when solitons are present. This sector was referred to as Sector IV in \cite{CLmain}. In the special case when solitons are absent, our formula reduces to the formula announced in \cite{CLmain}, thus providing a proof of this formula.
In the case when solitons are present, our formula quantifies the asymptotic effect of solitons on the solution in Sector IV. 

The soliton resolution conjecture for an integrable equation expresses the expectation that a solution of the equation with generic initial data eventually evolves into a set of solitons superimposed on a dispersive radiative background. Since $x/t < 1$ in Sector IV and all one-soliton solutions of (\ref{boussinesq}) travel with speeds greater than $1$ \cite{H1973}, the soliton resolution conjecture for (\ref{boussinesq}) suggests that there should be no order $1$ contributions to the asymptotics in Sector IV stemming from solitons. This is indeed what we find: our formula shows that the solution of (\ref{boussinesq}) in Sector IV is $O(t^{-1/2})$ as $t \to \infty$. Our results therefore verify the soliton resolution conjecture for the Boussinesq equation in this sector. We stress that although the solitons themselves are not directly observable in Sector IV for large times, they do have a nontrivial effect on the radiation in Sector IV; we compute the leading order contribution of this effect exactly. 

It is important to note that we restrict ourselves to initial data which contain no high-frequency modes (in a sense made precise in Assumption $(\ref{nounstablemodesassumption})$ below). This assumption is consistent with the derivation of the Boussinesq equation as a model for water waves, which assumes that high-frequency Fourier modes are absent, or at least highly suppressed. In fact, if high-frequency modes are present, the solution will contain components that grow exponentially in time and the soliton resolution conjecture will automatically fail.

Our main result is presented in Section \ref{mainsec}. Its proof relies on a Deift--Zhou \cite{DZ1993} steepest descent analysis of a row-vector Riemann--Hilbert (RH) problem, and this RH problem is stated in Section \ref{RHnsec}. Section \ref{overviewsec} contains a brief overview of the proof. In Sections \ref{nton1sec} and \ref{n1ton2sec}, we carry out two transformations of the RH problem. After introducing a global parametrix in Section \ref{globalparametrixsec}, we perform two further transformations in Sections \ref{n2ton3sec} and \ref{n3ton4sec}. Using two local parametrices constructed in Sections \ref{localparametrixsec1} and \ref{localparametrixsec2}, respectively, we arrive at a small-norm RH problem whose large $t$ behavior is analyzed in Section \ref{n3tonhatsec}. Finally, the long-time asymptotics of the solution $u(x,t)$ of (\ref{boussinesq}) is obtained in Section \ref{usec}.

\subsection{Notation}\label{notationsubsec}
We use the following notation throughout the paper.

\begin{enumerate}[$-$]
\item $C>0$ and $c>0$ denote generic constants that may change within a computation.

\item $[A]_1$, $[A]_2$, and $[A]_3$ denote the first, second, and third columns of a $3 \times 3$ matrix $A$.

\item If $A$ is an $n \times m$ matrix, we define $|A| \ge 0$ by $|A|^2=\Sigma_{i,j}|A_{ij}|^2$. %Note that $|A + B| \leq |A| + |B|$ and $|AB| \leq |A| |B|$.
For a piecewise smooth contour $\gamma \subset \C$ and $1 \le p \le \infty$, we define $\|A\|_{L^p(\gamma)} := \| |A|\|_{L^p(\gamma)}$.

\item $\D = \{k \in \C \, | \, |k| < 1\}$ denotes the open unit disk and $\partial \D = \{k \in \C \, | \, |k| = 1\}$ denotes the unit circle. 

\item $\mathcal{S}(\R)$ denotes the Schwartz space of rapidly decreasing functions on $\R$.

\item $\mathcal{Q} := \{\kappa_{j}\}_{j=1}^{6}$ and $\hat{\mathcal{Q}} := \mathcal{Q} \cup \{0\}$, where $\{\kappa_{j} = e^{\frac{\pi i(j-1)}{3}}\}_1^6$ are the sixth roots of unity.

\item $\{D_n\}_{n=1}^6$ denote the open subsets of the complex plane shown in Figure \ref{fig: Dn}.

\item $\Gamma = \cup_{j=1}^9 \Gamma_j$ denotes the contour shown and oriented as in Figure \ref{fig: Dn}.
% where $\Gamma_{j} = e^{\frac{\pi i(j-1)}{3}}\big( (0,i)\cup (-i,-i\infty) \big)$ for $j = 1, \dots, 6$ 

\item $\hat{\Gamma}_{j} = \Gamma_{j} \cup \partial \D$ denotes the union of $\Gamma_j$ and the unit circle. 

\end{enumerate}

\section{Main results}\label{mainsec}

We consider the initial value problem for (\ref{boussinesq}) with initial data
\begin{align}\label{initial data}
u(x,0)=u_{0}(x), \qquad u_{t}(x,0) = u_{1}(x),
\end{align}
where $u_0, u_1 \in \mathcal{S}(\mathbb{R})$ are real-valued functions in the Schwartz class.
We assume that $u_1$ obeys the condition $\int_{\mathbb{R}}u_{1}(x)dx = 0$; this assumption is natural because it ensures that the total mass $\int_{\mathbb{R}}u(x,t)dx$ is conserved in time. A result of \cite{CLmain} states that the solution to the initial value problem (\ref{boussinesq})--(\ref{initial data}) can be expressed in terms of the solution $n$ of a row-vector RH problem. 
The jump matrix of this RH problem is expressed in terms of two scalar reflection coefficients, $r_{1}(k)$ and $r_{2}(k)$, which are defined as follows.

\subsection{Definition of $r_{1}$ and $r_{2}$.} Let $\omega := e^{\frac{2\pi i}{3}}$ and define $\{l_j(k), z_j(k)\}_{j=1}^3$ by
\begin{align}\label{lmexpressions}
& l_{j}(k) = i \frac{\omega^{j}k + (\omega^{j}k)^{-1}}{2\sqrt{3}}, \qquad z_{j}(k) = i \frac{(\omega^{j}k)^{2} + (\omega^{j}k)^{-2}}{4\sqrt{3}}, \qquad k \in \C\setminus \{0\}.
\end{align}
Let $P(k)$ and $\mathsf{U}(x,k)$ be given by
\begin{align*}
P(k) = \begin{pmatrix}
1 & 1 & 1  \\
l_{1}(k) & l_{2}(k) & l_{3}(k) \\
l_{1}(k)^{2} & l_{2}(k)^{2} & l_{3}(k)^{2}
\end{pmatrix}, \quad \mathsf{U}(x,k) = P(k)^{-1} \begin{pmatrix}
0 & 0 & 0 \\
0 & 0 & 0 \\
-\frac{u_{0x}}{4}-\frac{iv_{0}}{4\sqrt{3}} & -\frac{u_{0}}{2} & 0
\end{pmatrix} P(k),
\end{align*} 
where $v_{0}(x) = \int_{-\infty}^{x}u_{1}(x')dx'$. Let $X, X^A, Y, Y^A$ be the unique solutions of the Volterra integral equations
\begin{align*}  
 & X(x,k) = I - \int_x^{\infty} e^{(x-x')\widehat{\mathcal{L}(k)}} (\mathsf{U}X)(x',k) dx',
	\\
 & X^A(x,k) = I + \int_x^{\infty} e^{-(x-x')\widehat{\mathcal{L}(k)}} (\mathsf{U}^T X^A)(x',k) dx',	
 	 \\\nonumber
& Y(x,k)  =  I  +  \int_{-\infty}^x  e^{(x-x')\widehat{\mathcal{L}(k)}} (\mathsf{U}Y)(x',k) dx', 
	\\ \nonumber
& Y^A(x,k)  =  I  -  \int_{-\infty}^x  e^{-(x-x')\widehat{\mathcal{L}(k)}} (\mathsf{U}^T Y^A)(x',k) dx', 
\end{align*}
where $\mathcal{L} = \diag(l_1 , l_2 , l_3)$, $(\cdot)^T$ denotes the transpose operation, and $e^{\hat{\mathcal{L}}}$ acts
on a $3\times 3$ matrix $M$ by $e^{\hat{\mathcal{L}}}M=e^{\mathcal{L}}Me^{-\mathcal{L}}$. Define $s(k)$ and $s^A(k)$ by 
\begin{align*}
& s(k) = I - \int_\R e^{-x \widehat{\mathcal{L}(k)}}(\mathsf{U}X)(x,k) dx, & & s^A(k) = I + \int_\R e^{x \widehat{\mathcal{L}(k)}}(\mathsf{U}^T X^A)(x,k) dx.
\end{align*}
The two spectral functions $\{r_j(k)\}_1^2$ are defined by
\begin{align}\label{r1r2def}
\begin{cases}
r_1(k) = \frac{(s(k))_{12}}{(s(k))_{11}}, & k \in \hat{\Gamma}_{1}\setminus \hat{\mathcal{Q}},
	\\ 
r_2(k) = \frac{(s^A(k))_{12}}{(s^A(k))_{11}}, \quad & k \in \hat{\Gamma}_{4}\setminus \hat{\mathcal{Q}}.
\end{cases}
\end{align}

\subsection{Solitons}
In \cite{CLscatteringsolitons}, we extended the approach of \cite{CLmain} to the case when solitons are present. In this case, in addition to $r_{1}(k)$ and $r_{2}(k)$, the formulation of the RH problem also involves a set of poles $\mathsf{Z}$ and a set of corresponding residue constants $\{c_{k_0}\}_{k_0 \in \mathsf{Z}} \subset \C$. In what follows, we describe how $\mathsf{Z}$ and $\{c_{k_0}\}_{k_0 \in \mathsf{Z}} \subset \C$ are defined.

Solitons are generated by the possible zeros of the functions $s_{11}(k)$ and $s^A_{11}(k)$. 
As shown in \cite{CLmain}, $s_{11}$ and $s^A_{11}$ have analytic extensions to the interiors of $\bar{D}_1 \cup \bar{D}_2$ and $\bar{D}_4 \cup \bar{D}_5$, respectively. We will restrict ourselves to the generic case when $s_{11}$ and $s^A_{11}(k)$ have no zeros on the contour $\Gamma$.
The symmetries $s_{11}(k) = s_{11}(\omega/k)$ and $s_{11}^A(k) = \overline{s_{11}(\bar{k}^{-1})}$ (also established in \cite{CLmain}) then imply that it is enough to consider the zeros of $s_{11}$ in $D_{2}$. 
We decompose the open set $D_{2}$ into three parts: $D_{2} = D_{\mathrm{reg}} \sqcup D_{\mathrm{sing}} \sqcup (D_{2}\cap \R)$, where 
\begin{align*}
& D_{\mathrm{reg}} := D_{2} \cap \big( \{k \,|\, |k|>1, \im k >0\} \cup \{k \,|\, |k|<1, \im k <0\} \big), \\
& D_{\mathrm{sing}} := D_{2} \cap \big( \{k \,|\, |k|>1, \im k <0\} \cup \{k \,|\, |k|<1, \im k >0\} \big),
\end{align*}
and $D_{2}\cap \R = (-1,0) \cup (1,\infty)$.
We denote the right and left parts of $D_{\mathrm{reg}}$ by $D_{\mathrm{reg}}^{R}$ and $D_{\mathrm{reg}}^{L}$, respectively, and similarly for $D_{\mathrm{sing}}$, see Figure \ref{fig: D2 splitting}.

\begin{figure}
\begin{center}
\begin{tikzpicture}[scale=0.7]
\node at (0,0) {};
\draw[line width=0.25 mm, dashed,fill=gray!30] (0,0) --  (210:2.5) arc(210:180:2.5) -- cycle;
\draw[line width=0.25 mm, dashed,fill=gray!30] (0,0) --  (30:4) arc(30:0:4) -- cycle;
\draw[line width=0.25 mm, dashed,fill=white!100] (0,0) --  (30:2.5) arc(30:0:2.5) -- cycle;
\draw[white,line width=0.35 mm] ([shift=(0:4cm)]0,0) arc (0:30:4cm);
\draw[white,line width=0.35 mm] (0,0)--(2.5,0);

\draw[black,line width=0.45 mm] (0,0)--(30:4);
\draw[black,line width=0.45 mm] (0,0)--(90:4);
\draw[black,line width=0.45 mm] (0,0)--(150:4);
\draw[black,line width=0.45 mm] (0,0)--(-30:4);
\draw[black,line width=0.45 mm] (0,0)--(-90:4);
\draw[black,line width=0.45 mm] (0,0)--(-150:4);

\draw[black,line width=0.45 mm] ([shift=(-180:2.5cm)]0,0) arc (-180:180:2.5cm);

\node at (193:1.85) {\footnotesize{$D_{\mathrm{reg}}^{L}$}};
\node at (15:3.3) {\footnotesize{$D_{\mathrm{reg}}^{R}$}};

\draw[fill] (0:2.5) circle (0.1);
\draw[fill] (60:2.5) circle (0.1);
\draw[fill] (120:2.5) circle (0.1);
\draw[fill] (180:2.5) circle (0.1);
\draw[fill] (240:2.5) circle (0.1);
\draw[fill] (300:2.5) circle (0.1);

\end{tikzpicture}
\hspace{1.7cm}
\begin{tikzpicture}[scale=0.7]
\node at (0,0) {};
\draw[line width=0.25 mm, dashed,fill=gray!30] (0,0) --  (180:2.5) arc(180:150:2.5) -- cycle;
\draw[line width=0.25 mm, dashed,fill=gray!30] (0,0) --  (0:4) arc(0:-30:4) -- cycle;
\draw[line width=0.25 mm, dashed,fill=white!100] (0,0) --  (0:2.5) arc(0:-30:2.5) -- cycle;
\draw[white,line width=0.35 mm] ([shift=(-30:4cm)]0,0) arc (-30:0:4cm);
\draw[white,line width=0.35 mm] (0,0)--(2.5,0);

\draw[black,line width=0.45 mm] (0,0)--(30:4);
\draw[black,line width=0.45 mm] (0,0)--(90:4);
\draw[black,line width=0.45 mm] (0,0)--(150:4);
\draw[black,line width=0.45 mm] (0,0)--(-30:4);
\draw[black,line width=0.45 mm] (0,0)--(-90:4);
\draw[black,line width=0.45 mm] (0,0)--(-150:4);

\draw[black,line width=0.45 mm] ([shift=(-180:2.5cm)]0,0) arc (-180:180:2.5cm);

\node at (165:1.8) {\footnotesize{$D_{\mathrm{sing}}^{L}$}};
\node at (-15:3.3) {\footnotesize{$D_{\mathrm{sing}}^{R}$}};

\draw[fill] (0:2.5) circle (0.1);
\draw[fill] (60:2.5) circle (0.1);
\draw[fill] (120:2.5) circle (0.1);
\draw[fill] (180:2.5) circle (0.1);
\draw[fill] (240:2.5) circle (0.1);
\draw[fill] (300:2.5) circle (0.1);

\draw[dashed] (-6.3,-3.8)--(-6.3,3.8);

\end{tikzpicture}
\end{center}
\begin{figuretext}\label{fig: D2 splitting}
The open sets $D_{\mathrm{reg}}^{R}, D_{\mathrm{reg}}^{L}, D_{\mathrm{sing}}^{R}$, and $D_{\mathrm{sing}}^{L}$.
\end{figuretext}
\end{figure}

We showed in \cite{CLscatteringsolitons} that zeros in $D_{\mathrm{sing}}$ give rise to solitons with singularities. We will therefore assume that the zero-set $\mathsf{Z}$ of $s_{11}(k)$ is contained in $D_{\mathrm{reg}} \cup (-1,0) \cup (1,\infty)$. 
We will consider the generic case of a finite number of simple zeros. In the case when the initial data $u_0$ and $u_1$ have compact support, the associated residue constants are then defined by
 \begin{align}\label{ck0compactdef}
& c_{k_{0}} := \begin{cases} 
- \frac{s_{13}(k_0)}{\dot{s}_{11}(k_0)}, & k_{0} \in \mathsf{Z}\setminus \mathbb{R}, \\
- \frac{s_{12}(k_0)}{\dot{s}_{11}(k_0)}, & k_{0} \in \mathsf{Z}\cap \mathbb{R}.
\end{cases}
\end{align}
If $u_0$ and $v_0$ do not have compact support, then the expressions in (\ref{ck0compactdef}) are in general not well-defined and the following more involved definition is required: Define the vector-valued function $w(x,k)$ by
\begin{align}\label{wdef}
w = \begin{pmatrix}
Y_{21}^AX_{32}^A - Y_{31}^AX_{22}^A  \\
Y_{31}^AX_{12}^A - Y_{11}^AX_{32}^A  \\
 Y_{11}^AX_{22}^A - Y_{21}^AX_{12}^A 
\end{pmatrix}.
\end{align}
If $k_{0} \in \mathsf{Z}\setminus \mathbb{R}$, then $c_{k_0}$ is defined as the unique complex constant such that
\begin{subequations}\label{ck0def}
\begin{align}\label{ck0def1}
\frac{w(x,k_{0})}{\dot{s}_{11}(k_0)} = c_{k_0} e^{(l_1(k_0)-l_3(k_0))x} [X(x,k_0)]_1 \qquad \text{for all $x \in \R$}.
\end{align}
If $k_{0} \in \mathsf{Z} \cap \mathbb{R}$, then $c_{k_0}$ is defined as the unique complex constant such that
\begin{align}\label{ck0def2}
\frac{[Y(x,k_0)]_2}{\dot{s}_{22}^A(k_0)} = c_{k_0} e^{(l_1(k_0)-l_2(k_0))x} [X(x,k_0)]_1\qquad \text{for all $x \in \R$}.
\end{align}
\end{subequations}

It turns out that simple zeros of $s_{11}(k)$ in $D_{\mathrm{reg}}^R$ and $D_{\mathrm{reg}}^L$ generate right- and left-moving breather solitons, respectively. 
Similarly, simple zeros of $s_{11}(k)$ in $(1,\infty)$ and $(-1,0)$ generate right- and left-moving solitons, respectively, and it was shown in \cite{CLscatteringsolitons} that a soliton generated by a zero in $(-1,0) \cup (1,\infty)$ is non-singular if and only if the associated residue constant $c_{k_0}$ satisfies $i(\omega^{2}k_{0}^{2} - \omega)c_{k_{0}} \notin (-\infty, 0)$.
This motivates the following assumptions on $\mathsf{Z}$ and $\{c_{k_0}\}_{k_0 \in \mathsf{Z}}$.

\subsection{Assumptions}\label{assumptionssubsec}
As in \cite[Theorem 2.11]{CLscatteringsolitons}, our results will be valid under the following assumptions:
\begin{enumerate}[$(i)$]
\item \label{solitonassumption}
Finite number of non-singular solitons: suppose that $s_{11}(k)$ has a simple zero at each point in $\mathsf{Z}$, where $\mathsf{Z} \subset D_{\mathrm{reg}} \cup (-1,0) \cup (1,\infty)$ is a set of finite cardinality, and that $s_{11}(k)$ has no other zeros in $k \in (\bar{D}_2 \cup \partial \D) \setminus \hat{\mathcal{Q}}$.
If $k_0 \in \mathsf{Z} \setminus \R$, then we also assume that $s_{22}^A(k_0) \neq 0$. If $k_0 \in \mathsf{Z} \cap \R$, then we also assume that $i(\omega^{2}k_{0}^{2} - \omega)c_{k_{0}} \notin (-\infty, 0)$.

\item \label{originassumption}
The spectral functions $s$ and $s^{A}$ have generic behavior near $k=1$ and $k=-1$: we suppose for $k_{\star} =1$ and $k_{\star}=-1$ that
\begin{align*}
& \lim_{k \to k_{\star}} (k-k_{\star}) s(k)_{11} \neq 0, & & \hspace{-0.1cm} \lim_{k \to k_{\star}} (k-k_{\star}) s(k)_{13} \neq 0, & & \hspace{-0.1cm} \lim_{k \to k_{\star}} s(k)_{31} \neq 0, & & \hspace{-0.1cm} \lim_{k \to k_{\star}} s(k)_{33} \neq 0, \\
& \lim_{k \to k_{\star}} (k-k_{\star}) s^{A}(k)_{11} \neq 0, & & \hspace{-0.1cm} \lim_{k \to k_{\star}} (k-k_{\star}) s^{A}(k)_{31} \neq 0, & & \hspace{-0.1cm} \lim_{k \to k_{\star}} s^{A}(k)_{13} \neq 0, & & \hspace{-0.1cm} \lim_{k \to k_{\star}} s^{A}(k)_{33} \neq 0.
\end{align*}

\item \label{nounstablemodesassumption}
No high-frequency modes: we suppose that $r_{1}(k)=0$ for all $k \in [0,i]$, where $[0,i]$ is the vertical segment from $0$ to $i$.
\end{enumerate}

\begin{figure}
\begin{center}
\begin{tikzpicture}[scale=0.9]
\node at (0,0) {};
\draw[black,line width=0.45 mm,->-=0.4,->-=0.85] (0,0)--(30:4);
\draw[black,line width=0.45 mm,->-=0.4,->-=0.85] (0,0)--(90:4);
\draw[black,line width=0.45 mm,->-=0.4,->-=0.85] (0,0)--(150:4);
\draw[black,line width=0.45 mm,->-=0.4,->-=0.85] (0,0)--(-30:4);
\draw[black,line width=0.45 mm,->-=0.4,->-=0.85] (0,0)--(-90:4);
\draw[black,line width=0.45 mm,->-=0.4,->-=0.85] (0,0)--(-150:4);

\draw[black,line width=0.45 mm] ([shift=(-180:2.5cm)]0,0) arc (-180:180:2.5cm);
\draw[black,arrows={-Triangle[length=0.2cm,width=0.18cm]}]
($(3:2.5)$) --  ++(90:0.001);
\draw[black,arrows={-Triangle[length=0.2cm,width=0.18cm]}]
($(57:2.5)$) --  ++(-30:0.001);
\draw[black,arrows={-Triangle[length=0.2cm,width=0.18cm]}]
($(123:2.5)$) --  ++(210:0.001);
\draw[black,arrows={-Triangle[length=0.2cm,width=0.18cm]}]
($(177:2.5)$) --  ++(90:0.001);
\draw[black,arrows={-Triangle[length=0.2cm,width=0.18cm]}]
($(243:2.5)$) --  ++(330:0.001);
\draw[black,arrows={-Triangle[length=0.2cm,width=0.18cm]}]
($(297:2.5)$) --  ++(210:0.001);

\draw[black,line width=0.15 mm] ([shift=(-30:0.55cm)]0,0) arc (-30:30:0.55cm);

\node at (0.8,0) {$\tiny \frac{\pi}{3}$};

\node at (-5:2.77) {\footnotesize $\Gamma_8$};
\node at (60:2.86) {\footnotesize $\Gamma_9$};
\node at (113:2.73) {\footnotesize $\Gamma_7$};
\node at (181:2.83) {\footnotesize $\Gamma_8$};
\node at (234:2.71) {\footnotesize $\Gamma_9$};
\node at (300:2.83) {\footnotesize $\Gamma_7$};

\node at (77:1.45) {\footnotesize $\Gamma_1$};
\node at (160:1.45) {\footnotesize $\Gamma_2$};
\node at (-163:1.45) {\footnotesize $\Gamma_3$};
\node at (-77:1.45) {\footnotesize $\Gamma_4$};
\node at (-42:1.45) {\footnotesize $\Gamma_5$};
\node at (43:1.45) {\footnotesize $\Gamma_6$};

\node at (84:3.3) {\footnotesize $\Gamma_4$};
\node at (155:3.3) {\footnotesize $\Gamma_5$};
\node at (-156:3.3) {\footnotesize $\Gamma_6$};
\node at (-84:3.3) {\footnotesize $\Gamma_1$};
\node at (-35:3.3) {\footnotesize $\Gamma_2$};
\node at (35:3.3) {\footnotesize $\Gamma_3$};

\end{tikzpicture}
\hspace{0.5cm}
\begin{tikzpicture}[scale=0.9]
\node at (0,0) {};
\draw[black,line width=0.45 mm] (0,0)--(30:4);
\draw[black,line width=0.45 mm] (0,0)--(90:4);
\draw[black,line width=0.45 mm] (0,0)--(150:4);
\draw[black,line width=0.45 mm] (0,0)--(-30:4);
\draw[black,line width=0.45 mm] (0,0)--(-90:4);
\draw[black,line width=0.45 mm] (0,0)--(-150:4);

\draw[black,line width=0.45 mm] ([shift=(-180:2.5cm)]0,0) arc (-180:180:2.5cm);
\draw[black,line width=0.15 mm] ([shift=(-30:0.55cm)]0,0) arc (-30:30:0.55cm);

\node at (120:1.6) {\footnotesize{$D_{1}$}};
\node at (-60:3.7) {\footnotesize{$D_{1}$}};

\node at (180:1.6) {\footnotesize{$D_{2}$}};
\node at (0:3.7) {\footnotesize{$D_{2}$}};

\node at (240:1.6) {\footnotesize{$D_{3}$}};
\node at (60:3.7) {\footnotesize{$D_{3}$}};

\node at (-60:1.6) {\footnotesize{$D_{4}$}};
\node at (120:3.7) {\footnotesize{$D_{4}$}};

\node at (0:1.6) {\footnotesize{$D_{5}$}};
\node at (180:3.7) {\footnotesize{$D_{5}$}};

\node at (60:1.6) {\footnotesize{$D_{6}$}};
\node at (-120:3.7) {\footnotesize{$D_{6}$}};

\node at (0.8,0) {$\tiny \frac{\pi}{3}$};

\draw[fill] (0:2.5) circle (0.1);
\draw[fill] (60:2.5) circle (0.1);
\draw[fill] (120:2.5) circle (0.1);
\draw[fill] (180:2.5) circle (0.1);
\draw[fill] (240:2.5) circle (0.1);
\draw[fill] (300:2.5) circle (0.1);

\node at (0:2.9) {\footnotesize{$\kappa_1$}};
\node at (60:2.85) {\footnotesize{$\kappa_2$}};
\node at (120:2.85) {\footnotesize{$\kappa_3$}};
\node at (180:2.9) {\footnotesize{$\kappa_4$}};
\node at (240:2.85) {\footnotesize{$\kappa_5$}};
\node at (300:2.85) {\footnotesize{$\kappa_6$}};

\draw[dashed] (-4.3,-3.8)--(-4.3,3.8);

\end{tikzpicture}
\end{center}
\begin{figuretext}\label{fig: Dn}
The contour $\Gamma = \cup_{j=1}^9 \Gamma_j$ in the complex $k$-plane (left) and the open sets $D_{n}$, $n=1,\ldots,6$, together with the sixth roots of unity $\kappa_j$, $j = 1, \dots, 6$ (right).
\end{figuretext}
\end{figure}

We emphasize that Assumption $(\ref{originassumption})$ is generic. Assumption $(\ref{nounstablemodesassumption})$ ensures that the solution belongs to the physically relevant class of solutions which do not contain exponentially growing high-frequency modes. If Assumptions $(\ref{solitonassumption})$--$(\ref{nounstablemodesassumption})$ hold true, then the solution $u(x,t)$ of the initial value problem \eqref{boussinesq}--\eqref{initial data} exists globally \cite{CLscatteringsolitons}.

\subsection{Statement of main result}
The formulation of our main theorem involves a number of quantities that we now introduce. Let $\zeta := x/t$ and assume that $\zeta \in (\frac{1}{\sqrt{3}}, 1)$. For $k \in \mathbb{C}\setminus\{0\}$ and $1 \leq j<i \leq 3$, define $\Phi_{ij}(\zeta, k)$ by
\begin{align}\label{def of Phi ij}
\Phi_{ij}(\zeta,k) = (l_{i}(k)-l_{j}(k))\zeta + (z_{i}(k)-z_{j}(k)),
\end{align}
where $l_j(k)$ and $z_j(k)$ are as in (\ref{lmexpressions}).
The function $k\mapsto \Phi_{21}(\zeta,k)$ has four saddle points $\{k_{j}=k_{j}(\zeta)\}_{j=1}^{4}$ given by
\begin{subequations}\label{def of kj}
\begin{align}
& k_{2} = \frac{1}{4}\bigg( \zeta - \sqrt{8+\zeta^{2}} - i \sqrt{2}\sqrt{4-\zeta^{2}+\zeta\sqrt{8+\zeta^{2}}} \bigg), \label{def of k2} \\
& k_{4} = \frac{1}{4}\bigg( \zeta + \sqrt{8+\zeta^{2}} - i\sqrt{2}\sqrt{4-\zeta^{2}-\zeta\sqrt{8+\zeta^{2}}} \bigg), \label{def of k4}
\end{align}
\end{subequations} 
$k_{1} = \bar{k}_2$, and $k_{3}=\bar{k}_4$. Note that $|k_{2}|=|k_{4}|=1$, $\arg k_{2} \in (-\frac{3\pi}{4},-\frac{2\pi}{3})$, and $\arg k_{4} \in (-\frac{\pi}{6},0)$. Define also $\delta_j(\zeta, k)$, $j = 1, \dots, 5$, by
\begin{align}\nonumber
& \delta_{1}(\zeta, k) = \exp \bigg\{ \frac{-1}{2\pi i} \int_{i}^{\omega k_{4}} \frac{\ln(1 + r_1(s)r_{2}(s))}{s - k} ds \bigg\},  \; \delta_{4}(\zeta, k) = \exp \bigg\{ \frac{1}{2\pi i} \int_{\omega^{2}k_{2}}^{\omega} \frac{\ln f(s)}{s - k} ds \bigg\},
	 \\\nonumber
& \delta_{2}(\zeta, k) = \exp \bigg\{ \frac{1}{2\pi i} \int_{\omega k_{4}}^{\omega^{2}k_{2}} \frac{\ln(1 + r_1(s)r_{2}(s))}{s - k} ds \bigg\},
\; \delta_{5}(\zeta, k) = \exp \bigg\{ \frac{1}{2\pi i} \int_{\omega^{2}k_{2}}^{\omega} \frac{\ln f(\omega^{2}s)}{s - k} ds \bigg\},
  	\\ \label{IVdeltadef}
&  \delta_{3}(\zeta, k) = \exp \bigg\{ \frac{1}{2\pi i} \int_{\omega k_{4}}^{\omega^{2}k_{2}} \frac{\ln f(s)}{s - k} ds \bigg\},
\end{align}
where the paths follow the unit circle in the counterclockwise direction, the principal branch is used for the logarithms, and 
\begin{align}\label{def of f}
f(k) := 1+r_{1}(k)r_{2}(k) + r_{1}(\tfrac{1}{\omega^{2}k})r_{2}(\tfrac{1}{\omega^{2}k}), \qquad k \in \partial \mathbb{D}.
\end{align}
By \cite[Theorem 2.3 and Lemma 2.13]{CLmain}, $r_{1},r_{2},f$ are well-defined on $\partial \D \setminus \mathcal{Q}$ and the arguments of all logarithms appearing in the above definitions of $\{\delta_j(\zeta, k)\}_{j=1}^{5}$ are $>0$. 
Define 
\begin{align}\nonumber
& \mathcal{P}(k) = \prod_{k_{0}\in \mathsf{Z} \cap D_{\mathrm{reg}}^R} \hspace{-0.1cm} \frac{(k-k_{0})(k-k_{0}^{-1})(k-\omega^{2} \bar{k}_{0})(k-\omega \bar{k}_{0}^{-1})}{(k-\bar{k}_{0})(k-\bar{k}_{0}^{-1})(k-\omega k_{0})(k-\omega^{2} k_{0}^{-1})} \cdot 
\prod_{k_{0}\in \mathsf{Z}\cap (1,\infty)} \hspace{-0.06cm}  \frac{(k-\omega^{2}k_{0})(k-\omega k_{0}^{-1})}{(k-\omega k_{0})(k-\omega^{2}k_{0}^{-1})},
	\\\nonumber
& \mathcal{D}_{1}(k) = \frac{\delta_{1}(\omega k)\delta_{1}(\frac{1}{\omega^{2}k})^{2}}{\delta_{1}(\omega^{2}k)^{2}\delta_{1}(\frac{1}{\omega k})\delta_{1}(\frac{1}{k})} \frac{\delta_{2}(\omega^{2}k)\delta_{2}(\frac{1}{k})^{2}}{\delta_{2}(\omega k)^{2} \delta_{2}(\frac{1}{\omega^{2}k})\delta_{2}(\frac{1}{\omega k})} \frac{\delta_{3}(\omega k) \delta_{3}(\omega^{2} k) \delta_{3}(\frac{1}{\omega k})^{2}}{\delta_{3}(\frac{1}{k})\delta_{3}(\frac{1}{\omega^{2} k})}  
	\\ \nonumber
& \hspace{1.5cm} \times \frac{\delta_{4}(\omega^{2}k)^{2} \delta_{4}(\frac{1}{k})\delta_{4}(\frac{1}{\omega k})}{\delta_{4}(k)\delta_{4}(\omega k)\delta_{4}(\frac{1}{\omega^{2} k})^{2}} \frac{\delta_{5}(\omega k)^{2} \delta_{5}(\frac{1}{\omega k})\delta_{5}(\frac{1}{\omega^{2}k})}{\delta_{5}(k) \delta_{5}(\frac{1}{k})^{2}\delta_{5}(\omega^{2}k)}, 
	\\ \nonumber
& \mathcal{D}_{2}(k) = \frac{\delta_{1}(\omega k)^{2}\delta_{1}(\frac{1}{\omega^{2}k})\delta_{1}(\frac{1}{k})}{\delta_{1}(\omega^{2}k)\delta_{1}(\frac{1}{\omega k})^{2}\delta_{1}(k)} \frac{\delta_{2}(\frac{1}{k})\delta_{2}(\frac{1}{\omega k})}{\delta_{2}(\omega k) \delta_{2}(\frac{1}{\omega^{2}k})^{2}\delta_{2}(\omega^{2} k)} \frac{\delta_{3}(\omega^{2} k)^{2}  \delta_{3}(\frac{1}{\omega k})\delta_{3}(\frac{1}{\omega^{2} k})}{\delta_{3}(\frac{1}{k})^{2}\delta_{3}(\omega k)}  
	\\ \label{mathcalPdef}
& \hspace{1.5cm} \times \frac{\delta_{4}(\omega^{2}k) \delta_{4}(\frac{1}{\omega k})^{2}}{\delta_{4}(\omega k)^{2}\delta_{4}(\frac{1}{\omega^{2} k})\delta_{4}(\frac{1}{k})} \frac{\delta_{5}(\omega k) \delta_{5}(\frac{1}{\omega^{2}k})^{2}\delta_{5}(\omega^{2} k)}{\delta_{5}(\frac{1}{k}) \delta_{5}(\frac{1}{\omega k})}.
\end{align}
Define $\chi_{j}(\zeta,k)$, $j = 1, 2, 3$, and $\tilde{\chi}_{j}(\zeta,k)$, $j = 2,3,4,5$, by
\begin{align}
\chi_{1}(\zeta,k) =&\; - \frac{1}{2\pi i} \int_{i}^{\omega k_{4}}  \ln_{s}(k-s) d\ln(1+r_1(s)r_{2}(s)), \nonumber \\
\chi_{2}(\zeta,k) =&\; \frac{1}{2\pi i} \int_{\omega k_{4}}^{\omega^{2}k_{2}}  \ln_{s}(k-s) d\ln(1+r_1(s)r_{2}(s)), \nonumber \\
\chi_{3}(\zeta,k) =&\; \frac{1}{2\pi i} \int_{\omega k_{4}}^{\omega^{2}k_{2}}  \ln_{s}(k-s) d\ln(f(s)), \nonumber \\
\tilde{\chi}_{2}(\zeta,k) =&\; \frac{1}{2\pi i} \int_{\omega k_{4}}^{\omega^{2}k_{2}}  \tilde{\ln}_{s}(k-s) d\ln(1+r_1(s)r_{2}(s)), \nonumber \\
\tilde{\chi}_{3}(\zeta,k) =&\; \frac{1}{2\pi i} \int_{\omega k_{4}}^{\omega^{2}k_{2}}  \tilde{\ln}_{s}(k-s) d\ln(f(s)), \nonumber \\
\tilde{\chi}_{4}(\zeta,k) =&\;\frac{1}{2\pi i} \dashint_{\omega^{2}k_{2}}^{\omega}  \tilde{\ln}_{s}(k-s) d\ln(f(s)) \nonumber \\
:= &\; \frac{1}{2\pi i}\lim_{\epsilon \to 0_{+}} \bigg(  \int_{\omega^{2}k_{2}}^{e^{i(\frac{2\pi}{3}-\epsilon)}} \tilde{\ln}_{s}(k-s) d\ln(f(s)) - \tilde{\ln}_{\omega}(k-\omega)\ln(f(e^{i(\frac{2\pi}{3}-\epsilon)})) \bigg), \nonumber \\
\tilde{\chi}_{5}(\zeta,k) = &\; \frac{1}{2\pi i} \dashint_{\omega^{2}k_{2}}^{\omega} \tilde{\ln}_{s}(k-s) d\ln(f(\omega^{2}s)) \nonumber \\
:= &\; \frac{1}{2\pi i}\lim_{\epsilon \to 0_{+}} \bigg(  \int_{\omega^{2}k_{2}}^{e^{i(\frac{2\pi}{3}-\epsilon)}} \tilde{\ln}_{s}(k-s) d\ln(f(\omega^{2}s)) - \tilde{\ln}_{\omega}(k-\omega)\ln(f(e^{-i\epsilon})) \bigg), \label{def of chij}
\end{align}
where the paths follow $\partial \mathbb{D}$ in the counterclockwise direction. For $s \in \{e^{i\theta}\, | \, \theta \in [\frac{\pi}{2},\frac{2\pi}{3}]\}$, $k \mapsto \ln_{s}(k-s)=\ln |k-s|+i \arg_{s}(k-s)$ has a cut along $\{e^{i \theta} \, | \,  \theta \in [\frac{\pi}{2},\arg s]\}\cup(i,i\infty)$ and the branch is such that $\arg_{s}(1)=2\pi$. Also, for $s \in \{e^{i\theta} \,|\, \theta \in [\frac{\pi}{2},\frac{2\pi}{3}]\}$, $k \mapsto \tilde{\ln}_{s}(k-s):=\ln |k-s|+i \tilde{\arg}_{s}(k-s)$ has a cut along $\{e^{i \theta} \,|\, \theta \in [\arg s,\pi]\}\cup(-\infty,0)$, and satisfies $\tilde{\arg}_{s}(1)=0$. Regularized integrals are needed in the definitions of $\tilde{\chi}_{4},\tilde{\chi}_{5}$  because $f(1)=f(\omega)=0$ (see \cite[Lemma 2.13 $(i)$]{CLmain}). Define $\{\nu_{j}\}_{j=1}^{4}$, $\hat{\nu}_1$, and $\hat{\nu}_2$ by
\begin{align*}\nonumber
& \nu_1(k) = - \frac{1}{2\pi}\ln(1+r_{1}(\omega k)r_{2}(\omega k)), 
\qquad \nu_2(k) = - \frac{1}{2\pi}\ln(1+r_{1}(\omega^{2} k)r_{2}(\omega^{2} k)),  \\ 
& \nu_3(k) = - \frac{1}{2\pi}\ln(f(\omega k)),  \qquad
\nu_4(k) = - \frac{1}{2\pi}\ln(f(\omega^{2} k)),  \\
& \hat{\nu}_1(k) := \nu_3(k) - \nu_1(k), \qquad \hat{\nu}_2(k) = \nu_2(k) + \nu_3(k) - \nu_4(k).
\end{align*}
The following functions $d_{1,0}$ and $d_{2,0}$ appear in our final result:
\begin{align}
& d_{1,0}(\zeta,t) := e^{-\chi_{1}(\zeta,\omega k_{4}) - \tilde{\chi}_{2}(\zeta,\omega k_{4}) + 2\tilde{\chi}_{3}(\zeta,\omega k_{4}) } \nonumber 	\\ 
& \hspace{1.6cm} \times e^{i(\nu_{2}(k_{2})-2\nu_{4}(k_{2}))\tilde{\ln}_{\omega^{2}k_{2}}(\omega k_{4}-\omega^{2}k_{2})} t^{-i \hat{\nu}_{1}(k_{4})} z_{1,\star}^{-2i\hat{\nu}_{1}(k_{4})} \mathcal{D}_{1}(\omega k_{4}),  \label{def of d10} \\
& d_{2,0}(\zeta,t) := e^{-2\chi_{2}(\zeta,\omega^{2} k_{2}) + \chi_{3}(\zeta,\omega^{2} k_{2}) - \tilde{\chi}_{4}(\zeta,\omega^{2} k_{2}) + 2\tilde{\chi}_{5}(\zeta,\omega^{2} k_{2}) } \nonumber 
	\\
& \hspace{1.6cm} \times e^{i(\nu_{3}(k_{4})-2\nu_{1}(k_{4}))\ln_{\omega k_{4}}(\omega^{2} k_{2}-\omega k_{4})} t^{-i \hat{\nu}_{2}(k_{2})} z_{2,\star}^{-2i\hat{\nu}_{2}(k_{2})} \mathcal{D}_{2}(\omega^{2} k_{2}), \label{def of d20}
\end{align}
where $z_{1,\star}$, $z_{2,\star}$ are given by
\begin{align*}
& z_{1,\star} = z_{1,\star}(\zeta) := \sqrt{2}e^{\frac{\pi i}{4}} \sqrt{\omega \frac{4-3k_{4} \zeta - k_{4}^{3} \zeta}{4k_{4}^{4}}}, \qquad -i\omega k_{4}z_{1,\star}>0, \\
& z_{2,\star} = z_{2,\star}(\zeta) := \sqrt{2}e^{\frac{\pi i}{4}} \sqrt{-\omega^{2} \frac{4-3k_{2} \zeta - k_{2}^{3} \zeta}{4k_{2}^{4}}}, \qquad -i\omega^{2} k_{2}z_{2,\star}>0,
\end{align*}
and the branches for the complex powers $z_{j,\star}^{\alpha} = e^{\alpha \ln z_{j,\star}}$ are fixed by
\begin{align*}
& \ln z_{1,\star} = \ln |z_{1,\star}| + i \arg z_{1,\star} = \ln |z_{1,\star}| + i \big( \tfrac{\pi}{2}-\arg (\omega k_{4}) \big), 
	\\
& \ln z_{2,\star} = \ln |z_{2,\star}| + i \arg z_{2,\star} = \ln |z_{2,\star}| + i \big( \tfrac{\pi}{2}-\arg (\omega^{2} k_2) \big),
\end{align*}
with $ \arg (\omega k_{4}) \in (\tfrac{\pi}{2},\tfrac{2\pi}{3})$ and $\arg (\omega^{2} k_2) \in (\tfrac{\pi}{2},\tfrac{2\pi}{3})$. Define also
\begin{align*}
& q_{3} = |\tilde{r}(\tfrac{1}{k_{4}})|^{\frac{1}{2}}r_{1}(\tfrac{1}{k_{4}}), & & q_{2} = \tilde{r}(\omega^{2}k_{2})^{\frac{1}{2}}r_{1}(\omega^{2}k_{2}), & & q_{5} = |\tilde{r}(\omega k_{2})|^{\frac{1}{2}}r_{1}(\omega k_{2}), & & q_{6} = |\tilde{r}(\tfrac{1}{k_{2}})|^{\frac{1}{2}}r_{1}(\tfrac{1}{k_{2}}),
\end{align*}
where
\begin{align}\label{def of tilde r}
\tilde{r}(k):=\frac{\omega^{2}-k^{2}}{1-\omega^{2}k^{2}}, \qquad k \in \mathbb{C}\setminus \{\omega^{2},-\omega^{2}\}.
\end{align}

We now state our main result. The Gamma function $\Gamma(k)$ appears in the statement, as well as the square roots of $\hat{\nu}_{1}(k_{4})$ and $\hat{\nu}_{2}(k_{2})$. These square roots are well-defined and $\geq 0$ thanks to the inequalities $\hat{\nu}_{1}(k_{4})\geq 0$ and $\hat{\nu}_{2}(k_{2}) \geq 0$, which follow from \cite[Lemma 2.13]{CLmain} and the fact that $\arg k_{2} \in (-\frac{3\pi}{4},-\frac{2\pi}{3})$ and $\arg k_{4} \in (-\frac{\pi}{6},0)$.

\begin{theorem}[Asymptotics in Sector IV]\label{asymptoticsth}
Let $u_0,u_1 \in \mathcal{S}(\R)$ be real-valued and such that $\int_{\mathbb{R}}u_{1} dx =0$. Let $v_{0}(x) = \int_{-\infty}^{x}u_{1}(x')dx'$ and suppose Assumptions $(\ref{solitonassumption})$--$(\ref{nounstablemodesassumption})$ are fulfilled. Let $\mathcal{I}$ be a fixed compact subset of $(\smash{\frac{1}{\sqrt{3}}}, 1)$. Then the global solution $u(x,t)$ of the initial value problem for (\ref{boussinesq}) with initial data $u_0, u_1$ enjoys the following asymptotics as $t \to \infty$:
\begin{align}\label{asymp for u sector V}
& u(x,t) =
 \frac{A_{1}(\zeta)}{\sqrt{t}}   \cos \alpha_{1}(\zeta,t) +\frac{A_{2}(\zeta)}{\sqrt{t}} \cos \alpha_{2}(\zeta,t) + O\bigg(\frac{\ln t}{t}\bigg), 
\end{align}
uniformly for $\zeta:= \frac{x}{t} \in \mathcal{I}$, where 
\begin{align}\nonumber
& A_{1}(\zeta) := \frac{4\sqrt{3}\sqrt{\hat{\nu}_1(k_4)} \im k_4}{-i\omega k_{4} z_{1,\star}|\tilde{r}(\frac{1}{k_{4}})|^{\frac{1}{2}}}\sin(\arg(\omega k_{4})),    
	\\ \nonumber
& A_{2}(\zeta) := \frac{-4\sqrt{3}\sqrt{\hat{\nu}_2(k_2)}|\tilde{r}(\frac{1}{k_{2}})|^{\frac{1}{2}} \im k_2}{-i\omega^{2} k_{2} z_{2,\star}}\sin(\arg(\omega^{2} k_{2})), 
	\\ \nonumber 
& \alpha_{1}(\zeta,t) := \frac{3\pi}{4}+\arg q_{3}+\arg\Gamma(i\hat{\nu}_1(k_4))+\arg d_{1,0} + \arg \tfrac{\mathcal{P}(\omega k_{4})}{\mathcal{P}(\omega^{2} k_{4})} -t\, \im \Phi_{31}(\zeta,\omega k_{4}),
	\\ \nonumber
& \alpha_{2}(\zeta,t) := \frac{3\pi}{4}-\arg (q_{6}-q_{2}q_{5})+\arg\Gamma(i\hat{\nu}_2(k_2))+\arg d_{2,0} + \arg \tfrac{\mathcal{P}(\omega^{2} k_{2})}{\mathcal{P}(\omega k_{2})} -t \, \im \Phi_{32}(\zeta,\omega^{2} k_{2}).
\end{align}
\end{theorem}

The effect of the solitons on the asymptotic behavior (\ref{asymp for u sector V}) of $u(x,t)$ is encoded in the phase shifts $\arg \tfrac{\mathcal{P}(\omega k_{4})}{\mathcal{P}(\omega^{2} k_{4})}$ and $\arg \tfrac{\mathcal{P}(\omega^{2} k_{2})}{\mathcal{P}(\omega k_{2})}$ that appear in the definitions of $\alpha_1$ and $\alpha_2$, respectively. If there are no solitons (i.e., if the set $\mathsf{Z}$ is empty), then these phase shifts are absent, and the formula (\ref{asymp for u sector V}) reduces to the formula for Sector IV given in \cite{CLmain}. Indeed, the only place where the set $\mathsf{Z}$ appears in the above definitions is in the definition (\ref{mathcalPdef}) of $\mathcal{P}(k)$, and if $\mathsf{Z}$ is empty, then $\mathcal{P}(k)$ is identically equal to $1$. Thus, these phase shifts describe the leading order soliton-radiation interaction in Sector IV.
We observe that the products in the definition of $\mathcal{P}$ only run over the zeros in $\mathsf{Z}$ with positive real part; this corresponds to the fact that only right-moving solitons generate a phase shift in Sector IV.

\section{The RH problem for $n$}\label{RHnsec}
We will prove Theorem \ref{asymptoticsth} by analyzing a row-vector RH problem, whose solution is denoted by $n$. In what follows, we recall this RH problem from \cite{CLscatteringsolitons}. Let $\theta_{ij}(x,t,k) = t \, \Phi_{ij}(\zeta,k)$. For $j = 1, \dots, 6$, we write $\Gamma_j = \Gamma_{j'} \cup \Gamma_{j''}$, where $\Gamma_{j'} = \Gamma_j \setminus \D$ and $\Gamma_{j''} := \Gamma_j \setminus \Gamma_{j'}$ with $\Gamma_j$ as in Figure \ref{fig: Dn}. The jump matrix $v(x,t,k)$ is defined for $k \in \Gamma$ by
\begin{align}
& v_{1'} = \begin{pmatrix}
1 & -r_{1}(k)e^{-\theta_{21}} & 0 \\
0 & 1 & 0 \\
0 & 0 & 1
\end{pmatrix}, \; v_{1''} = \begin{pmatrix}
1 & 0 & 0 \\
r_{1}(\frac{1}{k})e^{\theta_{21}} & 1 & 0 \\
0 & 0 & 1
\end{pmatrix}, \; v_{2'} = \begin{pmatrix}
1 & 0 & 0 \\
0 & 1 & -r_{2}(\frac{1}{\omega k})e^{-\theta_{32}} \\
0 & 0 & 1
\end{pmatrix}, \nonumber \\
& v_{2''} = \begin{pmatrix}
1 & 0 & 0 \\
0 & 1 & 0 \\
0 & r_{2}(\omega k)e^{\theta_{32}} & 1
\end{pmatrix}, \;  v_{3'} = \begin{pmatrix}
1 & 0 & 0 \\
0 & 1 & 0 \\
-r_{1}(\omega^{2}k)e^{\theta_{31}} & 0 & 1
\end{pmatrix}, \; v_{3''} = \begin{pmatrix}
1 & 0 & r_{1}(\frac{1}{\omega^{2}k})e^{-\theta_{31}} \\
0 & 1 & 0 \\
0 & 0 & 1
\end{pmatrix}, \nonumber \\
& v_{4'} = \begin{pmatrix}
1 & -r_{2}(\frac{1}{k})e^{-\theta_{21}} & 0 \\
0 & 1 & 0 \\
0 & 0 & 1
\end{pmatrix}, \; v_{4''} = \begin{pmatrix}
1 & 0 & 0 \\
r_{2}(k)e^{\theta_{21}} & 1 & 0 \\
0 & 0 & 1
\end{pmatrix}, \; v_{5'} = \begin{pmatrix}
1 & 0 & 0 \\
0 & 1 & -r_{1}(\omega k)e^{-\theta_{32}} \\
0 & 0 & 1
\end{pmatrix}, \nonumber \\
& v_{5''} = \begin{pmatrix}
1 & 0 & 0 \\
0 & 1 & 0 \\
0 & r_{1}(\frac{1}{\omega k})e^{\theta_{32}} & 1
\end{pmatrix}, \;  v_{6'} = \begin{pmatrix}
1 & 0 & 0 \\
0 & 1 & 0 \\
-r_{2}(\frac{1}{\omega^{2} k})e^{\theta_{31}} & 0 & 1
\end{pmatrix}, \; v_{6''} = \begin{pmatrix}
1 & 0 & r_{2}(\omega^{2}k)e^{-\theta_{31}} \\
0 & 1 & 0 \\
0 & 0 & 1
\end{pmatrix}, \nonumber \\
& v_{7} = \begin{pmatrix}
1 & -r_{1}(k)e^{-\theta_{21}} & r_{2}(\omega^{2}k)e^{-\theta_{31}} \\
-r_{2}(k)e^{\theta_{21}} & 1+r_{1}(k)r_{2}(k) & \big(r_{2}(\frac{1}{\omega k})-r_{2}(k)r_{2}(\omega^{2}k)\big)e^{-\theta_{32}} \\
r_{1}(\omega^{2}k)e^{\theta_{31}} & \big(r_{1}(\frac{1}{\omega k})-r_{1}(k)r_{1}(\omega^{2}k)\big)e^{\theta_{32}} & f(\omega^{2}k)
\end{pmatrix}, \nonumber \\
& v_{8} = \begin{pmatrix}
f(k) & r_{1}(k)e^{-\theta_{21}} & \big(r_{1}(\frac{1}{\omega^{2} k})-r_{1}(k)r_{1}(\omega k)\big)e^{-\theta_{31}} \\
r_{2}(k)e^{\theta_{21}} & 1 & -r_{1}(\omega k) e^{-\theta_{32}} \\
\big( r_{2}(\frac{1}{\omega^{2}k})-r_{2}(\omega k)r_{2}(k) \big)e^{\theta_{31}} & -r_{2}(\omega k) e^{\theta_{32}} & 1+r_{1}(\omega k)r_{2}(\omega k)
\end{pmatrix}, \nonumber \\
& v_{9} = \begin{pmatrix}
1+r_{1}(\omega^{2}k)r_{2}(\omega^{2}k) & \big( r_{2}(\frac{1}{k})-r_{2}(\omega k)r_{2}(\omega^{2} k) \big)e^{-\theta_{21}} & -r_{2}(\omega^{2}k)e^{-\theta_{31}} \\
\big(r_{1}(\frac{1}{k})-r_{1}(\omega k) r_{1}(\omega^{2} k)\big)e^{\theta_{21}} & f(\omega k) & r_{1}(\omega k)e^{-\theta_{32}} \\
-r_{1}(\omega^{2}k)e^{\theta_{31}} & r_{2}(\omega k) e^{\theta_{32}} & 1
\end{pmatrix}, \label{vdef}
\end{align}
where $v_j, v_{j'}, v_{j''}$ are the restrictions of $v$ to $\Gamma_{j}$, $\Gamma_{j'}$, and $\Gamma_{j''}$, respectively. Let $\Gamma_{\star} = \{i\kappa_j\}_{j=1}^6 \cup \{0\}$ be the set of intersection points of $\Gamma$. 
Let $\mathsf{Z}$ be as in Assumption $(\ref{solitonassumption})$ and let
\begin{align*}
\hat{\mathsf{Z}} := \mathsf{Z} \cup \mathsf{Z}^{*} \cup \omega \mathsf{Z} \cup \omega \mathsf{Z}^{*} \cup \omega^{2} \mathsf{Z} \cup \omega^{2} \mathsf{Z}^{*} \cup \mathsf{Z}^{-1} \cup \mathsf{Z}^{-*} \cup \omega \mathsf{Z}^{-1} \cup \omega \mathsf{Z}^{-*} \cup \omega^{2} \mathsf{Z}^{-1} \cup \omega^{2} \mathsf{Z}^{-*},
\end{align*}
with $\mathsf{Z}^{-1}=\{k_{0}^{-1} \,|\, k_{0}\in \mathsf{Z}\}$, $\mathsf{Z}^{*}=\{\bar{k}_{0} \,|\, k_{0}\in \mathsf{Z}\}$, and $\mathsf{Z}^{-*}=\{\bar{k}_{0}^{-1} \,|\, k_{0}\in \mathsf{Z}\}$.

\begin{RHproblem}[RH problem for $n$]\label{RHn}
Find a $1 \times 3$-row-vector valued function $n(x,t,k)$ with the following properties:
\begin{enumerate}[$(a)$]
\item\label{RHnitema} $n(x,t,\cdot) : \C \setminus (\Gamma\cup \hat{\mathsf{Z}}) \to \mathbb{C}^{1 \times 3}$ is analytic.

\item\label{RHnitemb} The limits of $n(x,t,k)$ as $k$ approaches $\Gamma \setminus \Gamma_\star$ from the left and right exist, are continuous on $\Gamma \setminus \Gamma_\star$, and are denoted by $n_+$ and $n_-$, respectively. Furthermore, they are related by
\begin{align}\label{njump}
  n_+(x,t,k) = n_-(x, t, k) v(x, t, k) \qquad \text{for} \quad k \in \Gamma \setminus \Gamma_\star.
\end{align}

\item\label{RHnitemc} $n(x,t,k) = O(1)$ as $k \to k_{\star} \in \Gamma_\star$.

\item\label{RHnitemd} For $k \in \C \setminus ( \Gamma \cup \hat{\mathsf{Z}})$, $n$ obeys the symmetries
\begin{align}\label{nsymm}
n(x,t,k) = n(x,t,\omega k)\mathcal{A}^{-1} = n(x,t,k^{-1}) \mathcal{B},
\end{align}
where $\mathcal{A}$ and $\mathcal{B}$ are the matrices defined by 
\begin{align}\label{def of Acal and Bcal}
\mathcal{A} := \begin{pmatrix}
0 & 0 & 1 \\
1 & 0 & 0 \\
0 & 1 & 0
\end{pmatrix} \qquad \mbox{ and } \qquad \mathcal{B} := \begin{pmatrix}
0 & 1 & 0 \\
1 & 0 & 0 \\
0 & 0 & 1
\end{pmatrix}.
\end{align}

\item\label{RHniteme} $n(x,t,k) = (1,1,1) + O(k^{-1})$ as $k \to \infty$.

\item\label{RHnitemf} 
At each point of $\hat{\mathsf{Z}}$, one entry of $n$ has (at most) a simple pole while two entries are analytic. The following residue conditions hold at the points in $\hat{\mathsf{Z}}$: for each $k_{0}\in \mathsf{Z}\setminus \mathbb{R}$, 
\begin{align}\nonumber
& \underset{k = k_0}{\res} n_3(k) = c_{k_{0}}e^{-\theta_{31}(k_0)} n_1(k_0), & & \underset{k = \omega k_0}{\res} n_2(k) = \omega c_{k_{0}}e^{-\theta_{31}(k_0)} n_3(\omega k_0), 
	\\ \nonumber
& \underset{k = \omega^2 k_0}{\res} n_1(k) = \omega^2 c_{k_{0}}e^{-\theta_{31}(k_0)} n_2(\omega^2 k_0), & & \underset{k = k_0^{-1}}{\res} n_3(k) = - k_0^{-2} c_{k_{0}}e^{-\theta_{31}(k_0)} n_2(k_0^{-1}),	
	\\\nonumber
&  \underset{k = \omega^2k_0^{-1}}{\res} n_1(k) = -\tfrac{\omega^2}{k_0^{2}} c_{k_{0}}e^{-\theta_{31}(k_0)} n_3(\omega^2 k_0^{-1}), & & \underset{k = \omega k_0^{-1}}{\res} n_2(k) 
= -\tfrac{\omega}{k_0^{2}} c_{k_{0}}e^{-\theta_{31}(k_0)} n_1(\omega k_0^{-1}), 
	 \\ \nonumber
& \underset{k = \bar{k}_0}{\res} n_2(k) = d_{k_0} e^{\theta_{32}(\bar{k}_0)} n_3(\bar{k}_0), & & \underset{k = \omega \bar{k}_0}{\res} n_1(k) = \omega d_{k_0} e^{\theta_{32}(\bar{k}_0)} n_2( \omega \bar{k}_0),
	\\ \nonumber
& \underset{k = \omega^2 \bar{k}_0}{\res} n_3(k) = \omega^2 d_{k_0} e^{\theta_{32}(\bar{k}_0)} n_1(\omega^2 \bar{k}_0), & & \underset{k = \bar{k}_0^{-1}}{\res} n_1(k) = - \bar{k}_0^{-2} d_{k_0} e^{\theta_{32}(\bar{k}_0)} n_3(\bar{k}_0^{-1}),	
	\\ \label{nresiduesk0}
&  \underset{k = \omega^2 \bar{k}_0^{-1}}{\res} n_2(k) = -\tfrac{\omega^2}{\bar{k}_0^{2}} d_{k_0} e^{\theta_{32}(\bar{k}_0)} n_1( \omega^2 \bar{k}_0^{-1}), & & \underset{k = \omega \bar{k}_0^{-1}}{\res} n_3(k) = -\tfrac{\omega}{\bar{k}_0^{2}} d_{k_0} e^{\theta_{32}(\bar{k}_0)} n_2(\omega \bar{k}_0^{-1}), 
\end{align}
and, for each $k_{0}\in \mathsf{Z}\cap \mathbb{R}$, 
\begin{align}\nonumber
& \underset{k = k_0}{\res} n_2(k) = c_{k_0} e^{-\theta_{21}(k_0)} n_1(k_0), & & \underset{k = \omega k_0}{\res} n_1(k) = \omega c_{k_0} e^{-\theta_{21}(k_0)} n_3(\omega k_0), 
	\\  \nonumber
& \underset{k = \omega^2 k_0}{\res} n_3(k) = \omega^2 c_{k_0} e^{-\theta_{21}(k_0)} n_2(\omega^2 k_0), & & \underset{k = k_0^{-1}}{\res} n_1(k) = - k_0^{-2} c_{k_0} e^{-\theta_{21}(k_0)} n_2(k_0^{-1}),		
	\\ \label{nresiduesk0real}
&  \underset{k = \omega^2k_0^{-1}}{\res} n_2(k) = -\tfrac{\omega^2}{k_0^{2}} c_{k_0} e^{-\theta_{21}(k_0)} n_3(\omega^2 k_0^{-1}), & & \underset{k = \omega k_0^{-1}}{\res} n_3(k) 
= -\tfrac{\omega}{k_0^{2}} c_{k_0} e^{-\theta_{21}(k_0)} n_1(\omega k_0^{-1}),
\end{align}
where the $(x,t)$-dependence of $n$ and of $\theta_{ij}$ has been omitted for conciseness and the complex constants $d_{k_0}$ are defined by
\begin{align}\label{dk0def}
d_{k_{0}} := \frac{\bar{k}_0^2-1}{\omega^2 (\omega^2 - \bar{k}_0^2)} \bar{c}_{k_0}, \qquad k_{0}\in \mathsf{Z}\setminus \mathbb{R}.
\end{align}
%Each condition in \eqref{nresiduesk0} and \eqref{nresiduesk0real} should be interpreted as stating that exactly one entry of $n$ has a simple pole, while the other two entries remain bounded; for example, in \eqref{res n 1 ab}, $\underset{k = k_0}{\res} n_3(k) = C_{k_{0}}(x,t) n_1(k_0)$ is shorthand notation for
%\begin{align*}
%n(z) = \frac{C_{k_{0}}(x,t)}{k-k_{0}} n(z) \begin{pmatrix}
%0 & 0 & 1 \\
%0 & 0 & 0 \\
%0 & 0 & 0 
%\end{pmatrix} + O(1), \qquad \mbox{as } k \to k_{0}.
%\end{align*}
\end{enumerate}
\end{RHproblem}
If $k_0 \in \mathsf{Z}$ is such that $c_{k_0} = 0$, then the pole of $n$ at $k_0$ is removable so $k_0$ can be removed from $\mathsf{Z}$ without affecting RH problem \ref{RHn}. In the rest of this paper, we will therefore assume that $c_{k_0} \neq 0$ for all $k_0 \in \mathsf{Z}$.

\section{Brief overview of the proof}\label{overviewsec}
Under the assumptions of Theorem \ref{asymptoticsth}, RH problem \ref{RHn} has a unique solution $n(x,t,k)$ for each $(x,t) \in \R \times [0,\infty)$, the function
$$n_{3}^{(1)}(x,t) := \lim_{k\to \infty} k (n_{3}(x,t,k) -1)$$ 
is well-defined and smooth for $(x,t) \in \R \times [0,\infty)$, and
\begin{align}\label{recoveruvn}
u(x,t) := -i\sqrt{3}\frac{\partial}{\partial x}n_{3}^{(1)}(x,t)
\end{align}
is a Schwartz class solution of (\ref{boussinesq}) on $\R \times [0,\infty)$ with initial data $u_{0},u_{1}$ \cite{CLscatteringsolitons}.
By using Deift--Zhou steepest descent arguments, we will derive an asymptotic formula for $n_{3}^{(1)}$. Substituting this formula into (\ref{recoveruvn}), we will obtain the formula of Theorem \ref{asymptoticsth}.

The steepest descent analysis of RH problem \ref{RHn} involves the saddle points of the three phase functions $\Phi_{21}$, $\Phi_{31}$, and $\Phi_{32}$ whose signature tables are displayed in Figure \ref{II fig: Re Phi 21 31 and 32 for zeta=0.7}.
The saddle points $\{k_{j}\}_{j=1}^{4}$ of $\Phi_{21}$ are given in (\ref{def of kj}).
Using the relations
\begin{align}\label{relations between the different Phi}
\Phi_{31}(\zeta,k) = - \Phi_{21}(\zeta,\omega^{2}k), \qquad \Phi_{32}(\zeta, k) = \Phi_{21}(\zeta, \omega k),
\end{align}
we see that $\{\omega k_{j}\}_{j=1}^{4}$ are the saddle points of $\Phi_{31}$ and that $\{\omega^{2} k_{j}\}_{j=1}^{4}$ are the saddle points of $\Phi_{32}$. It turns out that the main contributions to the long-time asymptotics of $u(x,t)$ in Sector IV are generated by a global parametrix $\Delta^{-1}$ and from twelve local parametrices near the saddle points $\{k_{j},\omega k_{j},\omega^{2}k_{j}\}_{j=1}^{4}$.

\begin{figure}[h]
\begin{center}
\begin{tikzpicture}[master]
\node at (0,0) {\includegraphics[width=4.5cm]{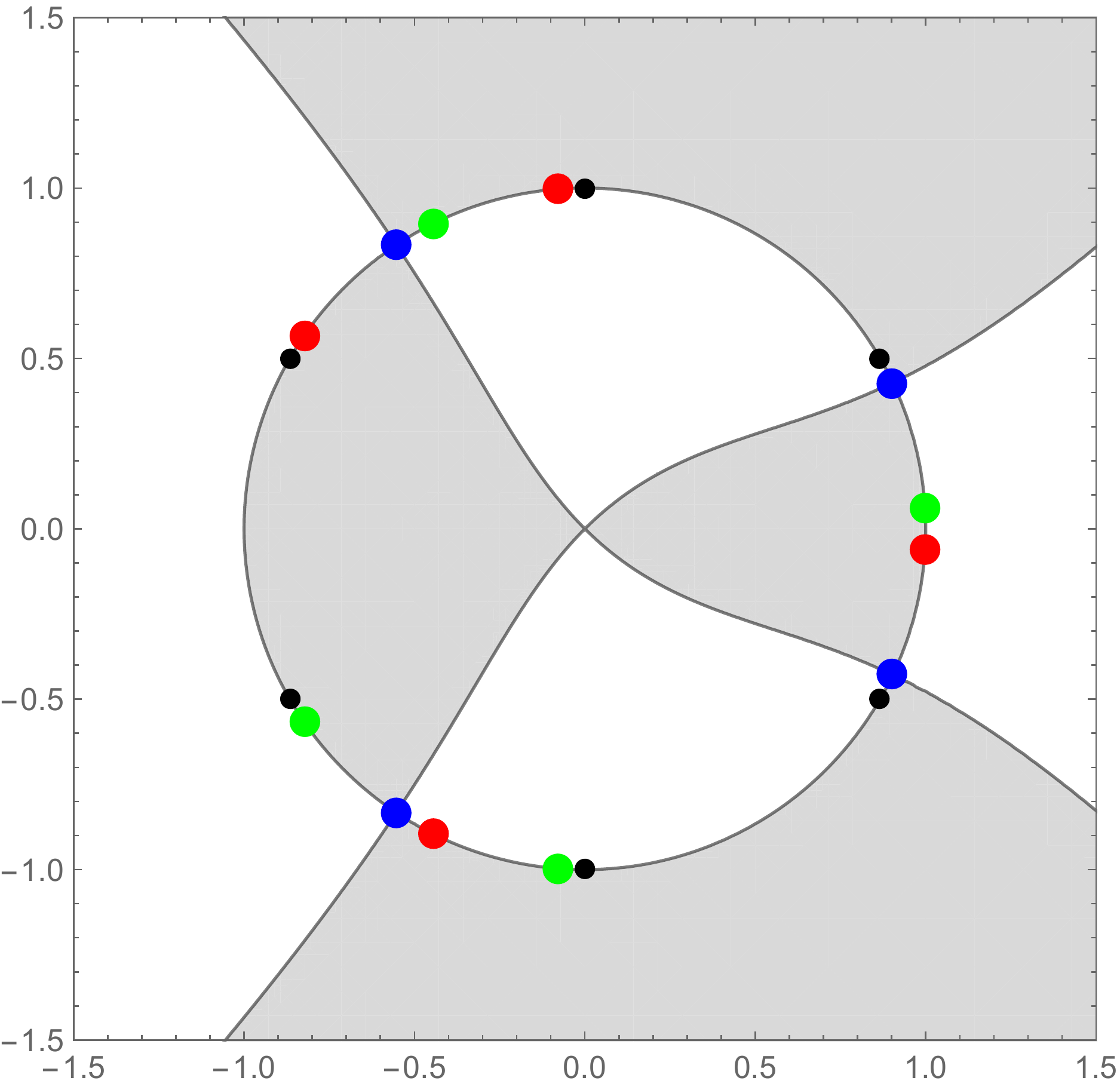}};

\node at (-0.87,1.27) {\tiny $k_1$};
\node at (-0.87,-1.06) {\tiny $k_2$};
\node at (1.55,0.58) {\tiny $k_3$};
\node at (1.55,-0.5) {\tiny $k_4$};

\end{tikzpicture} \hspace{0.1cm} \begin{tikzpicture}[slave]
\node at (0,0) {\includegraphics[width=4.5cm]{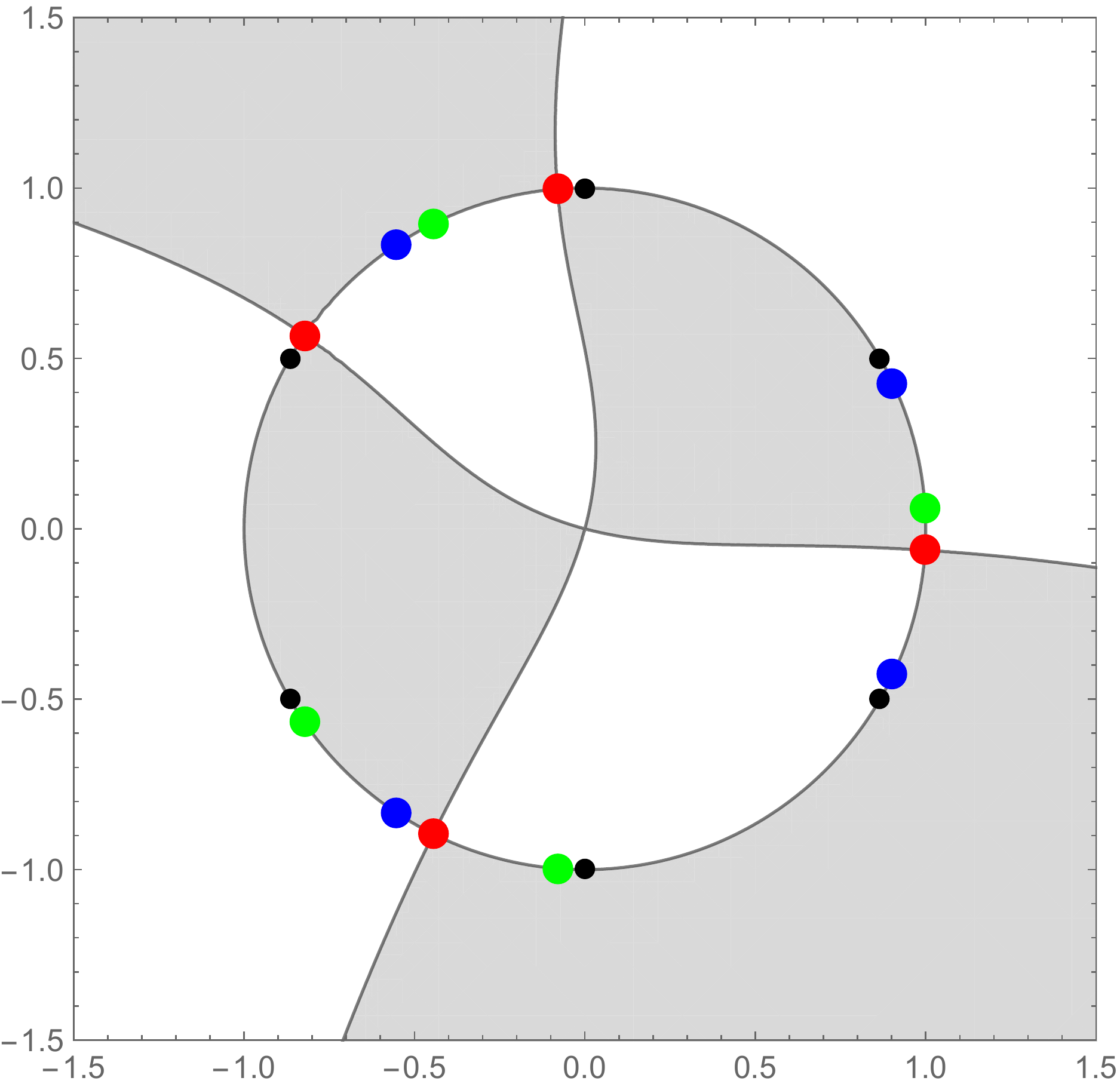}};

\node at (-.37,-1.35) {\tiny $\omega k_1$};
\node at (1.75,-0.15) {\tiny $\omega k_2$};
\node at (-1.15,1.08) {\tiny $\omega k_3$};
\node at (-0.05,1.6) {\tiny $\omega k_4$};
\end{tikzpicture} \hspace{0.1cm} \begin{tikzpicture}[slave]
\node at (0,0) {\includegraphics[width=4.5cm]{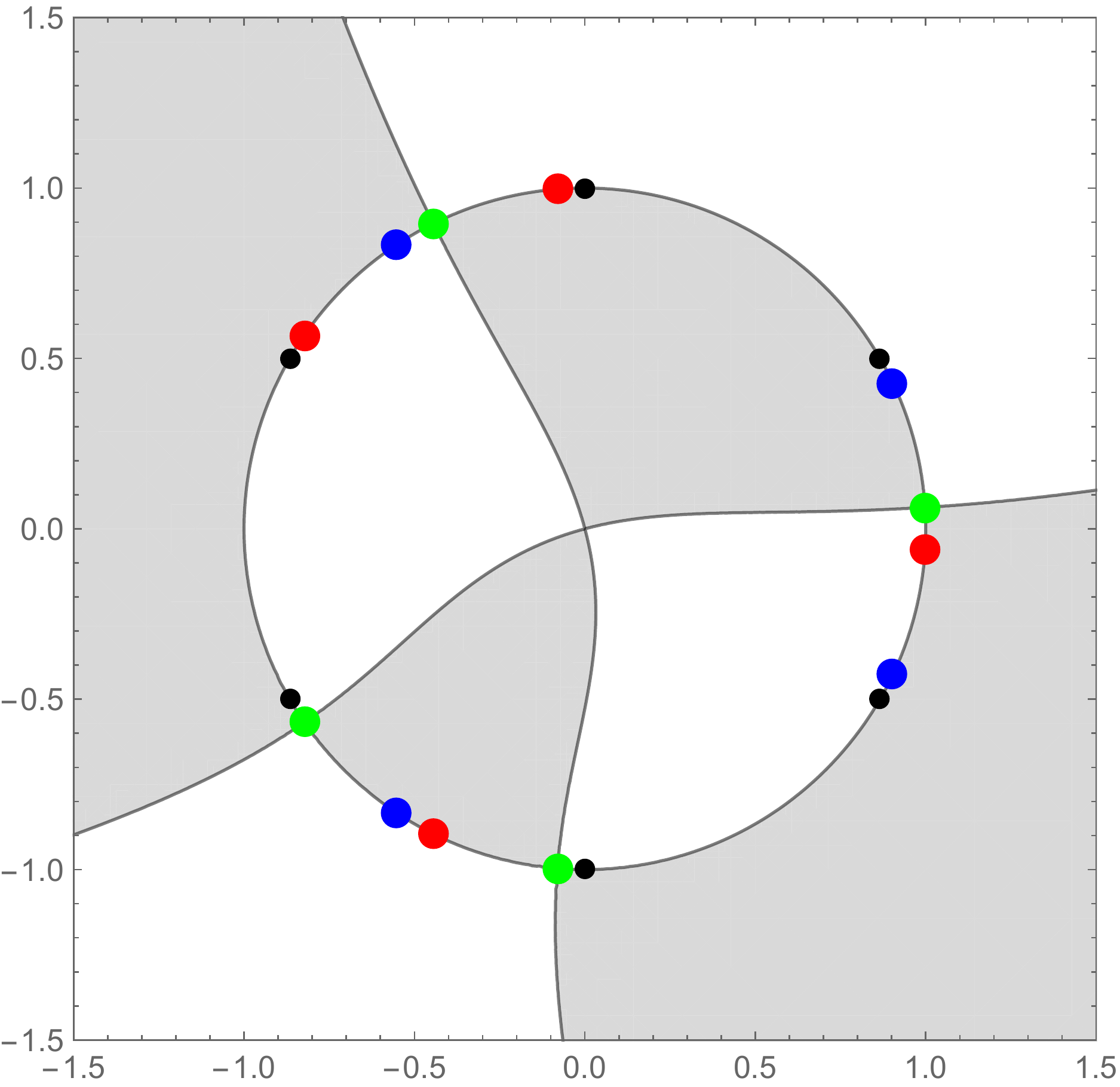}};

\node at (1.8,.3) {\tiny $\omega^2 k_1$};
\node at (-.6,1.5) {\tiny $\omega^2 k_2$};
\node at (-0.05,-1.5) {\tiny $\omega^2 k_3$};
\node at (-1.4,-.7) {\tiny $\omega^2 k_4$};
\end{tikzpicture}
\end{center}
\begin{figuretext}
\label{II fig: Re Phi 21 31 and 32 for zeta=0.7} From left to right: The signature tables for $\Phi_{21}$, $\Phi_{31}$, and $\Phi_{32}$ in Sector \IV. % for $\zeta=0.7$. 
In all images, the shaded regions correspond to $\{k \, | \, \re \Phi_{ij}>0\}$ and the white regions to $\{k \, | \, \re \Phi_{ij}<0\}$. The points $k_{1},k_{2}, k_{3}, k_{4}$ are represented in blue, $\omega k_{1},\omega k_{2},\omega k_{3},\omega k_{4}$ in red, and $\omega^{2} k_{1},\omega^{2} k_{2},\omega^{2} k_{3},\omega^{2} k_{4}$ in green. The smaller black dots are the points $i\kappa_j$, $j=1,\ldots,6$.
\end{figuretext}
\end{figure}

Our proof uses a series of transformations $n \to n^{(1)} \to n^{(2)} \to n^{(3)}  \to n^{(4)} \to \hat{n}$ to bring RH problem \ref{RHn} to a small-norm RH problem. These transformations are invertible, implying that the RH problems satisfied by $\{n^{(j)}\}_1^4$ and $\hat{n}$ are equivalent to RH problem \ref{RHn}. The jump contours for the RH problems for $\{n^{(j)}\}_1^4$ and $\hat{n}$ will be denoted by $\{\Gamma^{(j)}\}_1^4$ and $\hat{\Gamma}$ and the associated jump matrices by $\{v^{(j)}\}_1^4$ and $\hat{v}$, respectively.

The jump matrix $v$ in \eqref{vdef} obeys the following symmetries in accordance with (\ref{nsymm}):
\begin{align}\label{vsymm}
v(x,t,k) = \mathcal{A} v(x,t,\omega k)\mathcal{A}^{-1}
 = \mathcal{B} v(x,t, k^{-1})^{-1}\mathcal{B}, \qquad k \in \Gamma.
\end{align}
We will preserve the symmetries (\ref{nsymm}) and (\ref{vsymm}) at each step, so that, for $j = 1, 2,3,4$,
\begin{align}\label{vjsymm}
& v^{(j)}(x,t,k) = \mathcal{A} v^{(j)}(x,t,\omega k)\mathcal{A}^{-1}
 = \mathcal{B} v^{(j)}(x,t,k^{-1})^{-1}\mathcal{B}, & & k \in \Gamma^{(j)},
	\\ \label{njsymm}
& n^{(j)}(x,t, k) = n^{(j)}(x,t,\omega k)\mathcal{A}^{-1}
 = n^{(j)}(x,t, k^{-1}) \mathcal{B}, & & k \in \C \setminus \Gamma^{(j)},
\end{align}
and similarly for $\hat{n}$ and $\hat{v}$. These symmetries imply that we only need to construct the local parametrices near $\omega k_{4}$ and $\omega^{2}k_{2}$, and that we only need to define the transformations $n \to n^{(1)}$, $n^{(j)} \to n^{(j+1)}$, and $n^{(4)}\to \hat{n}$ in the sector $\mathsf{S}:=\{k \in \mathbb{C}|  \arg k \in [\frac{\pi}{3},\frac{2\pi}{3}]\}$. In $\mathbb{C}\setminus \mathsf{S}$, we will define the local parametrices and the transformations using the $\mathcal{A}$- and $\mathcal{B}$-symmetries. 

We know from \cite[Theorem 2.3]{CLmain} that the functions $r_{1}$ and $r_{2}$ satisfy the following properties: $r_1 \in C^\infty(\hat{\Gamma}_{1})$, $r_2 \in C^\infty(\hat{\Gamma}_{4}\setminus \{\omega^{2}, -\omega^{2}\})$, $r_{1}(\kappa_{j})\neq 0$ for $j=1,\ldots,6$, $r_{2}(k)$ has simple poles at $k=\pm \omega^2$, and simple zeros at $k=\pm\omega$, and $r_{1},r_{2}$ are rapidly decreasing as $|k|\to \infty$. Moreover,
\begin{align}
& r_{1}(\tfrac{1}{\omega k}) + r_{2}(\omega k) + r_{1}(\omega^{2} k) r_{2}(\tfrac{1}{k}) = 0, & & k \in \partial \mathbb{D}\setminus \{\omega, -\omega\}, \label{r1r2 relation on the unit circle} \\
& r_{2}(k) = \tilde{r}(k)\overline{r_{1}(\bar{k}^{-1})}, \qquad \tilde{r}(k) :=\frac{\omega^{2}-k^{2}}{1-\omega^{2}k^{2}}, & & k \in \hat{\Gamma}_{4}\setminus \{0, \omega^{2}, -\omega^{2}\}, \label{r1r2 relation with kbar symmetry} \\
& r_{1}(1) = r_{1}(-1) = 1, \qquad r_{2}(1) = r_{2}(-1) = -1. \label{r1r2at0}
\end{align}
These properties will be used repeatedly throughout the proof.

The residue conditions (\ref{nresiduesk0}) involve the exponentials $e^{-\theta_{31}(x,t,k_0)}$ and $e^{\theta_{32}(x,t,\bar{k}_0)}$. We see from Figure \ref{II fig: Re Phi 21 31 and 32 for zeta=0.7} that these exponentials are small as $t \to \infty$ for $k_0 \in \mathsf{Z}\setminus \mathbb{R}$ if $\re k_0 < 0$ (recall that $\mathsf{Z}\setminus \mathbb{R} \subset D_{\mathrm{reg}}$). 
Similarly, the residue conditions (\ref{nresiduesk0real}) involve the exponential $e^{-\theta_{21}(x,t,k_0)}$, and we see from Figure \ref{II fig: Re Phi 21 31 and 32 for zeta=0.7} that this exponential is small as $t \to \infty$ for $k_0 \in \mathsf{Z} \cap \mathbb{R}$ if $\re k_0 < 0$ (recall that $\mathsf{Z}\cap \mathbb{R} \subset (-1,0)\cup (1,\infty)$). 
This suggests that points $k_0$ in $\mathsf{Z}$ with $\re k_0 < 0$ (i.e., points corresponding to left-moving solitons) will not contribute to the asymptotics in Sector IV at any finite order, and this is indeed what we will find. 
On the other hand, for $k_0 \in \mathsf{Z}$ with $\re k_0 > 0$, the above exponentials are large as $t \to \infty$. By including an appropriate factor in the global parametrix, we can flip the signs of the exponents in these exponentials, thereby making them small for large $t$. The factor included in the global parametrix gives rise to the phase shifts $\arg \tfrac{\mathcal{P}(\omega k_{4})}{\mathcal{P}(\omega^{2} k_{4})}$ and $\arg \tfrac{\mathcal{P}(\omega^{2} k_{2})}{\mathcal{P}(\omega k_{2})}$ generated by right-moving solitons seen in Theorem \ref{asymptoticsth}.

\section{The $n \to n^{(1)}$ transformation}\label{nton1sec}

We will open lenses around $\partial \mathbb{D}\setminus \mathcal{Q}$. Define
$$\hat{r}_{j}(k) := \frac{r_{j}(k)}{1+r_{1}(k)r_{2}(k)}, \qquad j=1,2,$$
and let
\begin{align*}
& r_{j,1}(k) := r_{j}(k), \quad r_{j,2}(k) := \hat{r}_{j}(k), \quad r_{j,3}(k) := \frac{r_{j}(k)-r_{j}(\frac{1}{\omega k})r_{j}(\frac{1}{\omega^{2} k})}{1+r_{1}(\frac{1}{\omega k})r_{2}(\frac{1}{\omega k})}, \\
& r_{j,4}(k) := \frac{r_{j}(k)}{f(k)}, \quad r_{j,5}(k) = \frac{r_{j}(k)-r_{j}(\frac{1}{\omega k})r_{j}(\frac{1}{\omega^{2}k})}{f(k)}.
\end{align*}
The domains of definition of these spectral functions are (in general) only subsets of $\Gamma$. To open lenses around $\partial \mathbb{D}$, it is therefore necessary to decompose them into  an analytic part and a small remainder. Let $M>1$ and define open sets $\{U_{1}^{\ell}=U_{1}^{\ell}(\zeta,M)\}_{\ell=1}^{5}$ by (see Figure \ref{IIbis fig: U1 and U2})
\begin{align*}
U_{1}^{1} & = \{k| \re \Phi_{21}>0 \} \cap \big(\{k| \arg k \in (\arg(\omega k_{3}),\pi] \cup [-\pi,-\tfrac{5\pi}{6}), \; M^{-1}<|k|<1 \} \\
& \quad \; \cup \{k| \arg k \in(-\tfrac{\pi}{2},-\tfrac{\pi}{6})\cup (\arg k_{3},\tfrac{\pi}{6}), \; 1 < |k| < M\} \big), \\
U_{1}^{2} & = \{k| \re \Phi_{21}>0 \} \cap \{k|\arg k \in (\tfrac{\pi}{2},\arg(\omega^{2}k_{2})), \; 1 < |k| < M\}, \\
U_{1}^{3} & = \{k| \re \Phi_{21}>0 \} \cap \{k| \arg k \in (\arg k_{1}, \arg (\omega k_{3})), \; M^{-1}<|k|<1\}, \\
U_{1}^{4} & = \{k| \re \Phi_{21}>0 \} \cap \{k| \arg k \in (\arg(\omega k_{2}),\arg k_{3}), \; M^{-1}<|k|<1\}, \\
U_{1}^{5} & = \{k| \re \Phi_{21}>0 \} \cap \{k| \arg k \in (\arg(\omega^{2}k_{2}),\arg k_{1}), \; 1<|k|<M)\}.
\end{align*}
It follows from \cite[Lemma 2.13]{CLmain} that $f(k)=0$ if and only if $k \in \{\pm 1, \pm \omega\}$, and that $1+r_{1}(k)r_{2}(k)>0$ for all $k \in \partial \mathbb{D}$ with $\arg k \in (\pi/3,\pi)\cup (-2\pi/3,0)$. Therefore, for each $\ell \in \{1,2,3\}$, $r_{1,\ell}$ is well-defined for $k \in \partial U_{1}^{\ell} \cap \partial \mathbb{D}$; $r_{1,4}$ is well-defined for $k \in (\partial U_{1}^{4} \cap \partial \mathbb{D})\setminus\{1\}$; and $r_{1,5}$ is well-defined for $k \in (\partial U_{1}^{5} \cap \partial \mathbb{D})\setminus\{\omega\}$. Let $n_{1}\geq 0$ denote the smallest integer such that $(k-1)^{n_{1}}r_{1,4}(k)=O(1)$ as $k \to 1$, $k \in \partial \mathbb{D}$, and let $n_{\omega} \geq 0$ denote the smallest integer such that $(k-\omega)^{n_{\omega}}r_{1,5}(k)=O(1)$ as $k \to \omega$, $k \in \partial \mathbb{D}$.

We now find decompositions of $\{r_{1,\ell}\}_{\ell=1}^{5}$;  the decompositions of $\{r_{2,\ell}\}_{\ell=1}^{5}$ will then be obtained using the symmetry \eqref{r1r2 relation with kbar symmetry}. Let $N \geq 1$ be an integer.

\begin{figure}
\begin{center}
\begin{tikzpicture}[master]
\node at (0,0) {\includegraphics[width=4.5cm]{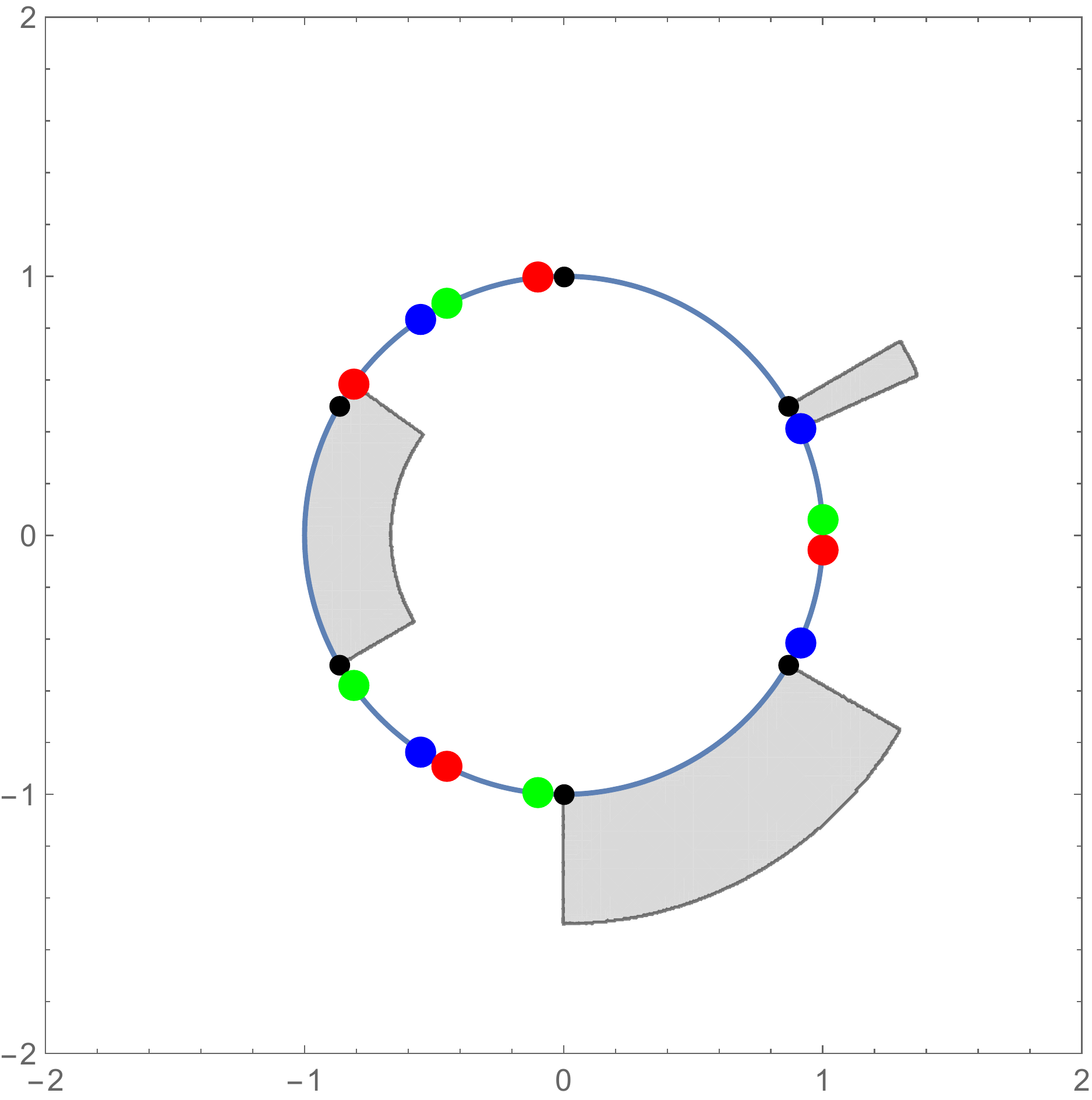}};
\node at (-0.8,0.15) {\tiny $U_{1}^{1}$};
\node at (0.95,-0.95) {\tiny $U_{1}^{1}$};
\node at (1.2,1.5) {\tiny $U_{1}^{1}$};
\draw[black,line width=0.15 mm,->-=1] (1.2,1.35)--(1.2,0.65);
\end{tikzpicture} \hspace{0.5cm} \begin{tikzpicture}[slave]
\node at (0,0) {\includegraphics[width=4.5cm]{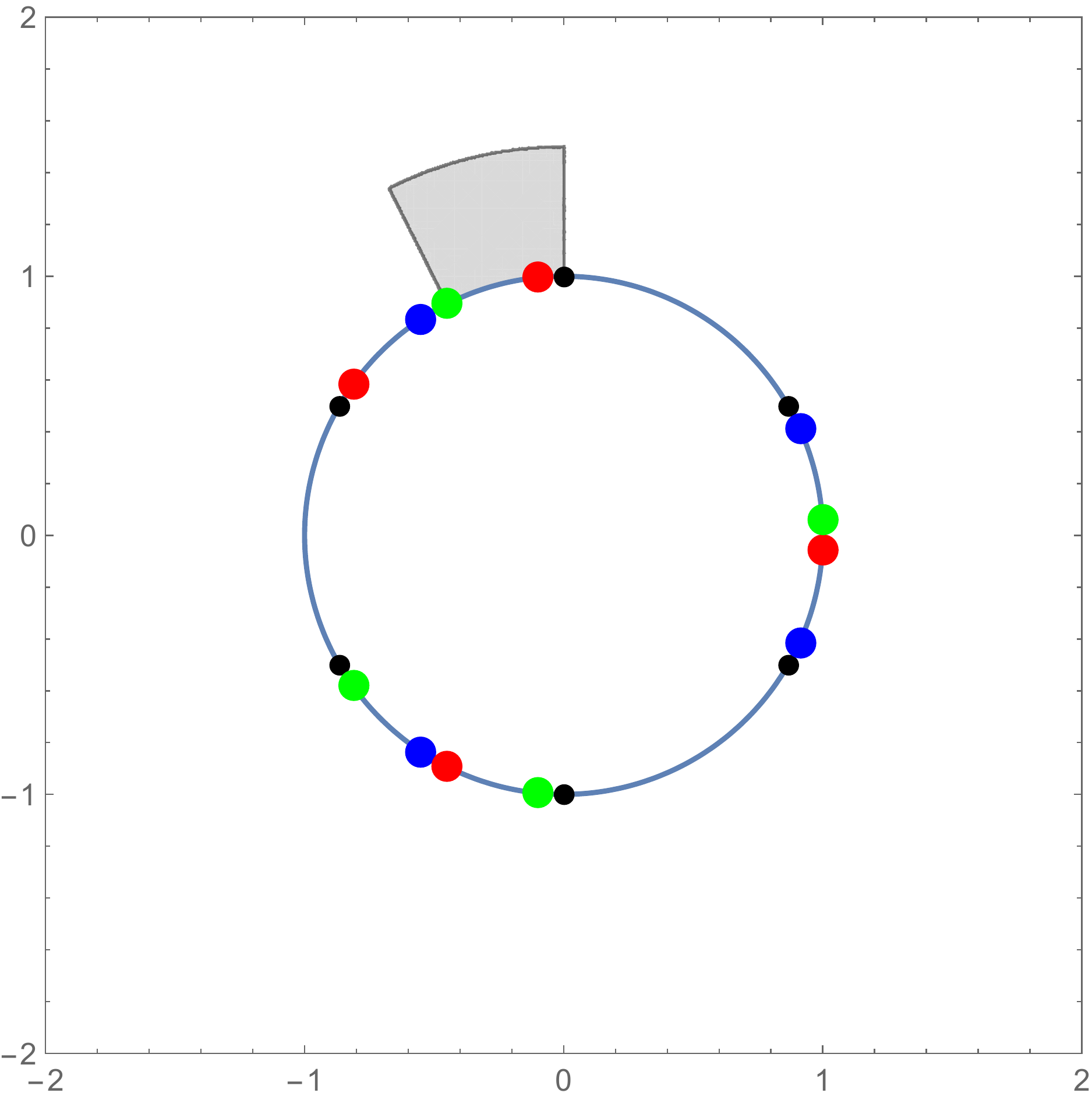}};
\node at (-0.25,1.35) {\tiny $U_{1}^{2}$};
\end{tikzpicture} \\
\begin{tikzpicture}[slave]
\node at (0,0) {\includegraphics[width=4.5cm]{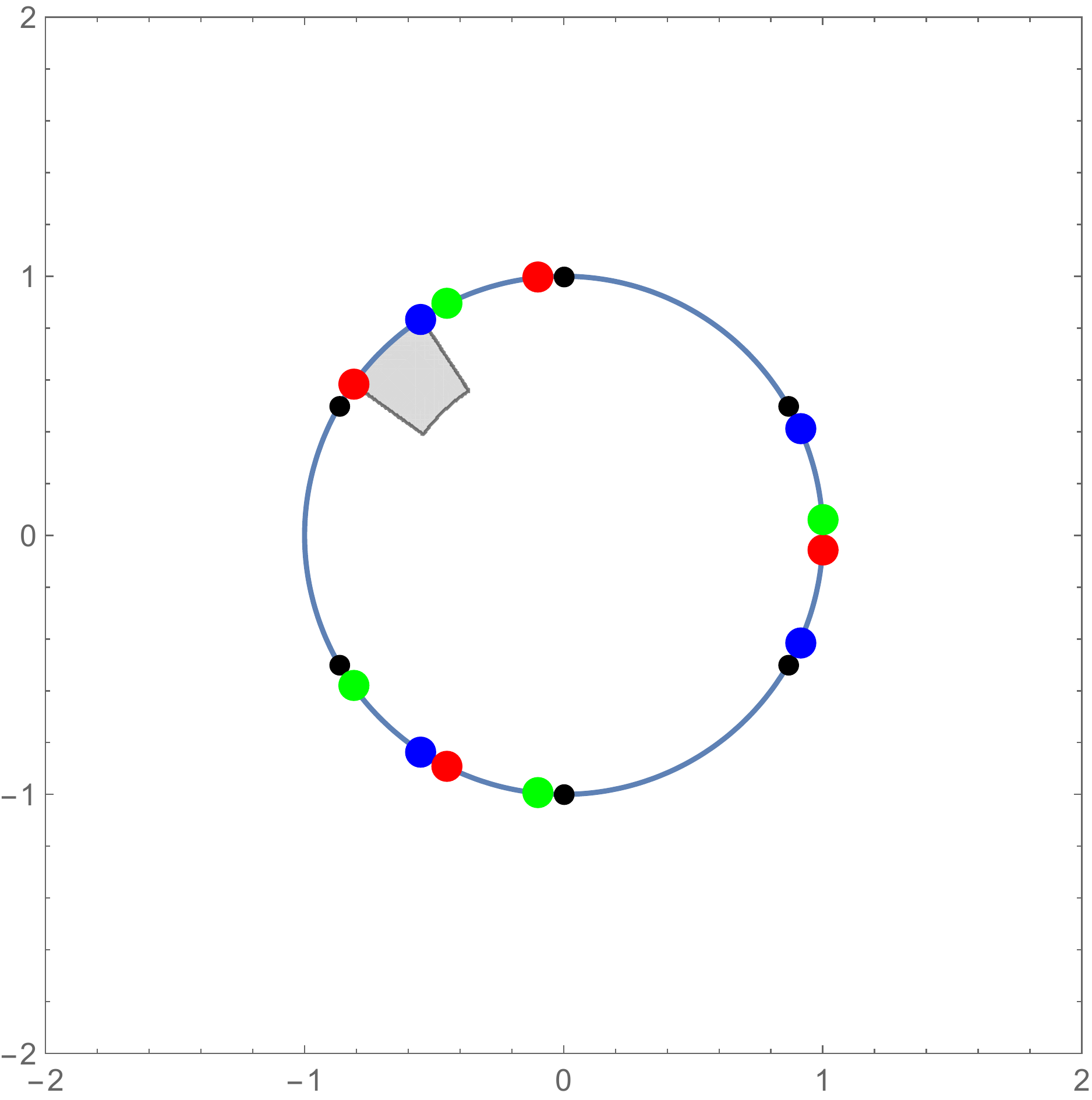}};
\node at (-0.55,0.65) {\tiny $U_{1}^{3}$};
\end{tikzpicture} \hspace{-0.05cm}
\begin{tikzpicture}[slave]
\node at (0,0) {\includegraphics[width=4.5cm]{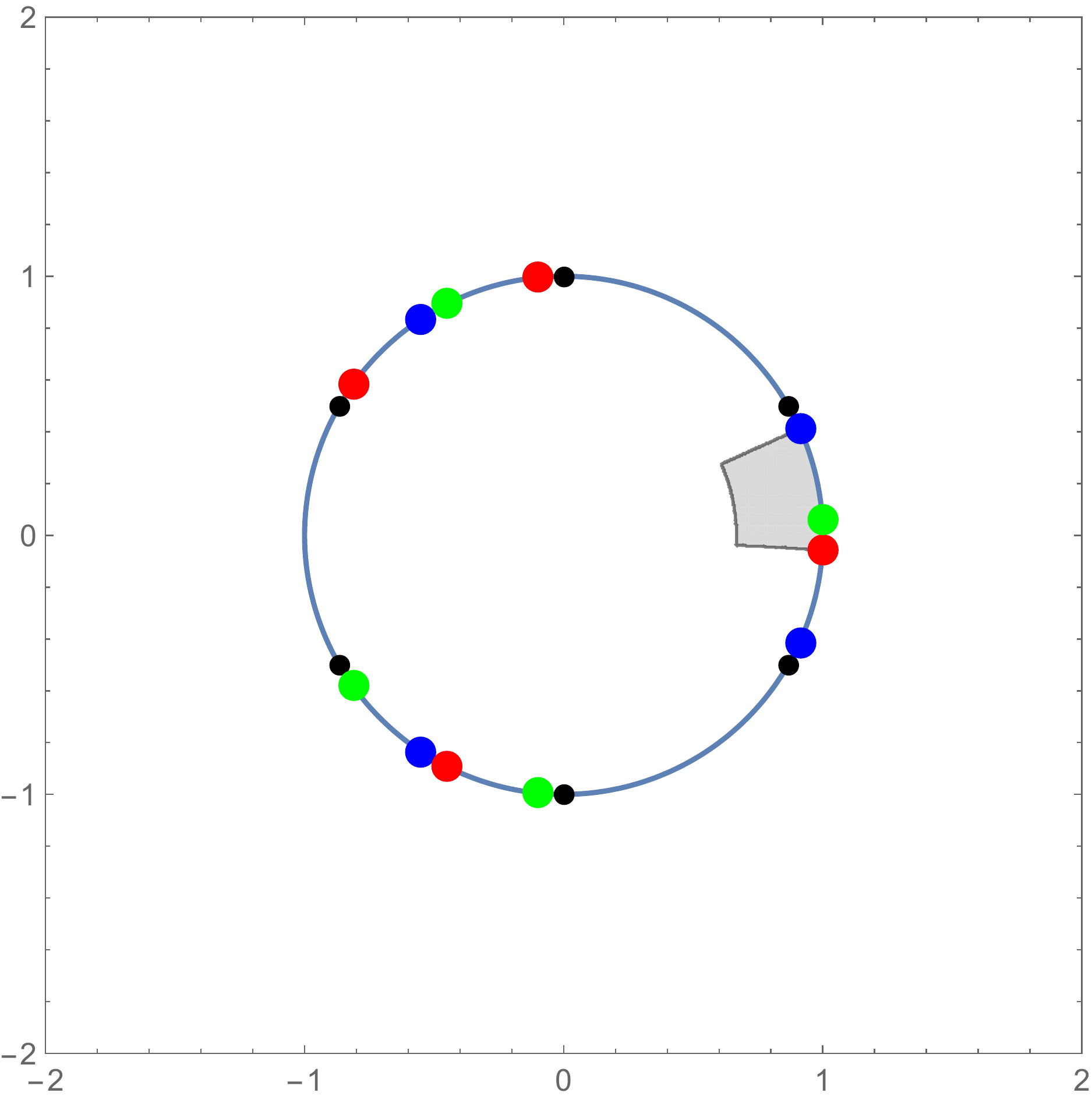}};
\node at (0.95,0.2) {\tiny $U_{1}^{4}$};
\end{tikzpicture} \hspace{-0.05cm}
\begin{tikzpicture}[slave]
\node at (0,0) {\includegraphics[width=4.5cm]{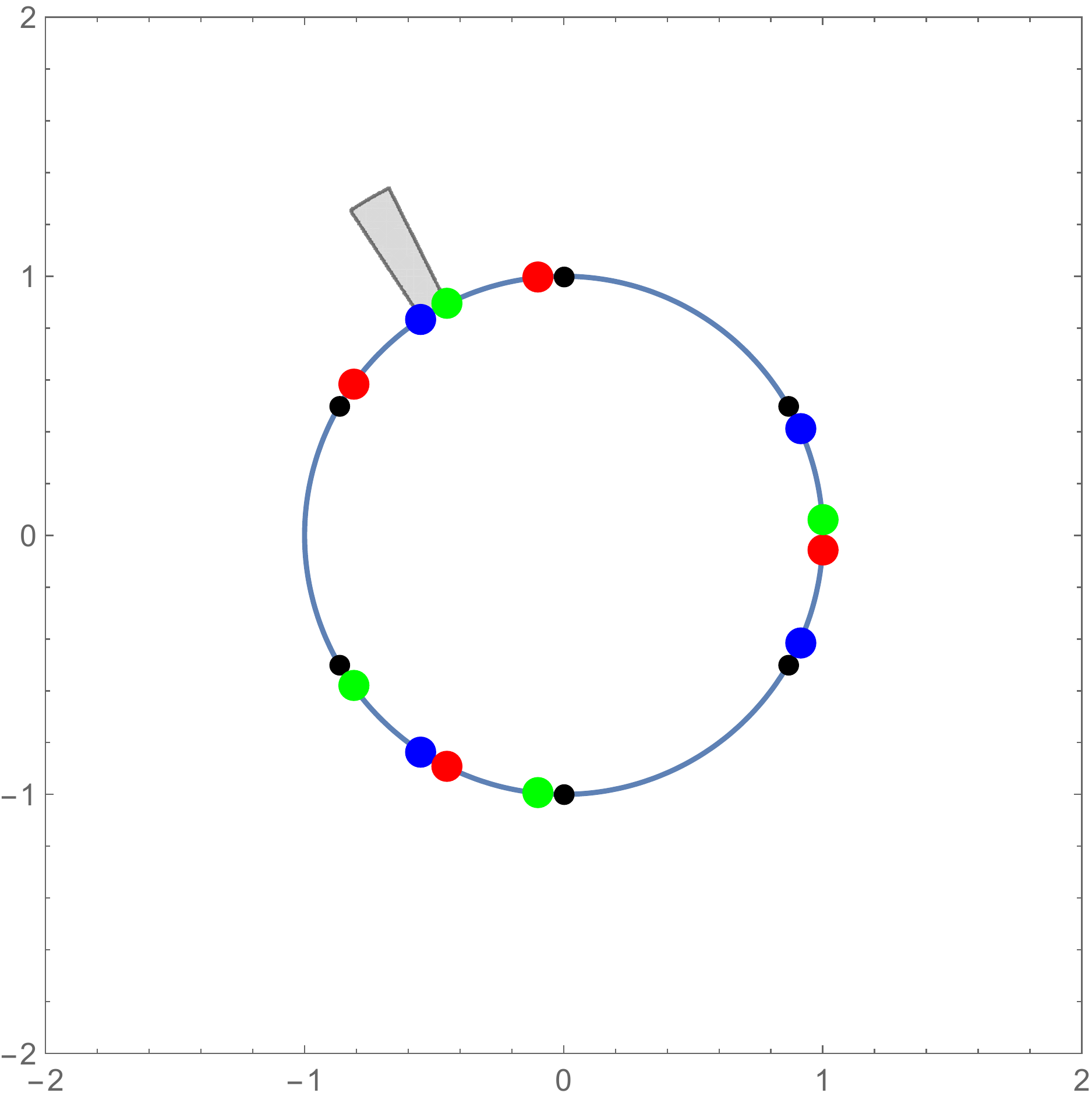}};
\node at (0,1.6) {\tiny $U_{1}^{5}$};
\draw[black,line width=0.15 mm,->-=1] (-0.25,1.5)--(-0.65,1.3);
\end{tikzpicture}
\end{center}
\begin{figuretext}
\label{IIbis fig: U1 and U2} The open subsets $\{U_{1}^{\ell}\}_{\ell=1}^{5}$ of the complex $k$-plane.
\end{figuretext}
\end{figure}

\begin{lemma}[Decomposition lemma]\label{IIbis decompositionlemma}
There exist $M>1$ and decompositions
\begin{align}
& r_{1,\ell}(k) = r_{1,\ell,a}(x, t, k) + r_{1,\ell,r}(x, t, k), & & k \in \partial U_{1}^{\ell} \cap \partial \mathbb{D}, \; \ell =1,\ldots,5, \label{IIbis decomposition lemma analytic + remainder}
\end{align}
such that $\{r_{1,\ell,a},r_{1,\ell,r}\}_{\ell=1}^{5}$ satisfy the following properties:
\begin{enumerate}[$(a)$]
\item 
For each $\zeta \in \mathcal{I}$, $t \geq 1$, and $\ell \in \{1,\dots,5\}$, $r_{1,\ell,a}(x, t, k)$ is defined and continuous for $k \in \bar{U}_1^{\ell}$ and analytic for $k \in U_1^{\ell}$.

\item For each $\zeta \in \mathcal{I}$, $t \geq 1$, and $\ell \in \{1,\dots,5\}$, $r_{1,\ell,a}$ obeys
\begin{align*}
& \Big| r_{1,\ell,a}(x, t, k)-\sum_{j=0}^{N}\frac{r_{1,\ell}^{(j)}(k_{\star})}{j!}(k-k_{\star})^{j}  \Big| \leq C |k-k_{\star}|^{N+1}e^{\frac{t}{4}|\re \Phi_{21}(\zeta,k)|}, \; k \in \bar{U}_{1}^{\ell}, \; k_{\star} \in \mathcal{R}_{\ell},   
\end{align*}
where $\mathcal{R}_{1}=\{\omega k_{3},e^{\pm 5\pi i/6},-1,-i,-\omega,e^{\pm\frac{\pi i}{6}} ,k_{1}\}$, $\mathcal{R}_{2}=\{i,\omega k_{4},\omega^{2}k_{2}\}$, $\mathcal{R}_{3}=\{k_{1},\omega k_{3}\}$, $\mathcal{R}_{4}=\{\omega k_{2},\omega^{2}k_{1},k_{3}\}$, $\mathcal{R}_{5}=\{\omega^{2} k_{2}, k_{1}\}$, and the constant $C$ is independent of $\zeta, t, k$. Moreover, for $\zeta \in \mathcal{I}$ and $t\geq 1$,
\begin{align*}
& \Big| r_{1,4,a}(x, t, k)-\sum_{j=0}^{N+n_{1}}\frac{[(\cdot-1)^{n_{1}}r_{1,4}(\cdot)]^{(j)}(1)}{j!}(k-1)^{j-n_{1}}  \Big| \leq C |k-1|^{N+1}e^{\frac{t}{4}|\re \Phi_{21}(\zeta,k)|}, \\
& \Big| r_{1,5,a}(x, t, k)-\sum_{j=0}^{N+n_{\omega}}\frac{[(\cdot-\omega)^{n_{\omega}}r_{1,5}(\cdot)]^{(j)}(\omega)}{j!}(k-\omega)^{j-n_{\omega}}  \Big| \leq C |k-\omega|^{N+1}e^{\frac{t}{4}|\re \Phi_{21}(\zeta,k)|},
\end{align*}
and these estimates hold for $k \in \bar{U}_{1}^{4}$ and $k \in \bar{U}_{1}^{5}$, respectively.
\item For each $1 \leq p \leq \infty$ and $\ell \in \{1,\dots,5\}$, the $L^p$-norm of $r_{1,\ell,r}(x,t,\cdot)$ on $\partial \bar{U}_{1}^{\ell} \cap \partial \mathbb{D}$ is $O(t^{-N})$ uniformly for $\zeta \in \mathcal{I}$ as $t \to \infty$.
\end{enumerate}
\end{lemma}
\begin{proof}
The function $\theta \mapsto -i \Phi_{21}(\zeta,e^{i\theta})=(\zeta-\cos \theta)\sin \theta$ is real-valued and bijective on each connected component of $\partial U_{1}^{\ell}\cap \partial \mathbb{D}$, $\ell = 1,\ldots,5$. Thus the statement can be proved using similar arguments as in \cite{DZ1993}. Since these arguments are rather standard by now, we omit details. 
\end{proof}

In what follows, shrinking $M>1$ if necessary, we assume that $|k_{0}|>M$ for all $k_{0}\in \mathsf{Z} \cap \big(D_{\mathrm{reg}}^{R}\cup (1,\infty)\big)$ and that $|k_{0}|<M^{-1}$ for all $k_{0}\in \mathsf{Z}\cap \big(D_{\mathrm{reg}}^{L}\cup (-1,0)\big)$.

It is easy to check (using the second equation in \eqref{r1r2 relation with kbar symmetry}) that $\tilde{r}(k) \in \mathbb{R}$ for $k \in \partial \mathbb{D}\setminus \{-\omega^{2},\omega^{2}\}$ and that $\tilde{r}(k)= \tilde{r}(\frac{1}{\omega k})\tilde{r}(\frac{1}{\omega^{2} k})$. Using also the first equation in \eqref{r1r2 relation with kbar symmetry}, we infer that $r_{2,\ell}(k) = \tilde{r}(k)\overline{r_{1,\ell}(\bar{k}^{-1})}$, $k \in \partial \mathbb{D}\cap U_{1}^{\ell}$, $\ell=1,\ldots,5$. Hence, for each $\ell\in \{1,\ldots,5\}$ we let $U_{2}^{\ell}:=\{k|\bar{k}^{-1}\in U_{1}^{\ell}\}$ and define a decomposition $r_{2,\ell}=r_{2,\ell,a}+r_{2,\ell,r}$ by
\begin{align*}
& r_{2,\ell,a}(k) := \tilde{r}(k)\overline{r_{1,\ell,a}(\bar{k}^{-1})}, \quad k \in U_{2}^{\ell},
& & r_{2,\ell,r}(k) := \tilde{r}(k)\overline{r_{1,\ell,r}(\bar{k}^{-1})}, \quad k \in \partial U_{2}^{\ell} \cap \partial \mathbb{D}.
\end{align*}

\begin{figure}[h]
\begin{center}
\begin{tikzpicture}[master,scale=0.9]
\node at (0,0) {};
\draw[black,line width=0.65 mm] (0,0)--(30:7.5);
\draw[black,line width=0.65 mm,->-=0.45,->-=0.91] (0,0)--(90:7);
\draw[black,line width=0.65 mm] (0,0)--(150:7.5);
\draw[dashed,black,line width=0.15 mm] (0,0)--(60:7.8);
\draw[dashed,black,line width=0.15 mm] (0,0)--(120:7.8);

\draw[black,line width=0.65 mm] ([shift=(30:3*1.5cm)]0,0) arc (30:150:3*1.5cm);

\node at (82.5:1.97*1.5) {\small $1''$};

\node at (86.5:6.15) {\small $4'$};

\draw[black,arrows={-Triangle[length=0.27cm,width=0.18cm]}]
($(92:3*1.5)$) --  ++(92-90:0.001);
\draw[black,arrows={-Triangle[length=0.27cm,width=0.18cm]}]
($(110:3*1.5)$) --  ++(110+90:0.001);
\draw[black,arrows={-Triangle[length=0.27cm,width=0.18cm]}]
($(72:3*1.5)$) --  ++(72-90:0.001);
\draw[black,arrows={-Triangle[length=0.24cm,width=0.16cm]}]
($(120.5:3*1.5)$) --  ++(120.5+90:0.001);

% The angles have been slightly modified to make the figure more visible
\draw[blue,fill] (125:3*1.5) circle (0.1cm);
\draw[green,fill] (-125+240:3*1.5) circle (0.1cm);
\draw[red,fill] (100:3*1.5) circle (0.1cm);
\draw[red,fill] (140:3*1.5) circle (0.1cm);

%(111.471:3) to [out=111.471+45, in=120-90] (120:3.5)

\node at (97:2.73*1.5) {\small $\omega k_{4}$};
\node at (113:2.7*1.5) {\small $\omega^{2}k_{2}$};
\node at (124:2.7*1.5) {\small $k_{1}$};
\node at (140:2.7*1.5) {\small $\omega k_{3}$};

\node at (93:3.15*1.5) {\small $2$};
\node at (108:3.15*1.5) {\small $5$};
\node at (118:3.15*1.5) {\small $8$};
\node at (71:3.15*1.5) {\small $9_r$};

% ----------------------start the poles---------------------

%\draw[black,fill] (120:6.9) circle (0.1cm);
%\draw[black,fill] (120:4.5^2/6.9+0.01) circle (0.1cm);
%\draw[black,fill] (120+21:6.5) circle (0.1cm);
%\draw[black,fill] (120-21:6.5) circle (0.1cm);
%\draw[black,fill] (-21+120:3.125) circle (0.1cm);
%\draw[black,fill] (21+120:3.125) circle (0.1cm);

% ----------------------end the poles---------------------
\end{tikzpicture}
\end{center}
\begin{figuretext}
\label{Gamma0fig}The contour $\Gamma^{(0)}$ (solid), the boundary of $\mathsf{S}$ (dashed), %some possible poles of $n$ (black dots), 
and the saddle points $\omega k_{4}$ (red), $\omega^{2}k_{2}$ (green), $k_{1}$ (blue), and $\omega k_{3}$ (red).
\end{figuretext}
\end{figure}

We now define the first transformation $n \to n^{(1)}$, which consists in opening lenses around $\partial \mathbb{D}\setminus \mathcal{Q}$. As mentioned in Section \ref{overviewsec}, we only need to define it explicitly in the sector $\mathsf{S}$. We introduce a new contour $\Gamma^{(0)}$, which coincides as a set with $\Gamma\cap \mathsf{S}$, but is oriented and labeled differently, see Figure \ref{Gamma0fig}. We let $\Gamma_{j}^{(0)}$ denote the subcontour of $\Gamma^{(0)}$ labeled by $j$ in Figure \ref{Gamma0fig}. For example,
\begin{align*}
\Gamma_{8}^{(0)}:= \{e^{i\theta} \,|\, \theta \in (\arg(\omega^{2}k_{2}),\tfrac{2\pi}{3})\}, \qquad \Gamma_{9_{r}}^{(0)}:= \{e^{i\theta} \,|\, \theta \in (\tfrac{\pi}{3},\tfrac{\pi}{2})\}.
\end{align*}
We will open lenses differently on the four parts of $\Gamma^{(0)}$. On $\Gamma_{9_{r}}^{(0)}$, we will use the factorization
\begin{align*}
v_{9_{r}} = v_{9_L}^{(1)} v_{9_r}^{(1)} v_{9_R}^{(1)},
\end{align*}%def of $v_{9_r}^{(1)}$ is changed compared to main file. More consistent with other sectors.
where
\begin{align} \nonumber
& v_{9_L}^{(1)} = \begin{pmatrix}
1 & 0 & -r_{2,a}(\omega^{2}k)e^{-\theta_{31}} \\
r_{1,a}(\frac{1}{k})e^{\theta_{21}} & 1 & r_{1,a}(\omega k)e^{-\theta_{32}} \\
0 & 0 & 1
\end{pmatrix}, \quad v_{9_R}^{(1)} = \begin{pmatrix}
1 & r_{2,a}(\frac{1}{k})e^{-\theta_{21}} & 0 \\
0 & 1 & 0  \\
-r_{1,a}(\omega^{2}k)e^{\theta_{31}} & r_{2,a}(\omega k)e^{\theta_{32}} & 1
\end{pmatrix}, 
	\\ 
& v_{9_r}^{(1)} = I+\begin{pmatrix}
r_{1,r}(\omega^{2}k)r_{2,r}(\omega^{2}k) & g_{2}(\omega k)e^{-\theta_{21}} & -r_{2,r}(\omega^{2}k)e^{-\theta_{31}} \\
g_{1}(\omega k)e^{\theta_{21}} & g(\omega k) & h_{1}(\omega k)e^{-\theta_{32}} \\
-r_{1,r}(\omega^{2}k)e^{\theta_{31}} & h_{2}(\omega k)e^{\theta_{32}} & 0
\end{pmatrix} \label{v9Lv9Rdef}
\end{align}
and
\begin{align*}
h_{1}(k) = & \; r_{1,r}(k) + r_{1,a}(\tfrac{1}{\omega^{2}k})r_{2,r}(\omega k), \qquad h_{2}(k) =  r_{2,r}(k) + r_{2,a}(\tfrac{1}{\omega^{2}k})r_{1,r}(\omega k), \\
g_{1}(k) = & \; r_{1,r}(\tfrac{1}{\omega^{2}k})-r_{1,r}(\omega k) \big( r_{1,r}(k)+r_{1,a}(\tfrac{1}{\omega^{2}k})r_{2,r}(\omega k) \big), \\
g_{2}(k) = & \; r_{2,r}(\tfrac{1}{\omega^{2}k})-r_{2,r}(\omega k) \big( r_{2,r}(k)+r_{2,a}(\tfrac{1}{\omega^{2}k})r_{1,r}(\omega k) \big), \\
g(k) = & \; r_{1,r}(k)\big(r_{1,r}(\omega k)r_{2,a}(\tfrac{1}{\omega^{2}k})+r_{2,r}(k)\big) \\
& +r_{1,a}(\tfrac{1}{\omega^{2}k})r_{2,r}(\omega k)\big( r_{1,r}(\omega k)r_{2,a}(\tfrac{1}{\omega^{2}k})+r_{2,r}(k) \big) + r_{1,r}(\tfrac{1}{\omega^{2}k})r_{2,r}(\tfrac{1}{\omega^{2}k}).
\end{align*}
On $\Gamma_{2}^{(0)}$, we will use the factorization 
%In the main file we work with $(v_{3}^{(1)})^{-1}$, but it looks better to directly work with $v_{3}^{(1)}$
\begin{align}\label{v11v21v31def}
& v_7^{-1} = v_3^{(1)} v_2^{(1)} v_1^{(1)}, \qquad v_{2}^{(1)}  =  \begin{pmatrix}
1+r_{1}(k)r_{2}(k) & 0 & 0 \\
0 & \frac{1}{1+r_{1}(k)r_{2}(k)} & 0 \\
0 & 0 & 1
\end{pmatrix}  + v_{2,r}^{(1)}, 
	\\ \nonumber
& v_3^{(1)} 
= \begin{pmatrix}
1 & 0 & r_{1,a}(\frac{1}{\omega^{2} k})e^{-\theta_{31}} \\[0.05cm]
\hat{r}_{2,a}(k)e^{\theta_{21}} & 1 & -r_{2,a}(\frac{1}{\omega k})e^{-\theta_{32}} \\
0 & 0 & 1
\end{pmatrix},  \quad v_{1}^{(1)}  = \begin{pmatrix}
1 & \hat{r}_{1,a}(k)e^{-\theta_{21}} & 0 \\
0 & 1 & 0 \\
r_{2,a}(\frac{1}{\omega^{2}k})e^{\theta_{31}} & -r_{1,a}(\frac{1}{\omega k})e^{\theta_{32}} & 1
\end{pmatrix}.
\end{align}

On $\Gamma_{5}^{(0)}$, we will use the factorization
\begin{align}
& v_{7} = v_{4}^{(1)}v_{5}^{(1)}v_{6}^{(1)}, \qquad v_{5}^{(1)} = \begin{pmatrix}
\frac{1}{f(k)} & 0 & 0 \\
0 & 1+r_{1}(k)r_{2}(k) & 0 \\
0 & 0 & \frac{f(k)}{1+r_{1}(k)r_{2}(k)}
\end{pmatrix} + v_{5,r}^{(1)}, \nonumber \\
& v_{4}^{(1)} = \begin{pmatrix}
1 & a_{12,a}e^{-\theta_{21}} & a_{13,a}e^{-\theta_{31}} \\
0 & 1 & 0 \\
0 & a_{32,a}e^{\theta_{32}} & 1
\end{pmatrix}, \qquad v_{6}^{(1)} = \begin{pmatrix}
1 & 0 & 0 \\
a_{21,a}e^{\theta_{21}} & 1 & a_{23,a}e^{-\theta_{32}} \\
a_{31,a}e^{\theta_{31}} & 0 & 1
\end{pmatrix}, \label{v41v51v61def}
\end{align}
where
\begin{align*}
& a_{12} = -\hat{r}_{1}(k), & & a_{13} = -\frac{r_{1}(\frac{1}{\omega^{2}k})}{f(k)}, & & a_{23} = \frac{r_{2}(\frac{1}{\omega k})-r_{2}(k)r_{2}(\omega^{2}k)}{1+r_{1}(k)r_{2}(k)}, \\
& a_{21} = -\hat{r}_{2}(k), & & a_{31} = - \frac{r_{2}(\frac{1}{\omega^{2}k})}{f(k)}, & & a_{32} = \frac{r_{1}(\frac{1}{\omega k})-r_{1}(k)r_{1}(\omega^{2}k)}{1+r_{1}(k)r_{2}(k)}.
\end{align*}
and $a_{ij,a}, a_{ij,r}$ are the analytic continuation and the remainder of $a_{ij}$ from Lemma \ref{IIbis decompositionlemma}, e.g., $a_{13,a} = -r_{1,4,a}(\frac{1}{\omega^2 k})$ and $a_{23,a} = r_{2,3,a}(\frac{1}{\omega k})$.

Finally, on $\Gamma_{8}^{(0)}$, we will use the factorization
\begin{align}
& v_{7} = v_{7}^{(1)}v_{8}^{(1)}v_{9}^{(1)}, \qquad v_{8}^{(1)} = \begin{pmatrix}
\frac{1}{f(k)} & 0 & 0 \\
0 & \frac{f(k)}{f(\omega^{2}k)} & 0 \\
0 & 0 & f(\omega^{2}k)
\end{pmatrix} + v_{8,r}^{(1)}, \nonumber \\
& v_{7}^{(1)} = \begin{pmatrix}
1 & b_{12,a}e^{-\theta_{21}} & b_{13,a}e^{-\theta_{31}} \\
0 & 1 & b_{23,a}e^{-\theta_{32}} \\
0 & 0 & 1
\end{pmatrix}, \qquad v_{9}^{(1)} = \begin{pmatrix}
1 & 0 & 0 \\
b_{21,a}e^{\theta_{21}} & 1 & 0 \\
b_{31,a}e^{\theta_{31}} & b_{32,a}e^{\theta_{32}} & 1
\end{pmatrix}, \label{v71v81v91def}
\end{align}
where
\begin{align*}
& b_{12} = - \frac{r_{1}(k)- r_{1}(\frac{1}{\omega k})r_{1}(\frac{1}{\omega^{2}k})}{f(k)}, & & b_{13} = \frac{r_{2}(\omega^{2}k)}{f(\omega^{2}k)}, & & b_{23} = \frac{r_{2}(\frac{1}{\omega k})-r_{2}(k)r_{2}(\omega^{2}k)}{f(\omega^{2}k)}, \\
& b_{21} = -\frac{r_{2}(k) - r_{2}(\frac{1}{\omega k})r_{2}(\frac{1}{\omega^{2}k})}{f(k)}, & & b_{31} = \frac{r_{1}(\omega^{2}k)}{f(\omega^{2}k)}, & & b_{32} = \frac{r_{1}(\frac{1}{\omega k})-r_{1}(k)r_{1}(\omega^{2}k)}{f(\omega^{2}k)},
\end{align*}
and $b_{ij,a}, b_{ij,r}$ are the analytic continuation and the remainder of $b_{ij}$ from Lemma \ref{IIbis decompositionlemma}, e.g., $b_{12,a} = -r_{1,5,a}(k)$ and $b_{23,a} = r_{2,5,a}(\frac{1}{\omega k})$.
We do not write down the long expressions for $v_{2,r}^{(1)},v_{5,r}^{(1)}$ and $v_{8,r}^{(1)}$ which are similar to the expression for $v_{2,r}^{(1)}$ given in \cite[Section 8]{CLmain}; the only property of these matrices that is important for us is that 
\begin{align*}
\| v_{j,r}^{(1)} \|_{(L^{1}\cap L^{\infty})(\Gamma_{j}^{(0)})} = O(t^{-1}), \qquad \mbox{as } t \to \infty, \; j= 2,5,8,
\end{align*}
as a consequence of Lemma \ref{IIbis decompositionlemma}.

\begin{figure}
\begin{center}
\begin{tikzpicture}[master]
\node at (0,0) {};
%\draw[black,line width=0.65 mm,->-=0.25,->-=0.57,->-=0.71,->-=0.91] (0,0)--(30:7.5);
\draw[black,line width=0.55 mm] (0,0)--(30:6.5);
\draw[black,line width=0.55 mm,->-=0.35,->-=0.65,->-=0.82,->-=0.98] (0,0)--(90:6.5);
%\draw[black,line width=0.55 mm,->-=0.25,->-=0.57,->-=0.71,->-=0.91] (0,0)--(150:6.5);
\draw[black,line width=0.55 mm] (0,0)--(150:6.5);

%Boundary of sector S
\draw[dashed,black,line width=0.15 mm] (0,0)--(60:7.8);
\draw[dashed,black,line width=0.15 mm] (0,0)--(120:7.8);

\draw[black,line width=0.55 mm] ([shift=(30:3*1.5cm)]0,0) arc (30:150:3*1.5cm);
\draw[black,arrows={-Triangle[length=0.27cm,width=0.18cm]}]
($(70:3*1.5)$) --  ++(70-90:0.001);
\draw[black,arrows={-Triangle[length=0.27cm,width=0.18cm]}]
($(70:3.65)$) --  ++(70-90:0.001);
\draw[black,arrows={-Triangle[length=0.27cm,width=0.18cm]}]
($(70:5.8)$) --  ++(70-90:0.001);

%Arrow for contour 7
\draw[black,arrows={-Triangle[length=0.2cm,width=0.14cm]}]
($(117.7:3.25*1.5)$) --  ++(120+25:0.001);

%Arrow for contour 8
\draw[black,arrows={-Triangle[length=0.2cm,width=0.14cm]}]
($(120:3*1.5)$) --  ++(120+90:0.001);

%Arrow for contour 9
\draw[black,arrows={-Triangle[length=0.2cm,width=0.14cm]}]
($(118.5:2.75*1.5)$) --  ++(120+140:0.001);

%Contour through \omega
\draw[black,line width=0.4 mm] (120:2.75*1.5)--(120:3.3*1.5);
\draw[black,arrows={-Triangle[length=0.16cm,width=0.12cm]}]
($(120:2.92*1.5)$) --  ++(120:0.001);
\draw[black,arrows={-Triangle[length=0.16cm,width=0.12cm]}]
($(120:3.2*1.5)$) --  ++(120:0.001);

%Contour through -\omega^2
\draw[black,line width=0.55 mm,->-=0.3,->-=0.8] (60:2.45*1.5)--(60:3.85*1.5);

\node at (71:6.1) {\small $9_{R}$};
\node at (71:3.15*1.5) {\small $9_r$};
\node at (71:3.9) {\small $9_{L}$};

\node at (82.5:1.4*1.5) {\small $1''$};
\node at (85.5:2.7*1.5) {\small $1_r''$};
\node at (87:6.2) {\small $4'$};
\node at (86.5:3.45*1.5) {\small $4_{r}'$};

\node at (96.5:3.55*1.5) {\small $1$};
\node at (93:3.12*1.5) {\small $2$};
\node at (97:2.58*1.5) {\small $3$};
\node at (108:3.5*1.5) {\small $4$};
\node at (108:3.15*1.5) {\small $5$};
\node at (108:2.6*1.5) {\small $6$};
\node at (115:3.28*1.5) {\small $7$};
\node at (115.5:2.75*1.5) {\small $9$};

\node at (118:4.13*1.5) {\small $8$};
\draw[->-=1] (118.4:4*1.5)--(118.4:3.04*1.5);

\node at (124.8:3.75*1.5) {\small $1_s$};
\draw[->-=1] (124:3.6*1.5)--(120.8:3.14*1.5);
\node at (127:3.53*1.5) {\small $2_s$};
\draw[->-=1] (126:3.4*1.5)--(121:2.86*1.5);

\node at (57:3.5*1.5) {\small $3_s$};
\node at (56:2.8*1.5) {\small $4_s$};

%\node at (122:3.15*1.5) {\tiny $7$};
%\node at (122:2.9*1.5) {\tiny $9$};

%(111.471:3) to [out=111.471+45, in=120-90] (120:3.5)

\draw[black,line width=0.55 mm] (100:3*1.5) to [out=100+45, in=108-90] (108:3.3*1.5) to [out=108+90, in=115-45] (115:3*1.5) to [out=115+45, in=120-90] (120:3.3*1.5) to [out=120+90, in=140-45]  (125:3*1.5) to [out=125+45, in=132-90] (132:3.3*1.5) to [out=132+90, in=140-45] (140:3*1.5);

\draw[black,line width=0.55 mm] (100:3*1.5) to [out=100+45+90, in=108-90] (108:2.75*1.5) to [out=108+90, in=115-45-90] (115:3*1.5) to [out=115+45+90, in=120-90] (120:2.75*1.5) to [out=120+90, in=140-45-90]  (125:3*1.5) to [out=125+45+90, in=132-90] (132:2.75*1.5) to [out=132+90, in=140-45-90] (140:3*1.5);

\draw[black,line width=0.55 mm,-<-=0.14,->-=0.80] (90:3.65)--(100:3*1.5)--(90:5.8);

%\draw[black,line width=0.55 mm,->-=0.30,->-=0.80] (150:3.65)--(140:3*1.5)--(150:5.8);
\draw[black,line width=0.55 mm] (150:3.65)--(140:3*1.5)--(150:5.8);

\draw[black,line width=0.55 mm] ([shift=(30:3.65cm)]0,0) arc (30:90:3.65cm);
\draw[black,line width=0.55 mm] ([shift=(30:5.8cm)]0,0) arc (30:90:5.8cm);

\draw[black,arrows={-Triangle[length=0.27cm,width=0.18cm]}]
($(93:3*1.5)$) --  ++(93-90:0.001);
\draw[black,arrows={-Triangle[length=0.27cm,width=0.18cm]}]
($(110:3*1.5)$) --  ++(110+90:0.001);
\draw[black,arrows={-Triangle[length=0.27cm,width=0.18cm]}]
($(110:3.3*1.5)$) --  ++(110+90:0.001);
\draw[black,arrows={-Triangle[length=0.27cm,width=0.18cm]}]
($(110:2.75*1.5)$) --  ++(110+90:0.001);

% The angles have been slightly modified to make the figure more visible
\draw[blue,fill] (125:3*1.5) circle (0.1cm);
\draw[green,fill] (-125+240:3*1.5) circle (0.1cm);
\draw[red,fill] (100:3*1.5) circle (0.1cm);
\draw[red,fill] (140:3*1.5) circle (0.1cm);

\end{tikzpicture}
\end{center}
\begin{figuretext}
\label{Gamma1fig}
The contour $\Gamma^{(1)}$ (solid) and the boundary of $\mathsf{S}$ (dashed). 
From right to left, the dots are $\omega k_{4}$ (red), $\omega^{2}k_{2}$ (green), $k_{1}$ (blue), and $\omega k_{3}$ (red). 
%For visual convenience, these dots have been slightly moved from their exact locations.
\end{figuretext}
\end{figure}

Let $\Gamma^{(1)}$ be the contour shown in Figure \ref{Gamma1fig}, and define $G^{(1)}$ for $k \in \mathsf{S}$ by
\begin{align}\label{II Gp1pdef}
& G^{(1)} = \begin{cases} 
v_{9_L}^{(1)}, & k \mbox{ on the $-$ side of }\Gamma_{9_r}^{(1)}, 
	 \\
(v_{9_R}^{(1)})^{-1}, & k \mbox{ on the $+$ side of }\Gamma_{9_r}^{(1)}, 
	 \\
v_3^{(1)}, & k \mbox{ on the $-$ side of }\Gamma_2^{(1)}, 
	 \\
(v_1^{(1)})^{-1}, & k \mbox{ on the $+$ side of }\Gamma_2^{(1)}, 
\end{cases}
&& 
G^{(1)} = \begin{cases} 
v_4^{(1)}, & k \mbox{ on the $-$ side of }\Gamma_5^{(1)}, 
	 \\
(v_6^{(1)})^{-1}, & k \mbox{ on the $+$ side of }\Gamma_5^{(1)}, 
	 \\
v_7^{(1)}, & k \mbox{ on the $-$ side of }\Gamma_8^{(1)}, 
	 \\
(v_9^{(1)})^{-1}, & k \mbox{ on the $+$ side of }\Gamma_8^{(1)}, 
	\\
I, & \mbox{otherwise}.
\end{cases}
\end{align}
We now appeal to the $\mathcal{A}$- and $\mathcal{B}$-symmetries to extend the definition of $G^{(1)}$ to the whole complex plane:
\begin{align}\label{symmetry of G1}
G^{(1)}(x,t, k) = \mathcal{A} G^{(1)}(x,t,\omega k)\mathcal{A}^{-1}
 = \mathcal{B} G^{(1)}(x,t, k^{-1}) \mathcal{B}, \qquad k \in \mathbb{C}\setminus \Gamma^{(1)}.
\end{align}
Define the sectionally meromorphic function $n^{(1)}$ by
\begin{align}\label{Sector II first transfo}
n^{(1)}(x,t,k) = n(x,t,k)G^{(1)}(x,t,k),\qquad k \in \C \setminus (\Gamma^{(1)}\cup \hat{\mathsf{Z}}).
\end{align}
The functions $G^{(1)}$ and $n^{(1)}$ are analytic on $\C \setminus \Gamma^{(1)}$. Indeed, let us look for example at the region on the $-$ side of $\Gamma_{8}^{(1)}$ that is inside the lens; in this region $G^{(1)} = v_{7}^{(1)}$ is given in terms of $r_{1,5,a}(k)$, $r_{2,5,a}(\frac{1}{\omega k})$, and $r_{2,4,a}(\omega^2 k)$, and it follows from Lemma \ref{IIbis decompositionlemma} and the definitions of $U_1^5$, $U_2^5$, and $U_2^4$ that these functions are all analytic in this region.

The following lemma is a direct consequence of Figure \ref{II fig: Re Phi 21 31 and 32 for zeta=0.7}. In this lemma, the disks $D_\epsilon(\omega^j)$ have been excluded; this is because $v_{7}^{(1)},v_{8}^{(1)},v_{9}^{(1)}$ (and hence also $n^{(1)}(x,t,k)$) are singular at $k = \omega$, due to the fact that $f(\omega) = f(1) = 0$.

\begin{lemma}\label{lemma:G1p1p}
For any $\epsilon>0$, $G^{(1)}(x,t,k)$ and $G^{(1)}(x,t,k)^{-1}$ are uniformly bounded for $k \in \mathbb{C}\setminus (\Gamma^{(1)} \cup \cup_{j=0}^{2} D_\epsilon(\omega^j))$, $t\geq 1$, and $\zeta \in \mathcal{I}$. Furthermore, $G^{(1)}(x,t,k)=I$ for all large enough $|k|$.
\end{lemma}

Let $\Gamma^{(1)}_\star$ be the set of self-intersection points of $\Gamma^{(1)}$. The jump matrix $v^{(1)}$ for $n^{(1)}$ is given on $\Gamma^{(1)} \setminus \Gamma^{(1)}_\star$ as follows: For $j =  9_L, 9_r, 9_R, 1, \dots, 9$, the matrix $v_j^{(1)}$ is given by (\ref{v9Lv9Rdef})--(\ref{v71v81v91def}); for $j =1'', 1_r'', 4_r', 4'$ it is given by
\begin{align}\label{v1oniRplusdef}
& v_{1''}^{(1)} = v_{1''}, \quad 
v_{1_r''}^{(1)} = (v_{9_L}^{(1)})^{-1}v_{1''} v_3^{(1)}, \quad 
v_{4_r'}^{(1)} = v_{9_R}^{(1)} v_{4'} (v_1^{(1)})^{-1}, \quad 
v_{4'}^{(1)} = v_{4'};
\end{align}
for $j = 1_s$ it is given by
\begin{align}\nonumber
 v_{1_s}^{(1)} = &\; G_-^{(1)}(k)^{-1}G_+^{(1)}(k)
= v_7^{(1)}(k)^{-1} \mathcal{A} \mathcal{B} v_9^{(1)}(\tfrac{1}{\omega k})^{-1} \mathcal{B} \mathcal{A}^{-1} 
	\\ \label{v1sexpression}
=&\; \begin{pmatrix}
 1 & -(b_{12,a}(k)+b_{32,a}(\frac{1}{\omega k}))e^{-t\Phi_{21}} & (v_{1_s}^{(1)})_{13} \\
 0 & 1 & -(b_{21,a}(\frac{1}{\omega k})+b_{23,a}(k))e^{-t\Phi_{32}} \\
 0 & 0 & 1 \\
\end{pmatrix},
	\\ \nonumber
 (v_{1_s}^{(1)})_{13} := &\;  \big(b_{12,a}(k) (b_{21,a}(\tfrac{1}{\omega k})+b_{23,a}(k)) +b_{21,a}(\tfrac{1}{\omega k}) b_{32,a}(\tfrac{1}{\omega k}) -b_{13,a}(k) -b_{31,a}(\tfrac{1}{\omega k})\big) e^{-t\Phi_{31}},
\end{align}
and similar expressions are valid for $j = 2_s, 3_s, 4_s$. Using the symmetries in (\ref{vjsymm}) we extend $v^{(1)}$ to all of $\Gamma^{(1)}\setminus \Gamma^{(1)}_\star$.

Since $n^{(1)} = n$ near all the points of $\hat{\mathsf{Z}}$, the residue conditions \eqref{nresiduesk0}--\eqref{nresiduesk0real} are not affected by the transformation $n \to n^{(1)}$.
The function $n^{(1)}$ therefore satisfies the following RH problem for $j = 1$.

\begin{RHproblem}[RH problem for $n^{(j)}$]\label{RHnj}
Find $n^{(j)}(x,t,k)$ with the following properties:
\begin{enumerate}[$(a)$]
\item\label{RHnjitema} $n^{(j)}(x,t,\cdot) : \C \setminus (\Gamma^{(j)}\cup \hat{\mathsf{Z}}) \to \mathbb{C}^{1 \times 3}$ is analytic.

\item\label{RHnjitemb} On $\Gamma^{(j)} \setminus \Gamma^{(j)}_\star$, the boundary values of $n^{(j)}$ exist, are continuous, and satisfy $n^{(j)}_+ = n^{(j)}_-v^{(j)}$.

\item\label{RHnjitemc} $n^{(j)}(x,t,k) = O(1)$ as $k \to k_{\star} \in \Gamma^{(j)}_\star \setminus \{1, \omega, \omega^2\}$.

\item\label{RHnjitemd} $n^{(j)}$ obeys the symmetries $n^{(j)}(x,t, k) = n^{(j)}(x,t,\omega k)\mathcal{A}^{-1} = n^{(j)}(x,t, k^{-1}) \mathcal{B}$ for $k \in \C \setminus \Gamma^{(j)}$.

\item\label{RHnjiteme} $n^{(j)}(x,t,k) = (1,1,1) + O(k^{-1})$ as $k \to \infty$.

\item\label{RHnjitemf} At each point of $\hat{\mathsf{Z}}$, one entry of $n^{(j)}$ has (at most) a simple pole while two entries are analytic. Moreover, 
$n^{(j)}$ satisfies the residue conditions \eqref{nresiduesk0}--\eqref{nresiduesk0real} with $n$ replaced by $n^{(j)}$.

\end{enumerate}
\end{RHproblem}

Note that $n^{(1)}$ is singular at the points $1, \omega, \omega^2$  in general.
The behavior of $n^{(1)}$ near these singularities is as follows: 
As $k \to \omega$ from the right side of $\Gamma_8^{(1)}$ (i.e. $k \in D_{4}$ and $|k|>1$),
\begin{align*}
n^{(1)}(x,t,k) = \begin{pmatrix}
O(1) & O(1) & O(1) 
\end{pmatrix}v_7^{(1)}  = \begin{pmatrix}
O(1) & O(1) & O(1) 
\end{pmatrix} \begin{pmatrix}
1 & O(b_{12,a}) & O(\frac{1}{f(\omega^{2}k)}) \\
0 & 1 & O(\frac{k-\omega}{f(\omega^{2}k)}) \\
0 & 0 & 1
\end{pmatrix},
\end{align*}
and as $k \to \omega$ from the left side of $\Gamma_8^{(1)}$ (i.e. $k \in D_{1}$ and $|k|<1$),
\begin{align*}
n^{(1)}(x,t,k) & = \begin{pmatrix}
O(1) \; O(1) \; O(1) 
\end{pmatrix}(v_9^{(1)})^{-1} 
	\\
& = \begin{pmatrix}
O(1) \; O(1) \; O(1)
\end{pmatrix} \begin{pmatrix}
1 & 0 & 0 \\
O(\frac{k-\omega}{f(k)}) & 1 & 0 \\
O(b_{21,a}b_{32,a}-b_{31,a}) & O(b_{32,a}) & 1
\end{pmatrix}.
\end{align*}
We know from \cite[Lemma 2.13]{CLmain} that $f(\pm 1) = f(\pm \omega) = 0$. Since $f \geq 0$ on $\partial \D$, these zeros must be at least double zeros.
One could try to make the solution of this singular RH problem unique by specifying the pole structure at the points $1, \omega, \omega^2$ in further detail, but since the singularities at $1, \omega, \omega^2$ will be removed by the global parametrix below, we do not need to do this.

\section{The $n^{(1)} \to n^{(2)}$ transformation}\label{n1ton2sec}

The jump matrix $v^{(1)}$ admits the following factorizations on $\Gamma^{(1)}\cap \mathsf{S}$ (the subscripts $u$ and $d$ indicate that the corresponding matrix will be deformed up or down):
\begin{align*}
& v_1^{(1)} = v_1^{(2)} v_{1u}, & &
v_3^{(1)} = v_{3d}v_3^{(2)}, & &
v_4^{(1)} = v_{4u}v_4^{(2)} = v_{7u}v_7^{(2)}, 
	\\
& v_6^{(1)} = v_6^{(2)}v_{6d} = v_8^{(2)} v_{8d}, & &
v_7^{(1)} = v_{9u} v_9^{(2)}, & &
v_9^{(1)} = v_{11}^{(2)} v_{11d},
\end{align*}
where
\begin{align}\nonumber
& v_1^{(2)} := \begin{pmatrix}
1 & 0 & 0 \\
0 & 1 & 0 \\
r_{2,a}(\frac{1}{\omega^{2}k})e^{\theta_{31}} & 0 & 1
\end{pmatrix}, & & v_{1u} := \begin{pmatrix}
1 & \hat{r}_{1,a} e^{-\theta_{21}} & 0 \\
0 & 1 & 0 \\
0 & (-r_{1,a}(\frac{1}{\omega k})-\hat{r}_{1,a}(k) r_{2,a}(\frac{1}{\omega^{2}k})) e^{\theta_{32}} & 1
\end{pmatrix}, 
	\\\nonumber
& v_3^{(2)} := \begin{pmatrix}
1 & 0 & r_{1,a}(\frac{1}{\omega^{2}k})e^{-\theta_{31}} \\
0 & 1 & 0  \\
0 & 0 & 1
\end{pmatrix}, & &
v_{3d} := \hspace{-0.15cm} \begin{pmatrix}
\hspace{-0.05cm}1 & \hspace{-0.15cm} 0 & \hspace{-0.15cm}0 \\
\hspace{-0.05cm}  \hat{r}_{2,a}e^{\theta_{21}} & \hspace{-0.15cm} 1 & \hspace{-0.15cm}(-r_{2,a}(\frac{1}{\omega k})-r_{1,a}(\frac{1}{\omega^{2}k})\hat{r}_{2,a}(k))e^{-\theta_{32}} \\
\hspace{-0.05cm}0 & \hspace{-0.15cm} 0 & \hspace{-0.15cm}1
\end{pmatrix}\hspace{-0.1cm},  
	\\\nonumber
& v_4^{(2)} := \begin{pmatrix}
1 & 0 & a_{13,a}e^{-\theta_{31}} \\
0 & 1 & 0 \\
0 & 0 & 1
\end{pmatrix}, & & v_{4u} := \begin{pmatrix}
1 & a_{12,a}e^{-\theta_{21}} & 0 \\
0 & 1 & 0 \\
0 & a_{32,a}e^{\theta_{32}} & 1
\end{pmatrix}, 
	\\\nonumber
&v_6^{(2)} := \begin{pmatrix}
1 & 0 & 0 \\
0 & 1 & 0 \\
a_{31,a}e^{\theta_{31}} & 0 & 1
\end{pmatrix}, & &
 v_{6d} := \begin{pmatrix}
1 & 0 & 0 \\
a_{21,a}e^{\theta_{21}} & 1 & a_{23,a}e^{-\theta_{32}} \\
0 & 0 & 1
\end{pmatrix},  
	\\\nonumber
& v_7^{(2)} := \begin{pmatrix}
1 & 0 & 0 \\
0 & 1 & 0 \\
0 & a_{32,a}e^{\theta_{32}} & 1
\end{pmatrix}, & & v_{7u} := \begin{pmatrix}
1 & (a_{12,a}-a_{13,a}a_{32,a})e^{-\theta_{21}} & a_{13,a}e^{-\theta_{31}} \\
0 & 1 & 0 \\
0 & 0 & 1
\end{pmatrix}, 
	\\\nonumber
& v_8^{(2)}  := \begin{pmatrix}
1 & 0 & 0 \\
0 & 1 & a_{23,a}e^{-\theta_{32}} \\
0 & 0 & 1
\end{pmatrix}, & &
v_{8d} := \begin{pmatrix}
1 & 0 & 0 \\
(a_{21,a}-a_{23,a}a_{31,a})e^{\theta_{21}} & 1 & 0 \\
a_{31,a}e^{\theta_{31}} & 0 & 1
\end{pmatrix},
	\\\nonumber
& v_9^{(2)} := \begin{pmatrix}
1 & 0 & 0 \\
0 & 1 & b_{23,a}e^{-\theta_{32}} \\
0 & 0 & 1
\end{pmatrix}, & & v_{9u} := \begin{pmatrix}
1 & b_{12,a}e^{-\theta_{21}} & (b_{13,a}-b_{12,a}b_{23,a})e^{-\theta_{31}} \\
0 & 1 & 0 \\
0 & 0 & 1
\end{pmatrix}, 
	\\ \label{IVv2def}
& v_{11}^{(2)} := \begin{pmatrix}
1 & 0 & 0 \\
0 & 1 & 0 \\
0 & b_{32,a}e^{\theta_{32}} & 1
\end{pmatrix}, & &
v_{11d} := \begin{pmatrix}
1 & 0 & 0 \\
b_{21,a}e^{\theta_{21}} & 1 & 0 \\
(b_{31,a}-b_{21,a}b_{32,a})e^{\theta_{31}} & 0 & 1
\end{pmatrix}.
\end{align}

\begin{figure}
\begin{center}
\begin{tikzpicture}[master, scale=1.1]
\node at (0,0) {};
\draw[black,line width=0.55 mm] (0,0)--(30:6.5);
\draw[black,line width=0.55 mm,->-=0.35,->-=0.65,->-=0.82,->-=0.98] (0,0)--(90:6.5);
\draw[black,line width=0.55 mm] (0,0)--(150:6.5);

%Boundary of sector S
\draw[dashed,black,line width=0.15 mm] (0,0)--(60:7.8);
\draw[dashed,black,line width=0.15 mm] (0,0)--(120:7.8);

\draw[black,line width=0.55 mm] ([shift=(30:3*1.5cm)]0,0) arc (30:150:3*1.5cm);
\draw[black,arrows={-Triangle[length=0.27cm,width=0.18cm]}]
($(70:3*1.5)$) --  ++(70-90:0.001);
\draw[black,arrows={-Triangle[length=0.27cm,width=0.18cm]}]
($(70:3.65)$) --  ++(70-90:0.001);
\draw[black,arrows={-Triangle[length=0.27cm,width=0.18cm]}]
($(70:5.8)$) --  ++(70-90:0.001);
%\draw[black,arrows={-Triangle[length=0.27cm,width=0.18cm]}]
%($(122:3*1.5)$) --  ++(120+90:0.001);
%\draw[black,arrows={-Triangle[length=0.24cm,width=0.16cm]}]
%($(122:3.3*1.5)$) --  ++(120+90:0.001);
%\draw[black,arrows={-Triangle[length=0.24cm,width=0.16cm]}]
%($(122:2.75*1.5)$) --  ++(120+90:0.001);

\draw[black,line width=0.55 mm] ([shift=(90:3.65)]0,0) arc (90:150:3.65);
\draw[black,line width=0.55 mm] ([shift=(90:5.8)]0,0) arc (90:150:5.8);

%Contours 6_s and 5_s
\draw[black,line width=0.55 mm,-<-=0.08,->-=0.80] (100:3.65)--(100:5.8);

\draw[black,line width=0.55 mm,->-=0.9] (108:3.65)--(108:4.1);
\draw[black,line width=0.55 mm,->-=0.75] (108:4.9)--(108:5.8);

%Contours 8_s and 7_s
\draw[black,line width=0.55 mm,-<-=0.08,->-=0.80] (115:3.65)--(115:5.8);

\draw[black,line width=0.55 mm,->-=0.9] (120:3.65)--(120:4.1);
\draw[black,line width=0.55 mm,->-=0.75] (120:4.9)--(120:5.8);
\draw[black,line width=0.55 mm] (125:3.65)--(125:5.8);
\draw[black,line width=0.55 mm] (132:3.65)--(132:4.1);
\draw[black,line width=0.55 mm] (132:4.95)--(132:5.8);
\draw[black,line width=0.55 mm] (140:3.65)--(140:5.8);

%Curved lenses outside unit circle
\draw[black,line width=0.55 mm] (100:3*1.5) to [out=100+45, in=108-90] (108:3.3*1.5) to [out=108+90, in=115-45] (115:3*1.5) to [out=115+45, in=120-90] (120:3.3*1.5) to [out=120+90, in=140-45]  (125:3*1.5) to [out=125+45, in=132-90] (132:3.3*1.5) to [out=132+90, in=140-45] (140:3*1.5);

%Curved lenses inside unit circle
\draw[black,line width=0.55 mm] (100:3*1.5) to [out=100+45+90, in=108-90] (108:2.75*1.5) to [out=108+90, in=115-45-90] (115:3*1.5) to [out=115+45+90, in=120-90] (120:2.75*1.5) to [out=120+90, in=140-45-90]  (125:3*1.5) to [out=125+45+90, in=132-90] (132:2.75*1.5) to [out=132+90, in=140-45-90] (140:3*1.5);

%Contours 1 and 3
\draw[black,line width=0.55 mm,-<-=0.15,->-=0.80] (90:3.65)--(100:3*1.5)--(90:5.8);

\draw[black,line width=0.55 mm] (150:3.65)--(140:3*1.5)--(150:5.8);

\draw[black,line width=0.55 mm] ([shift=(30:3.65cm)]0,0) arc (30:90:3.65cm);
\draw[black,line width=0.55 mm] ([shift=(30:5.8cm)]0,0) arc (30:90:5.8cm);

\draw[black,arrows={-Triangle[length=0.27cm,width=0.18cm]}]
($(93:3*1.5)$) --  ++(93-90:0.001);
\draw[black,arrows={-Triangle[length=0.27cm,width=0.18cm]}]
($(110:3*1.5)$) --  ++(110+90:0.001);
\draw[black,arrows={-Triangle[length=0.27cm,width=0.18cm]}]
($(110:3.3*1.5)$) --  ++(110+90:0.001);
\draw[black,arrows={-Triangle[length=0.27cm,width=0.18cm]}]
($(110:2.75*1.5)$) --  ++(110+90:0.001);

\node at (95.5:3.58*1.5) {\scriptsize $1$};
\node at (93:3.12*1.5) {\scriptsize $2$};
\node at (93.5:2.77*1.5) {\scriptsize $3$};
\node at (108:3.12*1.5) {\scriptsize $5$};

\node at (118.5:4.15*1.5) {\tiny $10$};
\draw[->-=1] (118.5:4.05*1.5)--(118:3.02*1.5);
%\node at (118.05:3.10*1.5) {\tiny $10$};

%New notation: 8_{7L},R -> 4, 8_{7R},R -> 6, 8_{7L},L -> 7, 8_{7R},L -> 8, \tilde{7}_{L} -> 9, \tilde{7} -> 10, \tilde{7}_{R} -> 11

%\node at (105:4.3*1.5) {\small $4$};
%\draw[->-=1] (105:4.2*1.5)--(105:3.3*1.5);
\node at (104:3.33*1.5) {\scriptsize $4$};

%\node at (103:2*1.5) {\small $6$};
%\draw[->-=1] (103:2.1*1.5)--(103:2.8*1.5);
\node at (104:4.07) {\scriptsize $6$};

%\node at (112:4.75*1.5) {\small $7$};
%\draw[->-=1] (112:4.65*1.5)--(111:3.25*1.5);
\node at (112:4.94) {\scriptsize $7$};

%\node at (112:1.55*1.5) {\scriptsize $8$};
%\draw[->-=1] (112:1.65*1.5)--(110:2.77*1.5);
\node at (110.5:4.33) {\scriptsize $8$};

%\node at (115.7:4.3*1.5) {\scriptsize $9$};
%\draw[->-=1] (116:4.16*1.5)--(117.5:3.23*1.5);
\node at (116.7:4.96) {\scriptsize $9$};

\node at (113.8:1.57*1.5) {\scriptsize $11$};
\draw[->-=1] (114:1.7*1.5)--(117.5:2.78*1.5);
%\node at (117.8:4) {\tiny $11$};

% The angles have been slightly modified to make the figure more visible
\draw[blue,fill] (125:3*1.5) circle (0.1cm);
\draw[green,fill] (-125+240:3*1.5) circle (0.1cm);
\draw[red,fill] (100:3*1.5) circle (0.1cm);
\draw[red,fill] (140:3*1.5) circle (0.1cm);

% Contour through omega
\draw[black,line width=0.4 mm] (120:2.75*1.5)--(120:3.3*1.5);
\draw[black,arrows={-Triangle[length=0.16cm,width=0.12cm]}]
($(120:2.92*1.5)$) --  ++(120:0.001);
\draw[black,arrows={-Triangle[length=0.16cm,width=0.12cm]}]
($(120:3.2*1.5)$) --  ++(120:0.001);

%Contour through -\omega^2
\draw[black,line width=0.55 mm,->-=0.3,->-=0.8] (60:2.45*1.5)--(60:3.85*1.5);

%Arrow for contour 7
\draw[black,arrows={-Triangle[length=0.2cm,width=0.14cm]}]
($(117.7:3.25*1.5)$) --  ++(120+25:0.001);

%Arrow for contour 8
\draw[black,arrows={-Triangle[length=0.2cm,width=0.14cm]}]
($(120:3*1.5)$) --  ++(120+90:0.001);

%Arrow for contour 9
\draw[black,arrows={-Triangle[length=0.2cm,width=0.14cm]}]
($(118.5:2.75*1.5)$) --  ++(120+140:0.001);

%\node at (71:6.05) {\small $9_{R}$};
\node at (71:3.15*1.5) {\scriptsize $9_r$};
%\node at (71:3.9) {\small $9_{L}$};

\node at (127.8:3.7*1.5) {\scriptsize $1_s$};
\draw[->-=1] (126.5:3.63*1.5)--(120.8:3.14*1.5);
\node at (127:3.4*1.5) {\scriptsize $2_s$};
\draw[->-=1] (126:3.3*1.5)--(121:2.86*1.5);

\node at (57.3:3.5*1.5) {\scriptsize $3_s$};
\node at (56.5:2.8*1.5) {\scriptsize $4_s$};

\node at (102:5.3) {\scriptsize $5_{s}$};
\node at (97.2:3.87) {\scriptsize $6_{s}$};
\node at (112.9:5.25) {\scriptsize $7_{s}$};
\node at (112.4:3.9) {\scriptsize $8_{s}$};

\node at (85.7:2.7*1.5) {\small $1_r''$};
\node at (86.7:3.45*1.5) {\small $4_{r}'$};

\end{tikzpicture}

\end{center}
\begin{figuretext}
\label{II Gammap2p}The contour $\Gamma^{(2)}$ (solid) and the boundary of $\mathsf{S}$ (dashed). From right to left, the dots are $\omega k_{4}$ (red), $\omega^{2}k_{2}$ (green), $k_{1}$ (blue), and $\omega k_{3}$ (red). 
\end{figuretext}
\end{figure}

Let $\Gamma^{(2)}$ be the contour shown in Figure \ref{II Gammap2p}, and let $\Gamma_{j}^{(2)}$ be the subcontour of $\Gamma^{(2)}$ labeled by $j$ in Figure \ref{II Gammap2p}. We emphasize that $\Gamma_{10}^{(2)} := \{e^{i\theta} \,|\, \theta \in (\arg(\omega^{2}k_{2}),\tfrac{2\pi}{3})\} $ ends at $\omega$.
Define the sectionally meromorphic function $n^{(2)}$ by
\begin{align}\label{Sector II second transfo}
n^{(2)}(x,t,k) = n^{(1)}(x,t,k)G^{(2)}(x,t,k), \qquad k \in \C \setminus (\Gamma^{(2)} \cup \hat{\mathsf{Z}}),
\end{align}
where $G^{(2)}$ is defined for $k \in \mathsf{S}$ by
\begin{align}\label{II Gp2pdef}
G^{(2)}(x,t,k) = \begin{cases} 
v_{1u}^{-1}, & k \mbox{ above }\Gamma_1^{(2)}, \\
v_{3d}, & k \mbox{ below }\Gamma_3^{(2)}, \\
v_{4u}, & k \mbox{ above }\Gamma_4^{(2)}, \\
v_{6d}^{-1}, & k \mbox{ below }\Gamma_6^{(2)}, 
\end{cases}
\qquad
G^{(2)}(x,t,k) = \begin{cases} 
v_{7u}, & k \mbox{ above }\Gamma_7^{(2)}, \\
v_{8d}^{-1}, & k \mbox{ below }\Gamma_8^{(2)}, \\
v_{9u}, & k \mbox{ above }\Gamma_9^{(2)}, \\
v_{11d}^{-1}, & k \mbox{ below }\Gamma_{11}^{(2)}, \\
I, & \mbox{otherwise},
\end{cases}
\end{align}
and $G^{(2)}$ is extended to all of $\C \setminus \Gamma^{(2)}$ using the $\mathcal{A}$- and $\mathcal{B}$-symmetries (as in \eqref{symmetry of G1}). Using Lemma \ref{IIbis decompositionlemma} and Figure \ref{II fig: Re Phi 21 31 and 32 for zeta=0.7}, we infer that the following holds.

\begin{lemma}
For any $\epsilon>0$, $G^{(2)}(x,t,k)$ and $G^{(2)}(x,t,k)^{-1}$ are uniformly bounded for $k \in \mathbb{C}\setminus (\Gamma^{(2)} \cup \cup_{j=0}^{2} D_\epsilon(\omega^j))$, $t>0$, and $\zeta \in \mathcal{I}$. Furthermore, $G^{(2)}(x,t,k)=I$ whenever $|k|$ is large enough.
\end{lemma}

%The jumps of $n^{(2)}$ on the four parts of $\Gamma^{(2)}\setminus \partial \mathbb{D}$ near $\omega k_{4}$ (see Figure \ref{II Gammap2p}) are given by
%\begin{subequations}\label{jumps vp2p near omega k4}
%\begin{align}
%& v_1^{(2)} = \begin{pmatrix}
%1 & 0 & 0 \\
%0 & 1 & 0 \\
%c_{31,a}e^{\theta_{31}} & 0 & 1
%\end{pmatrix}, & & 
%v_4^{(2)} = \begin{pmatrix}
%1 & 0 & a_{13,a}e^{-\theta_{31}} \\
%0 & 1 & 0 \\
%0 & 0 & 1
%\end{pmatrix}, 
%	\\
%& v_3^{(2)} = \begin{pmatrix}
%1 & 0 & c_{13,a}e^{-\theta_{31}} \\
%0 & 1 & 0  \\
%0 & 0 & 1
%\end{pmatrix}, & & 
%v_6^{(2)} = \begin{pmatrix}
%1 & 0 & 0 \\
%0 & 1 & 0 \\
%a_{31,a}e^{\theta_{31}} & 0 & 1
%\end{pmatrix},
%\end{align}
%\end{subequations}
%and the jumps on the four parts of $\Gamma^{(2)}\setminus \partial \mathbb{D}$ near $\omega^{2} k_{2}$ (see again Figure \ref{II Gammap2p}) are given by
%\begin{subequations}\label{jumps vp2p near omega2 k2}
%\begin{align}
%& v_7^{(2)} = \begin{pmatrix}
%1 & 0 & 0 \\
%0 & 1 & 0 \\
%0 & a_{32,a}e^{\theta_{32}} & 1
%\end{pmatrix}, & &
%v_9^{(2)} = \begin{pmatrix}
%1 & 0 & 0 \\
%0 & 1 & b_{23,a}e^{-\theta_{32}} \\
%0 & 0 & 1
%\end{pmatrix}, 
%	\\
%& v_8^{(2)} = \begin{pmatrix}
%1 & 0 & 0 \\
%0 & 1 & a_{23,a}e^{-\theta_{32}} \\
%0 & 0 & 1
%\end{pmatrix}, & & 
%v_{11}^{(2)} = \begin{pmatrix}
%1 & 0 & 0 \\
%0 & 1 & 0 \\
%0 & b_{32,a}e^{\theta_{32}} & 1
%\end{pmatrix}.
%\end{align}
%\end{subequations}

The jumps $v_j^{(2)}$ of $n^{(2)}$ are given for $j = 1, \dots, 11$ by (\ref{IVv2def}) and
\begin{align}\label{II jumps vj diagonal p2p}
& v_2^{(2)} := v_2^{(1)}, \quad v_5^{(2)} := v_5^{(1)}, \quad v_{10}^{(2)} := v_8^{(1)}.
\end{align}
All other jumps on $\Gamma^{(2)}\cap \mathsf{S}$ are small as $t \to \infty$. Indeed, this follows from the signature tables in Figure \ref{II fig: Re Phi 21 31 and 32 for zeta=0.7} together with the following observations. On $\Gamma_{9_r}^{(2)}$, $v^{(2)} - I$ is small because it involves small remainders.
On $\Gamma_{4_r'}^{(2)}$ and $\Gamma_{1_r''}^{(2)}$, $v^{(2)} - I$ is small as a consequence of the following lemma which follows in the same way as \cite[Lemma 8.5]{CLmain}.

\begin{lemma}\label{vsmallnearilemmaIV}
The $L^\infty$-norm of $v^{(2)} - I$ on $\Gamma_{4_r'}^{(2)} \cup \Gamma_{1_r''}^{(2)}$ is $O(t^{-N-1})$ as $t \to \infty$ uniformly for $\zeta \in \mathcal{I}$.
\end{lemma}

We finally consider the jumps $v_{j}^{(2)} = v_{j}^{(1)}$, $j = 1_s, \dots, 4_s$ and
$$v_{5_s}^{(2)} =  v_{1u} v_{4u}, \quad
v_{6_s}^{(2)} =  v_{6d} v_{3d},\quad
v_{7_s}^{(2)} =  v_{7u}^{-1} v_{9u},\quad
v_{8_s}^{(2)} =  v_{11d} v_{8d}^{-1}.$$
If the initial data $u_0,v_0$ have compact support so that all the spectral functions have analytic continuations and we can choose $b_{ij,a} = b_{ij}$, then a straightforward calculation shows that these jumps are absent (i.e., $v_j^{(2)} = I$, $j = 1_s, \dots, 8_s$) as a consequence of the relation (\ref{r1r2 relation on the unit circle}) and the definition (\ref{def of f}) of $f(k)$. In general, it seems difficult to construct analytic approximations which preserve the nonlinear relation (\ref{r1r2 relation on the unit circle}). Therefore, the jump matrices 
$v_j^{(2)}$, $j = 1_s, \dots, 8_s$, will generally be nontrivial. However, as the next lemma demonstrates, it is not difficult to choose the analytic approximations so that these jumps are uniformly small for large $t$. 

\begin{lemma}\label{symmetryjumpslemma}
It is possible to choose the analytic approximations so that the $L^\infty$-norm of $v_j^{(2)} - I$ on $\Gamma_j^{(2)}$, $j = 1_s, \dots, 8_s$, is $O(t^{-N})$ as $t \to \infty$ uniformly for $\zeta \in \mathcal{I}$.
\end{lemma}
\begin{proof}
Let us consider $v_{1_s}^{(2)}$; the other jumps are handled similarly. We see from (\ref{v1sexpression}) that $v_{1_s}^{(2)} = v_{1_s}^{(1)}$ has ones along the diagonal and that each of its off-diagonal entries is suppressed by an exponential of the form $e^{-t |\re \Phi_{ij}|}$ where $|\re \Phi_{ij}(\zeta, k)| \geq c |k - \omega|$ on $\Gamma_{1_s}^{(2)}$. Thus, by Lemma \ref{IIbis decompositionlemma}, $v_{1_s}^{(2)} - I$ is small as $t \to \infty$ everywhere on $\Gamma_{1_s}^{(2)}$ except possibly near $\omega$. On the other hand, we know that the matrix $v_{1_s}^{(2)} - I$ vanishes identically on $\Gamma_{1_s}^{(2)}$ if all the spectral functions have analytic continuations thanks to (\ref{r1r2 relation on the unit circle}).
This means that if we choose all the analytic approximations such that they agree with the functions they are approximating to sufficiently high order at each of the sixth roots of unity $\kappa_j$ (in the sense of (\ref{b13anearomega}) below), then we can ensure that
$$\| v_{1_s}^{(2)} - I \|_{L^\infty(\Gamma_{1_s}^{(2)})} \leq C \sup_{k \in \Gamma_{1_s}^{(2)}} |k-\omega|^N e^{-\frac{c}{2}t|k-\omega|} \leq C t^{-N},$$
uniformly for $\zeta \in \mathcal{I}$, and the lemma follows.

To construct analytic approximations of the above type, we proceed as follows. For definiteness, let us consider $b_{13}$. The analytic approximation $b_{13,a}$ of $b_{13}$ is needed in a region $\mathcal{U} := \{k \,|\, \arg k \in (\arg (\omega^2 k_2), \tfrac{2\pi}{3}), \; 1<|k|<M\}$ to the right of $\Gamma_8^{(1)} = \Gamma_{10}^{(2)}$. Since $f(1) = 0$ and $r_2(1) = -1$, $b_{13}$ has a singularity at $\omega$. Thus, we first choose $m \geq 1$ such that $B(k) := (k-\omega)^m b_{13}(k)$ is regular at $\omega$. 
Applying the method of \cite{DZ1993}, we construct a decomposition $B = B_a + B_r$ such that 
$$\Big|B_a(x,t,k) - \sum_{j=0}^{N_1}\frac{B^{(j)}(\omega)}{j!}(k-\omega)^{j} \Big| \leq C |k- \omega|^{N_1} e^{\frac{t}{8}|\re \Phi_{31}(\zeta,k)|},  \qquad k \in \mathcal{U}, \;\; \zeta \in \mathcal{I}.$$
This implies in particular that $B_{r}$ has a zero of order at least $N_{1}$ at $k=\omega$.
Then we define $b_{13,a} := (k-\omega)^{-m}B_a$ and $b_{13,r} := (k-\omega)^{-m}B_r$. By choosing $N_1$ large enough, we can ensure that the $L^\infty$-norm of $b_{13,r}$ is uniformly small for large $t$ and that
\begin{align}\label{b13anearomega}
\Big|b_{13,a}(x,t,k) - \sum_{j=-m}^{N} b_j (k-\omega)^j \Big| \leq C |k- \omega|^{N} e^{\frac{t}{8}|\re \Phi_{31}(\zeta,k)|},  \qquad k \in \mathcal{U}, \;\; \zeta \in \mathcal{I},
\end{align}
where $\sum_{j=-m}^{N} b_j (k-\omega)^j$ denotes the expansion of $b_{13}$ to order $N$ at $\omega$, i.e., the expansion such that $b_{13}(k) - \sum_{j=-m}^{N} b_j (k-\omega)^j = O((k-\omega)^{N+1})$ as $k \to \omega$. 
This means that $b_{13,a}$ approximates $b_{13}$ to order $N$ at $\omega$. In a similar way, we construct analytic approximations which agree with the functions they are approximating to order $N$ at each $\kappa_j$ that lies in their domain of definition. 
\end{proof}

%\begin{remark}\label{symmetryjumpsremark}\upshape
%Interestingly, the analog of Lemma \ref{symmetryjumpslemma} fails in Sector III because the relevant jumps on $\partial \mathsf{S}$ do not vanish in the case when all spectral functions have analytic continuations. This is one reason why we did not attempt to preserve the $\mathcal{B}$-symmetry in \cite[Section 8]{CLmain}.
%\end{remark}

The jumps on $\Gamma^{(2)}\setminus \mathsf{S}$ can be obtained using the symmetries \eqref{vjsymm}. 

The function $n^{(2)}$ satisfies RH problem \ref{RHnj} with $j = 2$. 
The next transformation $n^{(2)}\to n^{(3)}$ will make $v^{(3)}-I$ small everywhere on a new contour $\Gamma^{(3)}$, except on twelve small crosses centered at the saddle points $\{k_{j},\omega k_{j},\omega^{2}k_{j}\}_{j=1}^{4}$. It will also change the residue conditions (\ref{nresiduesk0})--(\ref{nresiduesk0real}).
To define this transformation, we first need to construct a global parametrix.

\section{Global parametrix}\label{globalparametrixsec}

 For each $\zeta \in \mathcal{I}$, define the analytic functions
\begin{align*}
\delta_{1}(\zeta, \cdot): \mathbb{C}\setminus \Gamma_2^{(2)} \to \mathbb{C}, \quad \delta_{2}(\zeta, \cdot), \delta_{3}(\zeta, \cdot): \mathbb{C}\setminus \Gamma_5^{(2)}\to \mathbb{C}, \quad \delta_{4}(\zeta, \cdot), \delta_{5}(\zeta, \cdot): \mathbb{C}\setminus \Gamma_{10}^{(2)} \to \mathbb{C}
\end{align*}
by (\ref{IVdeltadef}). The functions $\{\delta_j\}_1^5$ obey the jump relations
\begin{align*}
& \delta_{1,+}(\zeta, k) = \delta_{1,-}(\zeta, k)(1 + r_{1}(k)r_{2}(k)), & &  k \in \Gamma_2^{(2)}, \\
& \delta_{2,+}(\zeta, k) = \delta_{2,-}(\zeta, k)(1 + r_{1}(k)r_{2}(k)), & &  k \in \Gamma_5^{(2)}, \\
& \delta_{3,+}(\zeta, k) = \delta_{3,-}(\zeta, k)f(k), & &  k \in \Gamma_5^{(2)}, \\
& \delta_{4,+}(\zeta, k) = \delta_{4,-}(\zeta, k)f(k), & &  k \in \Gamma_{10}^{(2)}, \\
& \delta_{5,+}(\zeta, k) = \delta_{5,-}(\zeta, k)f(\omega^{2}k), & &  k \in \Gamma_{10}^{(2)},
\end{align*}
where the contours $\Gamma_2^{(2)}$, $\Gamma_5^{(2)}$, and $\Gamma_{10}^{(2)}$ are oriented as in Figure \ref{II Gammap2p}, and 
\begin{align}\label{II delta asymp at inf}
\delta_{j}(\zeta, k) = 1 + O(k^{-1}), \qquad k \to \infty, \; j=1, \dots, 5.
\end{align} 
Further properties of the functions $\delta_j$ are collected in the following lemma.

\begin{lemma}\label{II deltalemma}
The functions $\delta_{j}(\zeta, k)$, $j=1,\ldots,5$, have the following properties:
\begin{enumerate}[$(a)$]
\item The functions $\{\delta_{j}(\zeta,k)\}_{j=1}^{5}$ can be written as
\begin{align*}
& \delta_{1}(\zeta,k) = \exp \Big( -i \nu_{1} \ln_{\omega k_{4}}(k-\omega k_{4}) -\chi_{1}(\zeta,k) \Big), \\ %\label{II delta expression in terms of log and chi}
& \delta_{2}(\zeta,k) = \exp \Big( -i \nu_{1} \ln_{\omega k_{4}}(k-\omega k_{4}) +i \nu_{2} \ln_{\omega^{2} k_{2}}(k-\omega^{2} k_{2}) -\chi_{2}(\zeta,k) \Big), \\
& \delta_{3}(\zeta,k) = \exp \Big( - i \nu_{3} \ln_{\omega k_{4}}(k-\omega k_{4}) + i \nu_{4} \ln_{\omega^{2} k_{2}}(k-\omega^{2} k_{2}) -\chi_{3}(\zeta,k) \Big), \\
& \delta_{4}(\zeta,k) = \exp \Big( -i \nu_{4} \ln_{\omega^{2} k_{2}}(k-\omega^{2} k_{2}) - \chi_{4}(\zeta,k) \Big), \\
& \delta_{5}(\zeta,k) = \exp \Big( -i \nu_{5} \ln_{\omega^{2} k_{2}}(k-\omega^{2} k_{2}) - \chi_{5}(\zeta,k) \Big),
\end{align*}
where for $s \in \{e^{i\theta}: \theta \in [\frac{\pi}{2},\frac{2\pi}{3}]\}$, $k \mapsto \ln_{s}(k-s):=\ln |k-s|+i \arg_{s}(k-s)$ has a cut along $\{e^{i \theta}: \theta \in [\frac{\pi}{2},\arg s]\}\cup(i,i\infty)$, and satisfies $\arg_{s}(1)=2\pi$. The $\nu_{j}$ are defined by
\begin{align*}
& \nu_{1} = - \frac{1}{2\pi}\ln(1+r_{1}(\omega k_{4})r_{2}(\omega k_{4})), \qquad \nu_{2} = - \frac{1}{2\pi}\ln(1+r_{1}(\omega^{2} k_{2})r_{2}(\omega^{2} k_{2})), \\
& \nu_{3} = - \frac{1}{2\pi}\ln(f(\omega k_{4})), \qquad \nu_{4} = - \frac{1}{2\pi}\ln(f(\omega^{2} k_{2})), \qquad \nu_{5} = - \frac{1}{2\pi} \ln(f(\omega k_{2})),
\end{align*}
and the functions $\chi_{j}$ are defined by
\begin{align*}
\chi_{1}(\zeta,k) = &\; \frac{1}{2\pi i} \int_{\Gamma_2^{(2)}}  \ln_{s}(k-s) d\ln(1+r_1(s)r_{2}(s)), 
	\\
 \chi_{2}(\zeta,k) = &\; \frac{1}{2\pi i} \int_{\Gamma_5^{(2)}}  \ln_{s}(k-s) d\ln(1+r_1(s)r_{2}(s)), 
	\\
\chi_{3}(\zeta,k) = &\; \frac{1}{2\pi i} \int_{\Gamma_5^{(2)}}  \ln_{s}(k-s) d\ln(f(s)), 
	\\
\chi_{4}(\zeta,k) =&\;\frac{1}{2\pi i} \dashint_{\omega^{2}k_{2}}^{\omega} \hspace{-0.4cm} \ln_{s}(k-s) d\ln(f(s)) 
	\\
:= &\; \frac{1}{2\pi i}\lim_{\epsilon \to 0_{+}} \bigg(  \int_{\omega^{2}k_{2}}^{e^{i(\frac{2\pi}{3}-\epsilon)}} \hspace{-0.4cm} \ln_{s}(k-s) d\ln(f(s)) - \ln_{\omega}(k-\omega)\ln(f(e^{i(\frac{2\pi}{3}-\epsilon)})) \bigg), 
	\\
\chi_{5}(\zeta,k) = &\; \frac{1}{2\pi i} \dashint_{\omega^{2}k_{2}}^{\omega} \hspace{-0.4cm} \ln_{s}(k-s) d\ln(f(\omega^{2}s)) 
	\\
:= &\; \frac{1}{2\pi i}\lim_{\epsilon \to 0_{+}} \bigg(  \int_{\omega^{2}k_{2}}^{e^{i(\frac{2\pi}{3}-\epsilon)}} \hspace{-0.4cm} \ln_{s}(k-s) d\ln(f(\omega^{2}s)) - \ln_{\omega}(k-\omega)\ln(f(e^{-i\epsilon})) \bigg),
\end{align*}
and the integration paths starting at $\omega^{2}k_{2}$ are subsets of $\partial \mathbb{D}$ oriented in the counterclockwise direction.

\item The functions $\{\delta_{j}(\zeta,k)\}_{j=1}^{5}$ can be written as
\begin{align*}
& \delta_{1}(\zeta,k) = \exp \big( -i \nu_{1} \tilde{\ln}_{\omega k_{4}}(k-\omega k_{4}) -\tilde{\chi}_{1}(\zeta,k) \big), \\ %\label{II delta expression in terms of log and chi}
& \delta_{2}(\zeta,k) = \exp \big( -i \nu_{1} \tilde{\ln}_{\omega k_{4}}(k-\omega k_{4}) +i \nu_{2} \tilde{\ln}_{\omega^{2} k_{2}}(k-\omega^{2} k_{2}) -\tilde{\chi}_{2}(\zeta,k) \big), \\
& \delta_{3}(\zeta,k) = \exp \big( - i \nu_{3} \tilde{\ln}_{\omega k_{4}}(k-\omega k_{4}) + i \nu_{4} \tilde{\ln}_{\omega^{2} k_{2}}(k-\omega^{2} k_{2}) -\tilde{\chi}_{3}(\zeta,k) \big), \\
& \delta_{4}(\zeta,k) = \exp \big( -i \nu_{4} \tilde{\ln}_{\omega^{2} k_{2}}(k-\omega^{2} k_{2}) - \tilde{\chi}_{4}(\zeta,k) \big), \\
& \delta_{5}(\zeta,k) = \exp \big( -i \nu_{5} \tilde{\ln}_{\omega^{2} k_{2}}(k-\omega^{2} k_{2}) - \tilde{\chi}_{5}(\zeta,k) \big),
\end{align*}
where for $s \in \{e^{i\theta}: \theta \in [\frac{\pi}{2},\frac{2\pi}{3}]\}$, $k \mapsto \tilde{\ln}_{s}(k-s):=\ln |k-s|+i \tilde{\arg}_{s}(k-s)$ has a cut along $\{e^{i \theta}: \theta \in [\arg s,\pi]\}\cup(-\infty,0)$, satisfies $\tilde{\arg}_{s}(1)=0$, and $\tilde{\chi}_{j}(\zeta,k)$ is defined in the same way as $\chi_{j}(\zeta,k)$ except that $\ln_{s}$ is replaced by $\tilde{\ln}_{s}$ and $\ln_{\omega}$ is replaced by $\tilde{\ln}_{\omega}$.

\item For each $\zeta \in \mathcal{I}$ and $j \in \{1,\ldots,5\}$, $\delta_{j}(\zeta, k)$ and $\delta_{j}(\zeta, k)^{-1}$ are analytic functions of $k$ in their respective domains of definition. Moreover, for any $\epsilon > 0$,
\begin{align}
& \sup_{\zeta \in \mathcal{I}} \sup_{\substack{k \in \C \setminus \Gamma_2^{(2)}}} |\delta_{1}(\zeta,k)^{\pm 1}| < \infty, \quad \sup_{\zeta \in \mathcal{I}} \sup_{\substack{k \in \C \setminus \Gamma_5^{(2)}}} |\delta_{2}(\zeta,k)^{\pm 1}| < \infty, \quad \sup_{\zeta \in \mathcal{I}} \sup_{\substack{k \in \C \setminus \Gamma_5^{(2)}}} |\delta_{3}(\zeta,k)^{\pm 1}| < \infty, \nonumber \\
& \sup_{\zeta \in \mathcal{I}} \sup_{\substack{k \in \C \setminus \Gamma_{10}^{(2)} \\ |k-\omega|>\epsilon}} |\delta_{4}(\zeta,k)^{\pm 1}| < \infty, \quad \sup_{\zeta \in \mathcal{I}} \sup_{\substack{k \in \C \setminus \Gamma_{10}^{(2)}  \\ |k-\omega|>\epsilon}} |\delta_{5}(\zeta,k)^{\pm 1}| < \infty. \label{II delta1bound}
\end{align}

\item As $k \to \omega k_{4}$ along a path which is nontangential to $\partial \mathbb{D}$, we have
\begin{align}
& |\chi_{j}(\zeta,k)-\chi_{j}(\zeta,\omega k_{4})| \leq C |k-\omega k_{4}| (1+|\ln|k-\omega k_{4}||), & & j=1,2,  \label{II asymp chi at k1}
\end{align}
and as $k \to \omega^{2} k_{2}$ along a path which is nontangential to $\partial \mathbb{D}$, we have
\begin{align}
& |\chi_{j}(\zeta,k)-\chi_{j}(\zeta,\omega^{2} k_{2})| \leq C |k-\omega^{2} k_{2}| (1+|\ln|k-\omega^{2} k_{2}||), & & j=2,3,4,5, \label{II asymp chi at omega2k2}
\end{align}
where $C$ is independent of $\zeta \in \mathcal{I}$. 
\end{enumerate}
\end{lemma}
\begin{proof}
All assertions follow from the definitions (\ref{IVdeltadef}) of the functions $\delta_{j}$.
\end{proof}

We will also need the following result.
\begin{lemma}\label{lemma: nuhat lemma II}
For all $\zeta \in (\frac{1}{\sqrt{3}},1)$, the following inequalities hold:
\begin{align*}
\hat{\nu}_{1} := \nu_{3}-\nu_{1}\geq 0, & & \hat{\nu}_{2}:=\nu_{2}+\nu_{5}-\nu_{4}\geq 0.
\end{align*}
\end{lemma}
\begin{proof}
For $\zeta \in (\frac{1}{\sqrt{3}},1)$, we have $-\frac{\pi}{6}<\arg k_{4} < 0$ and $-\frac{3\pi}{4}<\arg k_{2} < -\frac{2\pi}{3}$. Thus the claim follows from \cite[Lemma 2.13]{CLmain}.
\end{proof}

In what follows, we often write $\delta_j(k)$ for $\delta_j(\zeta, k)$ for conciseness. For $k \in \mathbb{C}\setminus (\partial \mathbb{D}\cup \hat{\mathsf{Z}})$, we define
\begin{align*}
& \Delta_{33}(\zeta,k) = \frac{\delta_{1}(\omega k)\delta_{1}(\frac{1}{\omega^{2} k})}{\delta_{1}(\omega^{2} k)\delta_{1}(\frac{1}{\omega k})}    \frac{\delta_{2}(k)\delta_{2}(\frac{1}{k})}{\delta_{2}(\omega k)\delta_{2}(\frac{1}{\omega^{2}k})}    \frac{\delta_{3}(\omega^{2}k)\delta_{3}(\frac{1}{\omega k})}{\delta_{3}(k)\delta_{3}(\frac{1}{k})}     \frac{\delta_{4}(\omega^{2} k)\delta_{4}(\frac{1}{\omega k})}{\delta_{4}(\omega k)\delta_{4}(\frac{1}{\omega^{2} k})}     \frac{\delta_{5}(\omega k)\delta_{5}(\frac{1}{\omega^{2} k})}{\delta_{5}(k)\delta_{5}(\frac{1}{k})} \mathcal{P}(k),     
\end{align*}
where $\mathcal{P}$ is given by \eqref{mathcalPdef}. The function $\Delta_{33}$ satisfies $\Delta_{33}(\zeta,\frac{1}{k}) = \Delta_{33}(\zeta,k)$, and
\begin{subequations}\label{II jumps Delta33}
\begin{align}
& \Delta_{33,+}(\zeta,k) = \Delta_{33,-}(\zeta,k) , & & k \in \Gamma_2^{(2)}, \\
& \Delta_{33,+}(\zeta,k) = \Delta_{33,-}(\zeta,k)\frac{1+r_{1}(k)r_{2}(k)}{f(k)}, & & k \in \Gamma_5^{(2)}, \\
& \Delta_{33,+}(\zeta,k) = \Delta_{33,-}(\zeta,k)\frac{1}{f(\omega^{2}k)}, & & k \in \Gamma_{10}^{(2)}.
\end{align}
\end{subequations}
Define $\Delta_{11}(\zeta,k)=\Delta_{33}(\zeta,\omega k)$ and $\Delta_{22}(\zeta,k)=\Delta_{33}(\zeta,\omega^{2} k)$, and define
\begin{align}\label{II def of Delta}
\Delta(\zeta,k) = \begin{pmatrix}
\Delta_{11}(\zeta,k) & 0 & 0 \\
0 & \Delta_{22}(\zeta,k) & 0 \\
0 & 0 & \Delta_{33}(\zeta,k)
\end{pmatrix}, \quad \zeta \in \mathcal{I}, \; k \in \mathbb{C}\setminus (\partial \mathbb{D} \cup \hat{\mathsf{Z}}).
\end{align}
Then $\Delta$ obeys the symmetries
\begin{align*}
\Delta(\zeta,k) = \mathcal{A}\Delta(\zeta,\omega k)\mathcal{A}^{-1} = \mathcal{B}\Delta(\zeta,\tfrac{1}{k})\mathcal{B},
\end{align*}
and satisfies the jump relations
\begin{subequations}\label{II jumps Delta11}
\begin{align}
& \Delta_{11,+}(\zeta,k) = \Delta_{11,-}(\zeta,k)\frac{1}{1+r_{1}(k)r_{2}(k)} , & & k \in \Gamma_2^{(2)}, \\
& \Delta_{11,+}(\zeta,k) = \Delta_{11,-}(\zeta,k)f(k), & & k \in \Gamma_5^{(2)}, \\
& \Delta_{11,+}(\zeta,k) = \Delta_{11,-}(\zeta,k)f(k), & & k \in \Gamma_{10}^{(2)},
\end{align}
\end{subequations}
and
\begin{subequations}\label{II jumps Delta22}
\begin{align}
& \Delta_{22,+}(\zeta,k) = \Delta_{22,-}(\zeta,k) (1+r_{1}(k)r_{2}(k)) , & & k \in \Gamma_2^{(2)}, \\
& \Delta_{22,+}(\zeta,k) = \Delta_{22,-}(\zeta,k) \frac{1}{1+r_{1}(k)r_{2}(k)}, & & k \in \Gamma_5^{(2)}, \\
& \Delta_{22,+}(\zeta,k) = \Delta_{22,-}(\zeta,k)\frac{f(\omega^{2}k)}{f(k)}, & & k \in \Gamma_{10}^{(2)}.
\end{align}
\end{subequations}
The following lemma is a direct consequence of \eqref{II def of Delta} and of Lemma \ref{II deltalemma}.

\begin{lemma}\label{DeltalemmaIV}
For any $\epsilon>0$, $\Delta(\zeta,k)$ and $\Delta(\zeta,k)^{-1}$ are uniformly bounded for $k \in \mathbb{C}\setminus (\cup_{j=0}^{2} D_\epsilon(\omega^j) \cup \partial \mathbb{D} \cup \cup_{k_{0}\in \hat{\mathsf{Z}}}D_{\epsilon}(k_{0}))$ and $\zeta \in \mathcal{I}$. Furthermore, $\Delta(\zeta,k)=I+O(k^{-1})$ as $k \to \infty$.
\end{lemma} 

\section{The $n^{(2)}\to n^{(3)}$ transformation}\label{n2ton3sec}

Define the sectionally meromorphic function $n^{(3)}$ by
\begin{align}\label{II def of mp3p}
n^{(3)}(x,t,k) = n^{(2)}(x,t,k)\Delta(\zeta,k), \qquad  k \in \mathbb{C}\setminus ( \Gamma^{(3)} \cup \hat{\mathsf{Z}}),
\end{align}
where $\Gamma^{(3)}=\Gamma^{(2)}$. The jumps for $n^{(3)}$ on $\Gamma^{(3)}$ are given by $v^{(3)}=\Delta_{-}^{-1}v^{(2)}\Delta_{+}$.

\begin{lemma}\label{II v3lemma}
The jump matrix $v^{(3)}$ converges to the identity matrix $I$ as $t \to \infty$ uniformly for $\zeta \in \mathcal{I}$ and $k \in \Gamma^{(3)}$ except near the twelve saddle points $\{k_j, \omega k_j, \omega^2 k_j\}_{j=1}^{4}$, i.e., for $\zeta \in \mathcal{I}$,
\begin{align}\label{II v3estimatesa}
& \|v^{(3)} - I\|_{(L^1 \cap L^\infty)(\Gamma'^{(3)})} \leq Ct^{-1}, \quad \text{where\;\; $\Gamma'^{(3)} := \Gamma^{(3)}\setminus \cup_{\substack{j=1,\ldots,4 \\ l =0,1,2}} D_\epsilon(\omega^{l}k_{j})$}.
\end{align}
\end{lemma}
\begin{proof}
From the expressions for $v_j^{(1)}$, $j = 2,5,8$, in (\ref{v11v21v31def})--(\ref{v71v81v91def}) together with \eqref{II jumps vj diagonal p2p}, \eqref{II jumps Delta33}, \eqref{II jumps Delta11}, \eqref{II jumps Delta22}, we deduce that
\begin{align}\label{lol3}
& v_j^{(3)} = \Delta_{-}^{-1}v_j^{(1)}\Delta_{+} = I + \Delta_{-}^{-1}v_{j,r}^{(1)}\Delta_{+}, \quad j = 2,5; \quad v_{10}^{(3)} = \Delta_{-}^{-1}v_8^{(1)}\Delta_{+} = I + \Delta_{-}^{-1}v_{8,r}^{(1)}\Delta_{+}.
\end{align}
Using that the matrices $\Delta_{-}^{-1}v_{j,r}^{(1)}\Delta_{+}$, $j = 2,5,8$, are small as $ t \to \infty$, and recalling the symmetries \eqref{vjsymm} and Lemmas \ref{vsmallnearilemmaIV} and \ref{symmetryjumpslemma}, we infer that $v^{(3)}-I$ tends to $0$ as $t \to \infty$, uniformly for $k \in \Gamma'^{(3)}$. 
\end{proof}

We next consider the residue conditions at the points in $\hat{\mathsf{Z}}$.

\begin{lemma}[Residue conditions for $n^{(3)}$]\label{n3residuelemma}
The function $n^{(3)}$ obeys the following residue conditions at the points in $\hat{\mathsf{Z}}$: 
\begin{enumerate}[$(a)$]
\item for each $k_{0}\in \mathsf{Z}\setminus \mathbb{R}$ with $\re k_0 > 0$, 
\begin{align}\nonumber
& \underset{k = k_0}{\res} n_{1}^{(3)}(k) = \frac{c_{k_0}^{-1}e^{\theta_{31}(k_0)}n_{3}^{(3)}(k_{0})}{\Delta_{33}'(k_{0})(\Delta_{11}^{-1})'(k_{0})}, 
	& & \hspace{-0.3cm}
\underset{k = \omega k_0}{\res} n_{3}^{(3)}(k) = \frac{\omega^{2}c_{k_0}^{-1}e^{\theta_{31}(k_0)}n_{2}^{(3)}(\omega k_{0})}{\Delta_{22}'(\omega k_{0})(\Delta_{33}^{-1})'(\omega k_{0})}, 
	\\  \nonumber
& \underset{k = \omega^{2} k_0}{\res} n_{2}^{(3)}(k) = \frac{\omega c_{k_0}^{-1}e^{\theta_{31}(k_0)}n_{1}^{(3)}(\omega^{2} k_{0})}{\Delta_{11}'(\omega^{2} k_{0})(\Delta_{22}^{-1})'(\omega^{2} k_{0})}, 
	& & \hspace{-0.3cm}
\underset{k = k_0^{-1}}{\res} n_{2}^{(3)}(k) = \frac{-k_{0}^{2}c_{k_0}^{-1}e^{\theta_{31}(k_0)}n_{3}^{(3)}(k_{0}^{-1})}{\Delta_{33}'(k_{0}^{-1})(\Delta_{22}^{-1})'(k_{0}^{-1})},
	\\ \nonumber
&\underset{k = \omega^{2} k_0^{-1}}{\res} \hspace{-0.1cm} n_{3}^{(3)}(k) \hspace{-0.05cm} = \hspace{-0.05cm} \frac{-\omega k_{0}^{2} c_{k_0}^{-1}e^{\theta_{31}(k_0)}n_{1}^{(3)}(\omega^{2} k_{0}^{-1})}{\Delta_{11}'(\omega^{2} k_{0}^{-1})(\Delta_{33}^{-1})'(\omega^{2} k_{0}^{-1})},
 & & \hspace{-0.3cm}
\underset{k = \omega k_0^{-1}}{\res} \hspace{-0.1cm} n_{1}^{(3)}(k) \hspace{-0.05cm} = \hspace{-0.05cm} \frac{-\omega^{2}k_{0}^{2}c_{k_0}^{-1}e^{\theta_{31}(k_0)}n_{2}^{(3)}(\omega k_{0}^{-1})}{\Delta_{22}'(\omega k_{0}^{-1})(\Delta_{11}^{-1})'(\omega k_{0}^{-1})},
	\\ \nonumber
& \underset{k = \bar{k}_0}{\res} n_{3}^{(3)}(k) = \frac{d_{k_0}^{-1} e^{-\theta_{32}(\bar{k}_0)} n_{2}^{(3)}(\bar{k}_{0})}{\Delta_{22}'(\bar{k}_{0})(\Delta_{33}^{-1})'(\bar{k}_{0})}, 
	& & \hspace{-0.3cm}
\underset{k = \omega \bar{k}_0}{\res} n_{2}^{(3)}(k) = \frac{\omega^{2}d_{k_0}^{-1} e^{-\theta_{32}(\bar{k}_0)} n_{1}^{(3)}(\omega \bar{k}_{0})}{\Delta_{11}'(\omega \bar{k}_{0})(\Delta_{22}^{-1})'(\omega \bar{k}_{0})}, 
	\\  \nonumber
& \underset{k = \omega^{2} \bar{k}_0}{\res} n_{1}^{(3)}(k) = \frac{\omega d_{k_0}^{-1} e^{-\theta_{32}(\bar{k}_0)} n_{3}^{(3)}(\omega^{2} \bar{k}_{0})}{\Delta_{33}'(\omega^{2} \bar{k}_{0})(\Delta_{11}^{-1})'(\omega^{2} \bar{k}_{0})}, 
	& & \hspace{-0.3cm}
\underset{k = \bar{k}_0^{-1}}{\res} n_{3}^{(3)}(k) = \frac{-\bar{k}_{0}^{2}d_{k_0}^{-1} e^{-\theta_{32}(\bar{k}_0)} n_{1}^{(3)}(\bar{k}_{0}^{-1})}{\Delta_{11}'(\bar{k}_{0}^{-1})(\Delta_{33}^{-1})'(\bar{k}_{0}^{-1})},	
	\\ \label{n3residuesk0}
& \underset{k = \omega^{2} \bar{k}_0^{-1}}{\res} \hspace{-0.1cm} n_{1}^{(3)}(k) \hspace{-0.05cm} = \hspace{-0.05cm} \frac{\omega \bar{k}_{0}^{2} d_{k_0}^{-1} e^{-\theta_{32}(\bar{k}_0)} n_{2}^{(3)}(\omega^{2} \bar{k}_{0}^{-1})}{-\Delta_{22}'(\omega^{2} \bar{k}_{0}^{-1})(\Delta_{11}^{-1})'(\omega^{2} \bar{k}_{0}^{-1})}, 
       & & \hspace{-0.3cm}
 \underset{k = \omega \bar{k}_0^{-1}}{\res} \hspace{-0.1cm} n_{2}^{(3)}(k) \hspace{-0.05cm} = \hspace{-0.05cm} \frac{\omega^{2}\bar{k}_{0}^{2}d_{k_0}^{-1} e^{-\theta_{32}(\bar{k}_0)} n_{3}^{(3)}(\omega \bar{k}_{0}^{-1})}{-\Delta_{33}'(\omega \bar{k}_{0}^{-1})(\Delta_{22}^{-1})'(\omega \bar{k}_{0}^{-1})};
\end{align}

\item for each $k_{0}\in \mathsf{Z}\cap \mathbb{R}$ with $\re k_0 > 0$, 
\begin{align}\nonumber
& \underset{k = k_0}{\res} n_{1}^{(3)}(k) = \frac{c_{k_0}^{-1} e^{\theta_{21}(k_0)}n_{2}^{(3)}(k_{0})}{\Delta_{22}'(k_{0})(\Delta_{11}^{-1})'(k_{0})},
	& & \hspace{-0.3cm}
 \underset{k = \omega k_0}{\res} n_{3}^{(3)}(k) = \frac{\omega^{2}c_{k_0}^{-1} e^{\theta_{21}(k_0)}n_{1}^{(3)}(\omega k_{0})}{\Delta_{11}'(\omega k_{0})(\Delta_{33}^{-1})'(\omega k_{0})}, 
	\\  \nonumber
& \underset{k = \omega^{2} k_0}{\res} n_{2}^{(3)}(k) = \frac{\omega c_{k_0}^{-1} e^{\theta_{21}(k_0)}n_{3}^{(3)}(\omega^{2} k_{0})}{\Delta_{33}'(\omega^{2} k_{0})(\Delta_{22}^{-1})'(\omega^{2} k_{0})}, 
	& & \hspace{-0.3cm}
\underset{k = k_0^{-1}}{\res} n_{2}^{(3)}(k) = \frac{-k_{0}^{2}c_{k_0}^{-1} e^{\theta_{21}(k_0)}n_{1}^{(3)}(k_{0}^{-1})}{\Delta_{11}'(k_{0}^{-1})(\Delta_{22}^{-1})'(k_{0}^{-1})},	
	\\ \label{n3residuesk0real}
&\underset{k = \omega^{2} k_0^{-1}}{\res} n_{3}^{(3)}(k) = \frac{-\omega k_{0}^{2} c_{k_0}^{-1} e^{\theta_{21}(k_0)}n_{2}^{(3)}(\omega^{2} k_{0}^{-1})}{\Delta_{22}'(\omega^{2} k_{0}^{-1})(\Delta_{33}^{-1})'(\omega^{2} k_{0}^{-1})}, & & \hspace{-0.3cm}
 \underset{k = \omega k_0^{-1}}{\res} n_{1}^{(3)}(k) = \frac{\omega^{2}k_{0}^{2}c_{k_0}^{-1} e^{\theta_{21}(k_0)}n_{3}^{(3)}(\omega k_{0}^{-1})}{-\Delta_{33}'(\omega k_{0}^{-1})(\Delta_{11}^{-1})'(\omega k_{0}^{-1})};
\end{align}

\item for each $k_{0}\in \mathsf{Z}\setminus \mathbb{R}$ with $\re k_0 < 0$, 
\begin{align}\nonumber
& \underset{k = k_0}{\res} n_3^{(3)}(k) = \frac{c_{k_{0}}e^{-\theta_{31}(k_0)} n_1^{(3)}(k_0) }{\Delta_{11}(k_0) \Delta_{33}^{-1}(k_0)}, 
	& & \hspace{-0.2cm}
\underset{k = \omega k_0}{\res} n_2^{(3)}(k) = \frac{\omega c_{k_{0}}e^{-\theta_{31}(k_0)} n_3^{(3)}(\omega k_0) }{\Delta_{33}(\omega k_0) \Delta_{22}^{-1}(\omega k_0)}, 
	\\ \nonumber
& \underset{k = \omega^2 k_0}{\res} n_1^{(3)}(k) = \frac{\omega^2 c_{k_{0}}e^{-\theta_{31}(k_0)} n_2^{(3)}(\omega^2 k_0) }{\Delta_{22}(\omega^2 k_0) \Delta_{11}^{-1}(\omega^2 k_0)},
	& & \hspace{-0.2cm}
\underset{k = k_0^{-1}}{\res} n_3^{(3)}(k) =  \frac{-k_0^{-2} c_{k_{0}}e^{-\theta_{31}(k_0)} n_2^{(3)}(k_0^{-1}) }{\Delta_{22}(k_0^{-1}) \Delta_{33}^{-1}(k_0^{-1})},	
	\\\nonumber
&  \underset{k = \omega^2k_0^{-1}}{\res} \hspace{-0.1cm} n_1^{(3)}(k) =  \frac{-\tfrac{\omega^2}{k_0^{2}} c_{k_{0}}e^{-\theta_{31}(k_0)} n_3^{(3)}(\omega^2 k_0^{-1}) }{\Delta_{33}(\omega^2k_0^{-1}) \Delta_{11}^{-1}(\omega^2k_0^{-1})}, 
	& & \hspace{-0.2cm}
\underset{k = \omega k_0^{-1}}{\res} \hspace{-0.1cm}  n_2^{(3)}(k) 
= \frac{-\tfrac{\omega}{k_0^{2}} c_{k_{0}}e^{-\theta_{31}(k_0)} n_1^{(3)}(\omega k_0^{-1}) }{\Delta_{11}(\omega k_0^{-1}) \Delta_{22}^{-1}(\omega k_0^{-1})}, 
	 \\ \nonumber
& \underset{k = \bar{k}_0}{\res} n_2^{(3)}(k) = \frac{d_{k_0} e^{\theta_{32}(\bar{k}_0)} n_3^{(3)}(\bar{k}_0)  }{\Delta_{33}(\bar{k}_0) \Delta_{22}^{-1}(\bar{k}_0)},
 	& & \hspace{-0.2cm}
\underset{k = \omega \bar{k}_0}{\res} n_1^{(3)}(k) = \frac{\omega d_{k_0} e^{\theta_{32}(\bar{k}_0)} n_2^{(3)}( \omega \bar{k}_0) }{\Delta_{22}(\omega \bar{k}_0) \Delta_{11}^{-1}(\omega \bar{k}_0)},
	\\ \nonumber
& \underset{k = \omega^2 \bar{k}_0}{\res} n_3^{(3)}(k) = \frac{\omega^2 d_{k_0} e^{\theta_{32}(\bar{k}_0)} n_1^{(3)}(\omega^2 \bar{k}_0) }{\Delta_{11}(\omega^2 \bar{k}_0) \Delta_{33}^{-1}(\omega^2 \bar{k}_0)},
	& & \hspace{-0.2cm}
\underset{k = \bar{k}_0^{-1}}{\res} n_1^{(3)}(k) =  \frac{-\bar{k}_0^{-2} d_{k_0} e^{\theta_{32}(\bar{k}_0)} n_3^{(3)}(\bar{k}_0^{-1}) }{\Delta_{33}(\bar{k}_0^{-1}) \Delta_{11}^{-1}(\bar{k}_0^{-1})},
	\\ \label{n3residuesk0negative}
&  \underset{k = \omega^2 \bar{k}_0^{-1}}{\res} n_2^{(3)}(k) = \frac{-\tfrac{\omega^2}{\bar{k}_0^{2}} d_{k_0} e^{\theta_{32}(\bar{k}_0)} n_1^{(3)}( \omega^2 \bar{k}_0^{-1}) }{\Delta_{11}( \omega^2 \bar{k}_0^{-1}) \Delta_{22}^{-1}( \omega^2 \bar{k}_0^{-1})}, 
	& & \hspace{-0.2cm}
\underset{k = \omega \bar{k}_0^{-1}}{\res} n_3^{(3)}(k) = \frac{-\tfrac{\omega}{\bar{k}_0^{2}} d_{k_0} e^{\theta_{32}(\bar{k}_0)} n_2^{(3)}(\omega \bar{k}_0^{-1}) }{\Delta_{22}(\omega \bar{k}_0^{-1}) \Delta_{33}^{-1}(\omega \bar{k}_0^{-1})};
\end{align}

\item for each $k_{0}\in \mathsf{Z}\cap \mathbb{R}$ with $\re k_0 < 0$, 
\begin{align}\nonumber
& \underset{k = k_0}{\res} n_2^{(3)}(k) = \frac{c_{k_0} e^{-\theta_{21}(k_0)} n_1^{(3)}(k_0) }{\Delta_{11}(k_0) \Delta_{22}^{-1}(k_0)}, 
	& & \hspace{-0.2cm}
\underset{k = \omega k_0}{\res} n_1^{(3)}(k) = \frac{\omega c_{k_0} e^{-\theta_{21}(k_0)} n_3^{(3)}(\omega k_0) }{\Delta_{33}(\omega k_0) \Delta_{11}^{-1}(\omega k_0)}, 
	\\  \nonumber
& \underset{k = \omega^2 k_0}{\res} n_3^{(3)}(k) = \frac{\omega^2 c_{k_0} e^{-\theta_{21}(k_0)} n_2^{(3)}(\omega^2 k_0) }{\Delta_{22}(\omega^2 k_0) \Delta_{33}^{-1}(\omega^2 k_0)}, 
	& & \hspace{-0.2cm}
\underset{k = k_0^{-1}}{\res} n_1^{(3)}(k) =  \frac{-k_0^{-2} c_{k_0} e^{-\theta_{21}(k_0)} n_2^{(3)}(k_0^{-1}) }{\Delta_{22}(k_0^{-1}) \Delta_{11}^{-1}(k_0^{-1})},		
	\\ \label{n3residuesk0realnegative}
&  \underset{k = \omega^2k_0^{-1}}{\res} \hspace{-0.1cm} n_2^{(3)}(k) =  \frac{-\tfrac{\omega^2}{k_0^{2}} c_{k_0} e^{-\theta_{21}(k_0)} n_3^{(3)}(\omega^2 k_0^{-1}) }{\Delta_{33}(\omega^2k_0^{-1}) \Delta_{22}^{-1}(\omega^2k_0^{-1})},
	& &\hspace{-0.2cm} 
\underset{k = \omega k_0^{-1}}{\res} \hspace{-0.1cm} n_3^{(3)}(k) 
= \frac{-\tfrac{\omega}{k_0^{2}} c_{k_0} e^{-\theta_{21}(k_0)} n_1^{(3)}(\omega k_0^{-1}) }{\Delta_{11}(\omega k_0^{-1}) \Delta_{33}^{-1}(\omega k_0^{-1})},
\end{align}
\end{enumerate}
where the $(x,t)$-dependence has been suppressed for clarity.
Moreover, at each point of $\hat{\mathsf{Z}}$, two entries of $n^{(3)}$ are analytic while one entry has (at most) a simple pole. 
\end{lemma}
\begin{proof}
Since $\Delta(k) =\Delta(\zeta,k)$ is analytic at each point of $k_0 \in \mathsf{Z} \cap \{\re k_0 < 0\}$, the claims $(c)$ and $(d)$ regarding points $k_0 \in \mathsf{Z}$ with $\re k_0 < 0$ follow easily from RH problem \ref{RHn} and the fact that $n^{(3)} = n \Delta$ near every point of $\mathsf{Z}$.
It remains to consider the points $k_0 \in \mathsf{Z}$ with $\re k_0 > 0$.
We prove the first residue condition in (\ref{n3residuesk0}); the other conditions are proved in a similar way. Let $k_{0}\in \mathsf{Z}\setminus \mathbb{R}$ be such that $\re k_0 > 0$. Recall that $n^{(2)}$ obeys the same residue conditions as $n$; in particular, by (\ref{nresiduesk0}), 
\begin{align}\label{n2residuesk0}
\underset{k = k_0}{\res} n_3^{(2)}(k) = c_{k_{0}}e^{-\theta_{31}(k_0)} n_1^{(2)}(k_0).
\end{align}
Since $\Delta_{33}$ has a simple zero at $k_{0}$ and $\Delta_{11}(k) = \Delta_{33}(\omega k)$ has a simple pole at $k_{0}$, $n_{3}^{(3)}$ has a removable singularity at $k_{0}$ and $n_{1}^{(3)}$ has a simple pole at $k_{0}$. Using \eqref{n2residuesk0}, we obtain
\begin{align*}
& n_{3}^{(3)}(k_{0}) = \lim_{k\to k_{0}} n_{3}^{(2)}(k)\Delta_{33}(k) = \underset{k = k_0}{\res} n_{3}^{(2)}(k) \cdot \Delta_{33}'(k_{0}) = c_{k_{0}}e^{-\theta_{31}(k_0)} n_{1}^{(2)}(k_{0}) \Delta_{33}'(k_{0}), 
\end{align*}
and hence
\begin{align*}
& \underset{k = k_0}{\res} n_{1}^{(3)}(k) = \underset{k = k_0}{\res} n_{1}^{(2)}(k)\Delta_{11}(k) = \frac{n_{1}^{(2)}(k_{0})}{(\Delta_{11}^{-1})'(k_{0})} = \frac{c_{k_{0}}^{-1}e^{\theta_{31}(k_0)}n_{3}^{(3)}(k_{0})}{\Delta_{33}'(k_{0})(\Delta_{11}^{-1})'(k_{0})}.
\end{align*} 
This shows that only the first entry of $n^{(3)}$ is singular at $k_0$ and that the corresponding residue is as stated in (\ref{n3residuesk0}).
\end{proof}

Using Lemma \ref{n3residuelemma}, we infer that $n^{(3)}$ satisfies RH problem \ref{RHnj} for $j = 3$ except that item $(\ref{RHnjitemf})$ describing the behavior near points in $\hat{\mathsf{Z}}$ must be replaced by the following:

\begin{enumerate}[$(f')$]

\item \label{RHn3itemf} At each point of $\hat{\mathsf{Z}}$, two entries of $n^{(3)}$ are analytic while one entry has (at most) a simple pole. Moreover, for $k_0 \in \mathsf{Z}$ such that $\re k_0 > 0$, $n^{(3)}$ satisfies the residue conditions \eqref{n3residuesk0}--\eqref{n3residuesk0real}, while for $k_0 \in \mathsf{Z}$ such that $\re k_0 < 0$, $n^{(3)}$ satisfies the residue conditions \eqref{n3residuesk0negative}--\eqref{n3residuesk0realnegative}.
\end{enumerate}

The residue conditions in (\ref{n3residuesk0})--(\ref{n3residuesk0realnegative}) involve the following exponentials:
$(a)$ $e^{t \Phi_{31}(x,t,k_0)}$ and $e^{- t \Phi_{32}(x,t,\bar{k}_0)}$ for $k_0 \in D_{\mathrm{reg}}^{R}$, $(b)$ $e^{t\Phi_{21}(x,t,k_0)}$ for $k_0 \in (1,\infty)$, $(c)$ $e^{-t \Phi_{31}(x,t,k_0)}$ and $e^{ t \Phi_{32}(x,t,\bar{k}_0)}$ for $k_0 \in D_{\mathrm{reg}}^{L}$, and $(d)$ $e^{-t\Phi_{21}(x,t,k_0)}$ for $k_0 \in (-1,0)$. It follows from the signature tables of $\Phi_{21}$, $\Phi_{31}$, and $\Phi_{32}$ (see Figure \ref{II fig: Re Phi 21 31 and 32 for zeta=0.7}) that the real parts of all the exponents in these exponentials are negative. Thus, all the residues of $n^{(3)}$ at points in $\hat{\mathsf{Z}}$ are suppressed by exponentially small factor as $t \to \infty$. This suggests that the poles in $\hat{\mathsf{Z}}$ have no effect on the leading large $t$ asymptotics of $n^{(3)}$. Our next transformation will help us make this claim precise.

\section{The $n^{(3)}\to n^{(4)}$ transformation}\label{n3ton4sec}
The fourth transformation $n^{(3)}\to n^{(4)}$ replaces the residue conditions (\ref{n3residuesk0})--(\ref{n3residuesk0realnegative}) by jumps on small circles. 
For each $k_{0}\in \mathsf{Z}$, we let $D_{\epsilon}(k_{0})$ be a small open disk centered at $k_0$ of radius $\epsilon>0$. We let $D_{\epsilon}(k_{0})^{-1}$ and $D_{\epsilon}(k_{0})^*$ be the images of $D_{\epsilon}(k_{0})$ under the maps $k \mapsto k^{-1}$ and $k \mapsto \bar{k}$, respectively. If $k_{0} \in \mathbb{R}$, then $D_{\epsilon}(k_{0})^*=D_{\epsilon}(k_{0})$. Let $\partial D_{\epsilon}(k_{0})$ and $\partial D_{\epsilon}(k_{0})^*$ be oriented counterclockwise, and let $\partial D_{\epsilon}(k_{0})^{-1}$ be oriented clockwise. Let 
\begin{align*}
& \mathcal{D}_{\sol} = \bigcup_{k_{0}\in \mathsf{Z}}\bigcup_{j=0,1,2} \bigg( \omega^{j} D_{\epsilon}(k_{0}) \cup \omega^{j}  D_{\epsilon}(k_{0})^{-1} \cup \omega^{j} D_{\epsilon}(k_{0})^{*} \cup \omega^{j}  D_{\epsilon}(k_{0})^{-*}\bigg), \\
& \partial \mathcal{D}_{\sol} = \bigcup_{k_{0}\in \mathsf{Z}}\bigcup_{j=0,1,2} \bigg( \omega^{j} \partial D_{\epsilon}(k_{0}) \cup \omega^{j}  \partial D_{\epsilon}(k_{0})^{-1} \cup \omega^{j} \partial D_{\epsilon}(k_{0})^{*} \cup \omega^{j}  \partial D_{\epsilon}(k_{0})^{-*}\bigg),
\end{align*}
so that $\partial \mathcal{D}_{\sol}$ is the union of $6 |\mathsf{Z}\cap \mathbb{R}| + 12 |\mathsf{Z}\setminus \mathbb{R}|$ small circles.
By shrinking $\epsilon>0$ if necessary, we may assume that these circles do not intersect each other, and that they do not intersect $\Gamma^{(3)}$. We let $\Gamma^{(4)} :=  \Gamma^{(3)} \cup \partial \mathcal{D}_{\sol}$ be the union of $\Gamma^{(3)}$ and the small circles $\partial \mathcal{D}_{\sol}$, see Figure \ref{Gamma4fig}. Let $\Gamma_{\star}^{(4)}$ be the set of intersection points of $\Gamma^{(4)}$. Clearly, $\Gamma_{\star}^{(4)}=\Gamma_{\star}^{(3)}$.

\begin{figure}
\begin{center}
\begin{tikzpicture}[master, scale=1.1]
\node at (0,0) {};
\draw[black,line width=0.55 mm] (0,0)--(30:6.5);
\draw[black,line width=0.55 mm,->-=0.35,->-=0.65,->-=0.82,->-=0.98] (0,0)--(90:6.5);
\draw[black,line width=0.55 mm] (0,0)--(150:6.5);

%Boundary of sector S
\draw[dashed,black,line width=0.15 mm] (0,0)--(60:7.8);
\draw[dashed,black,line width=0.15 mm] (0,0)--(120:7.8);

\draw[black,line width=0.55 mm] ([shift=(30:3*1.5cm)]0,0) arc (30:150:3*1.5cm);
\draw[black,arrows={-Triangle[length=0.27cm,width=0.18cm]}]
($(70:3*1.5)$) --  ++(70-90:0.001);
\draw[black,arrows={-Triangle[length=0.27cm,width=0.18cm]}]
($(70:3.65)$) --  ++(70-90:0.001);
\draw[black,arrows={-Triangle[length=0.27cm,width=0.18cm]}]
($(70:5.8)$) --  ++(70-90:0.001);
%\draw[black,arrows={-Triangle[length=0.27cm,width=0.18cm]}]
%($(122:3*1.5)$) --  ++(120+90:0.001);
%\draw[black,arrows={-Triangle[length=0.24cm,width=0.16cm]}]
%($(122:3.3*1.5)$) --  ++(120+90:0.001);
%\draw[black,arrows={-Triangle[length=0.24cm,width=0.16cm]}]
%($(122:2.75*1.5)$) --  ++(120+90:0.001);

\draw[black,line width=0.55 mm] ([shift=(90:3.65)]0,0) arc (90:150:3.65);
\draw[black,line width=0.55 mm] ([shift=(90:5.8)]0,0) arc (90:150:5.8);

%Contours 6_s and 5_s
\draw[black,line width=0.55 mm,-<-=0.08,->-=0.80] (100:3.65)--(100:5.8);

\draw[black,line width=0.55 mm,->-=0.9] (108:3.65)--(108:4.1);
\draw[black,line width=0.55 mm,->-=0.75] (108:4.9)--(108:5.8);

%Contours 8_s and 7_s
\draw[black,line width=0.55 mm,-<-=0.08,->-=0.80] (115:3.65)--(115:5.8);

\draw[black,line width=0.55 mm,->-=0.9] (120:3.65)--(120:4.1);
\draw[black,line width=0.55 mm,->-=0.75] (120:4.9)--(120:5.8);
\draw[black,line width=0.55 mm] (125:3.65)--(125:5.8);
\draw[black,line width=0.55 mm] (132:3.65)--(132:4.1);
\draw[black,line width=0.55 mm] (132:4.95)--(132:5.8);
\draw[black,line width=0.55 mm] (140:3.65)--(140:5.8);

%Curved lenses outside unit circle
\draw[black,line width=0.55 mm] (100:3*1.5) to [out=100+45, in=108-90] (108:3.3*1.5) to [out=108+90, in=115-45] (115:3*1.5) to [out=115+45, in=120-90] (120:3.3*1.5) to [out=120+90, in=140-45]  (125:3*1.5) to [out=125+45, in=132-90] (132:3.3*1.5) to [out=132+90, in=140-45] (140:3*1.5);

%Curved lenses inside unit circle
\draw[black,line width=0.55 mm] (100:3*1.5) to [out=100+45+90, in=108-90] (108:2.75*1.5) to [out=108+90, in=115-45-90] (115:3*1.5) to [out=115+45+90, in=120-90] (120:2.75*1.5) to [out=120+90, in=140-45-90]  (125:3*1.5) to [out=125+45+90, in=132-90] (132:2.75*1.5) to [out=132+90, in=140-45-90] (140:3*1.5);

%Contours 1 and 3
\draw[black,line width=0.55 mm,-<-=0.15,->-=0.80] (90:3.65)--(100:3*1.5)--(90:5.8);

\draw[black,line width=0.55 mm] (150:3.65)--(140:3*1.5)--(150:5.8);

\draw[black,line width=0.55 mm] ([shift=(30:3.65cm)]0,0) arc (30:90:3.65cm);
\draw[black,line width=0.55 mm] ([shift=(30:5.8cm)]0,0) arc (30:90:5.8cm);

\draw[black,arrows={-Triangle[length=0.27cm,width=0.18cm]}]
($(93:3*1.5)$) --  ++(93-90:0.001);
\draw[black,arrows={-Triangle[length=0.27cm,width=0.18cm]}]
($(110:3*1.5)$) --  ++(110+90:0.001);
\draw[black,arrows={-Triangle[length=0.27cm,width=0.18cm]}]
($(110:3.3*1.5)$) --  ++(110+90:0.001);
\draw[black,arrows={-Triangle[length=0.27cm,width=0.18cm]}]
($(110:2.75*1.5)$) --  ++(110+90:0.001);

\node at (95.5:3.58*1.5) {\scriptsize $1$};
\node at (93:3.12*1.5) {\scriptsize $2$};
\node at (93.5:2.77*1.5) {\scriptsize $3$};
\node at (108:3.12*1.5) {\scriptsize $5$};

\node at (117.5:3.7*1.5) {\tiny $10$};
\draw[->-=1] (117.5:3.6*1.5)--(118:3.02*1.5);
%\node at (118.05:3.10*1.5) {\tiny $10$};

%New notation: 8_{7L},R -> 4, 8_{7R},R -> 6, 8_{7L},L -> 7, 8_{7R},L -> 8, \tilde{7}_{L} -> 9, \tilde{7} -> 10, \tilde{7}_{R} -> 11

%\node at (105:4.3*1.5) {\small $4$};
%\draw[->-=1] (105:4.2*1.5)--(105:3.3*1.5);
\node at (105:3.35*1.5) {\scriptsize $4$};

%\node at (103:2*1.5) {\small $6$};
%\draw[->-=1] (103:2.1*1.5)--(103:2.8*1.5);
\node at (104:4.07) {\scriptsize $6$};

%\node at (112:4.75*1.5) {\small $7$};
%\draw[->-=1] (112:4.65*1.5)--(111:3.25*1.5);
\node at (112:4.9) {\scriptsize $7$};

%\node at (112:1.55*1.5) {\scriptsize $8$};
%\draw[->-=1] (112:1.65*1.5)--(110:2.77*1.5);
\node at (110.5:4.3) {\scriptsize $8$};

%\node at (115.7:4.3*1.5) {\scriptsize $9$};
%\draw[->-=1] (116:4.16*1.5)--(117.5:3.23*1.5);
\node at (116.5:4.96) {\scriptsize $9$};

\node at (115.8:2.2*1.5) {\scriptsize $11$};
\draw[->-=1] (116:2.3*1.5)--(117.5:2.78*1.5);
%\node at (117.8:4) {\tiny $11$};

% The angles have been slightly modified to make the figure more visible
\draw[blue,fill] (125:3*1.5) circle (0.1cm);
\draw[green,fill] (-125+240:3*1.5) circle (0.1cm);
\draw[red,fill] (100:3*1.5) circle (0.1cm);
\draw[red,fill] (140:3*1.5) circle (0.1cm);

% Contour through omega
\draw[black,line width=0.4 mm] (120:2.75*1.5)--(120:3.3*1.5);
\draw[black,arrows={-Triangle[length=0.16cm,width=0.12cm]}]
($(120:2.92*1.5)$) --  ++(120:0.001);
\draw[black,arrows={-Triangle[length=0.16cm,width=0.12cm]}]
($(120:3.2*1.5)$) --  ++(120:0.001);

%Contour through -\omega^2
\draw[black,line width=0.55 mm,->-=0.3,->-=0.8] (60:2.45*1.5)--(60:3.85*1.5);

%Arrow for contour 7
\draw[black,arrows={-Triangle[length=0.2cm,width=0.14cm]}]
($(117.7:3.25*1.5)$) --  ++(120+25:0.001);

%Arrow for contour 8
\draw[black,arrows={-Triangle[length=0.2cm,width=0.14cm]}]
($(120:3*1.5)$) --  ++(120+90:0.001);

%Arrow for contour 9
\draw[black,arrows={-Triangle[length=0.2cm,width=0.14cm]}]
($(118.5:2.75*1.5)$) --  ++(120+140:0.001);

%\node at (71:6.05) {\small $9_{R}$};
\node at (71:3.15*1.5) {\scriptsize $9_r$};
%\node at (71:3.9) {\small $9_{L}$};

\node at (127.8:3.7*1.5) {\scriptsize $1_s$};
\draw[->-=1] (126.5:3.63*1.5)--(120.8:3.14*1.5);
\node at (127:3.4*1.5) {\scriptsize $2_s$};
\draw[->-=1] (126:3.3*1.5)--(121:2.86*1.5);

\node at (57.3:3.5*1.5) {\scriptsize $3_s$};
\node at (56.5:2.8*1.5) {\scriptsize $4_s$};

\node at (102:5.3) {\scriptsize $5_{s}$};
\node at (97.2:3.87) {\scriptsize $6_{s}$};
\node at (112.9:5.25) {\scriptsize $7_{s}$};
\node at (112.4:3.9) {\scriptsize $8_{s}$};

\node at (85.7:2.7*1.5) {\small $1_r''$};
\node at (86.7:3.45*1.5) {\small $4_{r}'$};

% ----------------------start the solitons---------------------
\draw[line width=0.45 mm] (120:6.9) circle (0.4);
\draw[black,arrows={-Triangle[length=0.2cm,width=0.18cm]}]
($(120:7.3)+(-0.866*0.15,-0.5*0.15)$) --  ++(120+90:0.001);

% This is the true inverse of the circle. DO NOT ERASE
%\draw[red,line width=0.15 mm] plot[samples=400,variable=\t,domain=-pi:pi,smooth] ({4.5^2*(6.9+0.4*cos(-\t r))/((6.9+0.4*cos(\t r))^2+(0.4*sin(\t r))^2)},{4.5^2*(0.4*sin(-\t r))/((6.9+0.4*cos(\t r))^2+(0.4*sin(\t r))^2 ) });
%\draw[line width=0.15 mm] (0:4.5^2/6.9+0.01) circle (0.18);

\draw[line width=0.45 mm] (120:4.5^2/6.9+0.01) circle (0.18);
\draw[black,arrows={-Triangle[length=0.2cm,width=0.18cm]}]
($(120:4.5^2/6.9+0.018-0.18)+(-0.866*0.11,-0.5*0.11)$) --  ++(120+98:0.001);

\draw[line width=0.45 mm] (120+21:6.5) circle (0.4);
\draw[black,arrows={-Triangle[length=0.2cm,width=0.18cm]}]
($(120+21:6.5)+(120:0.4)+(-0.866*0.15,-0.5*0.15)$) --  ++(120+90:0.001);
\draw[line width=0.45 mm] (120-21:6.5) circle (0.4);
\draw[black,arrows={-Triangle[length=0.2cm,width=0.18cm]}]
($(120-21:6.5)+(120:0.4)+(-0.866*0.15,-0.5*0.15)$) --  ++(120+90:0.001);

% This is the true inverse of the circle. DO NOT ERASE
%\draw[red,line width=0.15 mm] plot[samples=400,variable=\t,domain=-pi:pi,smooth] ({4.5^2*(6.5*cos(-21)+0.4*cos(-\t r))/((6.5*cos(21)+0.4*cos(\t r))^2+(6.5*sin(21)+0.4*sin(\t r))^2)},{4.5^2*(6.5*sin(-21)+0.4*sin(-\t r))/((6.5*cos(21)+0.4*cos(\t r))^2+(6.5*sin(21)+0.4*sin(\t r))^2) });
%\draw[line width=0.15 mm] (-21:3.125) circle (0.20);

\draw[line width=0.45 mm] (-21+120:3.125) circle (0.2);
\draw[black,arrows={-Triangle[length=0.2cm,width=0.18cm]}]
($(-21+120:3.125)-(120:0.20)+(-0.866*0.10,-0.5*0.10)$) --  ++(98+120:0.001);
\draw[line width=0.45 mm] (21+120:3.125) circle (0.2);
\draw[black,arrows={-Triangle[length=0.2cm,width=0.18cm]}]
($(21+120:3.125)-(120:0.20)+(-0.866*0.10,-0.5*0.10)$) --  ++(98+120:0.001);

% ----------------------end the solitons---------------------
\end{tikzpicture}

\end{center}
\begin{figuretext}
\label{Gamma4fig}The contour $\Gamma^{(4)}$ (solid) in a case when $\mathsf{Z}$ contains one point in $D_{\mathrm{reg}}^R$, one point in $(1,\infty)$, and no other points; and the boundary of $\mathsf{S}$ (dashed). From right to left, the dots are $\omega k_{4}$ (red), $\omega^{2}k_{2}$ (green), $k_{1}$ (blue), and $\omega k_{3}$ (red).
\end{figuretext}
\end{figure}

The function $n^{(4)}$ is defined to equal $n^{(3)}$ except in $\mathcal{D}_{\sol}$ where we define $n^{(4)}(x,t,k)$ as follows. If $k_{0} \in \mathsf{Z}\setminus \mathbb{R}$, we define $n^{(4)}$ for $k \in D_{\epsilon}(k_{0}) \cup D_{\epsilon}(k_{0})^*$ by
\begin{align}\label{def of ntildeC}
n^{(4)}(x,t,k) = \begin{cases} 
n^{(3)}(x,t,k) Q_1^R(x,t,k), \quad \; & k \in D_{\epsilon}(k_{0}), \; k_{0} \in \mathsf{Z} \cap D_{\mathrm{reg}}^R,
	\\
n^{(3)}(x,t,k) Q_7^R(x,t,k), \quad & k \in D_{\epsilon}(k_{0})^*, \; k_{0} \in \mathsf{Z} \cap D_{\mathrm{reg}}^R,
	\\
n^{(3)}(x,t,k) Q_1^L(x,t,k), \quad & k \in D_{\epsilon}(k_{0}), \; k_{0} \in \mathsf{Z} \cap D_{\mathrm{reg}}^L,
	\\
n^{(3)}(x,t,k) Q_7^L(x,t,k), \quad & k \in D_{\epsilon}(k_{0})^*, \; k_{0} \in \mathsf{Z} \cap D_{\mathrm{reg}}^L,
\end{cases}
\end{align}
where
\begin{align*}
& Q_1^R = \begin{pmatrix} 1 & 0 & 0 \\ 0 & 1 & 0 \\ 
\frac{-c_{k_0}^{-1}e^{\theta_{31}(x,t,k_0)}}{\Delta_{33}'(k_{0})(\Delta_{11}^{-1})'(k_{0}) (k-k_0)}
%\frac{k^2-\omega}{k_0^2-\omega}
 & 0 & 1 \end{pmatrix}, &&
 Q_7^R = \begin{pmatrix} 1 & 0 & 0 \\ 0 & 1 & \frac{-d_{k_0}^{-1} e^{-\theta_{32}(x,t,\bar{k}_0)}}{\Delta_{22}'(\bar{k}_{0})(\Delta_{33}^{-1})'(\bar{k}_{0})(k- \bar{k}_0)} %\frac{k^2-1}{\bar{k}_0^2 -1} 
 \\ 0 & 0  & 1 \end{pmatrix},
 	\\
& Q_1^L = \begin{pmatrix} 1 & 0 & \frac{-c_{k_{0}}e^{-\theta_{31}(x,t,k_0)} }{\Delta_{11}(k_0) \Delta_{33}^{-1}(k_0)(k-k_0)} %\frac{k^2-\omega}{k_0^2-\omega} 
\\ 0 & 1 & 0 \\ 0 & 0 & 1 \end{pmatrix}, &&
 Q_7^L = \begin{pmatrix} 1 & 0 & 0 \\ 0 & 1 & 0 \\ 0 & \frac{-d_{k_0} e^{\theta_{32}(x,t,\bar{k}_0)} }{\Delta_{33}(\bar{k}_0) \Delta_{22}^{-1}(\bar{k}_0) (k-\bar{k}_0)} %\frac{k^2-1}{\bar{k}_0^2 -1}  
 & 1 \end{pmatrix}, 
\end{align*}
and if $k_{0} \in \mathsf{Z} \cap \mathbb{R}$, we define $n^{(4)}$ for $k \in D_{\epsilon}(k_{0})$ by
\begin{align}\label{def of ntildeR}
n^{(4)}(x,t,k) = 
\begin{cases}
n^{(3)}(x,t,k) P_1^R(x,t,k), \quad \; & k \in D_{\epsilon}(k_{0}), \; k_{0} \in \mathsf{Z}\cap (1,\infty),
	\\
n^{(3)}(x,t,k) P_1^L(x,t,k), \quad & k \in D_{\epsilon}(k_{0}), \; k_{0} \in \mathsf{Z}\cap (-1,0),
\end{cases}
\end{align}
where
\begin{align*}
& P_1^R = \begin{pmatrix} 1 & 0 & 0 \\ 
 \frac{-c_{k_0}^{-1} e^{\theta_{21}(x,t,k_0)}}{\Delta_{22}'(k_{0})(\Delta_{11}^{-1})'(k_{0}) (k-k_0)}
%\frac{k^2-\omega}{k_0^2-\omega} 
& 1 & 0 \\ 0 & 0 & 1 \end{pmatrix},
&&
P_1^L = \begin{pmatrix} 1 & \frac{-c_{k_0} e^{-\theta_{21}(x,t,k_0)}}{\Delta_{11}(k_0) \Delta_{22}^{-1}(k_0) (k- k_0)} %\frac{k^2-\omega}{k_0^2-\omega} 
& 0 \\ 0 & 1 & 0 \\ 0 & 0 & 1 \end{pmatrix}.
\end{align*}
We then extend $n^{(4)}$ to all of $\mathcal{D}_{\sol}$ by means of the $\mathcal{A}$- and $\mathcal{B}$-symmetries (\ref{njsymm}).

It follows from the above definition of $n^{(4)}$ and Lemma \ref{n3residuelemma} that $n^{(4)}$ has no poles at the points in $\hat{\mathsf{Z}}$. Moreover, on $\Gamma^{(4)} \setminus \Gamma^{(4)}_\star$, the boundary values of $n^{(4)}$ exist, are continuous, and satisfy $n^{(4)}_+ = n^{(4)}_-v^{(4)}$, where $v^{(4)}$ equals $v^{(3)}$ on $\Gamma$ and $v^{(4)}$ is defined on $\partial \mathcal{D}_{\sol}$ by setting
\begin{align}\label{vregdef}
v(x,t,k) = \begin{cases}
Q_1^R(x,t,k), \quad \; & k \in \partial D_{\epsilon}(k_{0}), \; k_{0} \in \mathsf{Z} \cap D_{\mathrm{reg}}^R,
    \\
Q_7^R(x,t,k), & k \in \partial D_{\epsilon}(k_{0})^*, \; k_{0} \in \mathsf{Z} \cap D_{\mathrm{reg}}^R,
	\\
Q_1^L(x,t,k), & k \in \partial D_{\epsilon}(k_{0}), \; k_{0} \in \mathsf{Z} \cap D_{\mathrm{reg}}^L,
    \\
Q_7^L(x,t,k), & k \in \partial D_{\epsilon}(k_{0})^*, \; k_{0} \in \mathsf{Z} \cap D_{\mathrm{reg}}^L,
	\\
P_1^R(x,t,k), & k \in \partial D_{\epsilon}(k_{0}), \; k_{0} \in \mathsf{Z}\cap (1,\infty),
	\\
P_1^L(x,t,k), & k \in \partial D_{\epsilon}(k_{0}), \; k_{0} \in \mathsf{Z}\cap (-1,0),
\end{cases}
\end{align}
and then extending it to all of $\partial \mathcal{D}_{\sol}$ by means of the $\mathcal{A}$- and $\mathcal{B}$-symmetries (\ref{vjsymm}).

\begin{lemma}\label{circleslemma}
The jump matrix $v^{(4)}$ converges to the identity matrix $I$ as $t \to \infty$ uniformly for $k \in \partial \mathcal{D}_{\sol}$ and $\zeta \in \mathcal{I}$. More precisely, for $\zeta \in \mathcal{I}$,
\begin{align*}
& \|v^{(4)} - I\|_{(L^1 \cap L^\infty)(\partial \mathcal{D}_{\sol})} \leq Ct^{-N}.
\end{align*}
\end{lemma}
\begin{proof}
The statement is a direct consequence of the uniform exponential decay of the exponentials appearing in the definitions of $Q_1^R, Q_7^R, Q_1^L, Q_7^L, P_1^R, P_1^L$, together with the $\mathcal{A}$- and $\mathcal{B}$-symmetries.
\end{proof}

Due to Lemma \ref{II v3lemma} and the symmetries \eqref{njsymm}, it is sufficient to construct local parametrices approximating $n^{(4)}$ near $\omega k_{4}$ and $\omega^{2}k_{2}$.

\section{Local parametrix near $\omega k_{4}$}\label{localparametrixsec1}

As $k \to \omega k_{4}$, we have
\begin{align*}
& -(\Phi_{31}(\zeta,k)-\Phi_{31}(\zeta,\omega k_{4})) =  \Phi_{31,\omega k_{4}}(k-\omega k_{4})^{2} + O((k-\omega k_{4})^{3}), \quad \Phi_{31,\omega k_{4}} := \omega \frac{4-3 k_{4} \zeta - k_{4}^{3} \zeta}{4 k_{4}^{4}}.
\end{align*}
We define $z_{1}=z_{1}(\zeta,t,k)$ by
\begin{align*}
z_{1}=z_{1,\star}\sqrt{t} (k-\omega k_{4}) \hat{z}_{1}, \qquad \hat{z}_{1} := \sqrt{\frac{2i(\Phi_{31}(\zeta,\omega k_{4})-\Phi_{31}(\zeta,k))}{z_{1,\star}^{2}(k-\omega k_{4})^{2}}},
\end{align*}
where the principal branch is chosen for $\hat{z}_{1}=\hat{z}_{1}(\zeta,k)$, and 
\begin{align*}
z_{1,\star} := \sqrt{2} e^{\frac{\pi i}{4}} \sqrt{\Phi_{31,\omega k_{4}}}, \qquad -i\omega k_{4}z_{1,\star}>0.
\end{align*}
Note that $\hat{z}_{1}(\zeta,\omega k_{4})=1$, and $t(\Phi_{31}(\zeta, k) - \Phi_{31}(\zeta,\omega k_{4})) = \frac{iz_{1}^{2}}{2}$. Let $\epsilon>0$ be small and independent of $t$. The map $z_{1}$ is conformal from $D_\epsilon(\omega k_4)$ to a neighborhood of $0$ and its expansion as $k \to \omega k_{4}$ is given by
\begin{align*}
z_{1} = z_{1,\star}\sqrt{t}(k-\omega k_{4})(1+O(k-\omega k_{4})).
\end{align*}
For all $k \in D_\epsilon(\omega k_4)$, we have
\begin{align*}
\ln_{\omega k_{4}}(k-\omega k_{4}) & = \ln_{0}[z_{1,\star}(k-\omega k_{4})\hat{z}_{1}]- \ln \hat{z}_{1} -\ln z_{1,\star}, \\
\tilde{\ln}_{\omega k_{4}}(k-\omega k_{4}) & = \ln[z_{1,\star}(k-\omega k_{4})\hat{z}_{1}]- \ln \hat{z}_{1} -\ln z_{1,\star},
\end{align*}
where $\ln_{0}(k):= \ln|k|+i\arg_{0}k$, $\arg_{0}(k)\in (0,2\pi)$, and $\ln$ is the principal logarithm. Shrinking $\epsilon > 0$ if necessary, $k \mapsto \ln \hat{z}_{1}$ is analytic in $D_{\epsilon}(\omega k_{4})$, $\ln \hat{z}_{1} = O(k-\omega k_{4})$ as $k \to \omega k_{4}$, and
\begin{align*}
\ln z_{1,\star} = \ln |z_{1,\star}| + i \arg z_{1,\star} = \ln |z_{1,\star}| + i \big( \tfrac{\pi}{2}-\arg (\omega k_{4}) \big),
\end{align*}
where $\arg (\omega k_{4}) \in (\tfrac{\pi}{2},\tfrac{2\pi}{3})$. By Lemma \ref{II deltalemma}, we have
\begin{align*}
& \delta_{1}(\zeta,k) = e^{-i\nu_{1}\ln_{\omega k_{4}}(k-\omega k_{4})}e^{-\chi_{1}(\zeta,k)} = e^{-i\nu_{1} \ln_{0} z_{1}} t^{\frac{i \nu_{1}}{2}} e^{i \nu_{1} \ln \hat{z}_{1}} e^{i \nu_{1} \ln z_{1,\star}} e^{-\chi_{1}(\zeta,k)}, \\
& \delta_{2}(\zeta,k) = e^{-i\nu_{1} \ln z_{1}} t^{\frac{i \nu_{1}}{2}} e^{i \nu_{1} \ln \hat{z}_{1}} e^{i \nu_{1} \ln z_{1,\star}} e^{i\nu_{2}\tilde{\ln}_{\omega^{2}k_{2}}(k-\omega^{2}k_{2})-\tilde{\chi}_{2}(\zeta,k)}, \\
& \delta_{3}(\zeta,k) = e^{-i\nu_{3} \ln z_{1}} t^{\frac{i \nu_{3}}{2}} e^{i \nu_{3} \ln \hat{z}_{1}} e^{i \nu_{3} \ln z_{1,\star}} e^{i\nu_{4}\tilde{\ln}_{\omega^{2}k_{2}}(k-\omega^{2}k_{2})-\tilde{\chi}_{3}(\zeta,k)}.
%	\\
%& \delta_{4}(\zeta,k) = e^{-i\nu_{4}\tilde{\ln}_{\omega^{2}k_{2}}(k-\omega^{2}k_{2}) -\tilde{\chi}_{4}(\zeta,k)}, \\
%& \delta_{5}(\zeta,k) = e^{-i \nu_{5} \tilde{\ln}_{\omega^{2} k_{2}}(k-\omega^{2} k_{2}) - \tilde{\chi}_{5}(\zeta,k)}.
\end{align*}
For simplicity, for $a \in \C$, we will write $z_{1,(0)}^{a} := e^{a \ln_{0} z_{1}}$, $z_{1,\star}^{a} := e^{a \ln z_{1,\star}}$, $\hat{z}_{1}^{a} := e^{a \ln \hat{z}_{1}}$, and $z_{1}^{a} := e^{a \ln z_{1}}$. Using the above expressions for $\{\delta_{j}(\zeta,k)\}_{j=1}^{3}$, we obtain
\begin{align} \label{II Delta33 Delta11}
& \frac{\Delta_{33}(\zeta,k)}{\Delta_{11}(\zeta,k)} = \frac{\delta_1(\zeta, k)\delta_2(\zeta, k)}{\delta_3(\zeta, k)^2}\mathcal{D}_1(k) \frac{\mathcal{P}(k)}{\mathcal{P}(\omega k)}
= d_{1,0}(\zeta,t)d_{1,1}(\zeta,k)z_{1}^{2i \nu_{3}} z_{1,(0)}^{-i \nu_{1}} z_{1}^{- i \nu_{1}},
\end{align}
where $d_{1,0}(\zeta,t)$ is defined in (\ref{def of d10}) and
\begin{align}\nonumber
& d_{1,1}(\zeta,k) := e^{-\chi_{1}(\zeta,k)+\chi_{1}(\zeta,\omega k_{4})-\tilde{\chi}_{2}(\zeta,k)+\tilde{\chi}_{2}(\zeta,\omega k_{4})+2\tilde{\chi}_{3}(\zeta,k)-2\tilde{\chi}_{3}(\zeta,\omega k_{4})}  
	\\ \nonumber
& \hspace{1.6cm} \times e^{i(\nu_{2}-2\nu_{4})\big[\tilde{\ln}_{\omega^{2}k_{2}}(k - \omega^{2}k_{2})-\tilde{\ln}_{\omega^{2}k_{2}}(\omega k_{4}-\omega^{2}k_{2})\big]} \hat{z}_{1}^{2i(\nu_{1}-\nu_{3})} \frac{\mathcal{D}_{1}(k)}{\mathcal{D}_{1}(\omega k_{4})} \frac{\mathcal{P}(k)}{\mathcal{P}(\omega k)},  
\end{align}
We summarize in Lemma \ref{II lemma: bound on Y} below some properties of $d_{1,0}(\zeta, t)$ and $d_{1,0}(\zeta, k)$. 
Define
\begin{align*}
Y_{1}(\zeta,t) = d_{1,0}(\zeta,t)^{\frac{\sigma}{2}}e^{-\frac{t}{2}\Phi_{31}(\zeta,\omega k_{4})\sigma}\lambda_{1}^{\sigma},
\end{align*}
where $\sigma = \diag (1,0,-1)$ and $\lambda_{1}\in \mathbb{C}$ is a parameter that will be chosen later. For $k \in D_\epsilon(\omega k_4)$, define
\begin{align}\label{lol4}
\tilde{n}(x,t,k) = n^{(4)}(x,t,k)Y_{1}(\zeta,t).
\end{align}
Let $\tilde{v}$ be the jump matrix of $\tilde{n}$, and let $\tilde{v}_{j}$ denote the restriction of $\tilde{v}$ to $\Gamma_{j}^{(4)}\cap D_\epsilon(\omega k_4)$. In view of \eqref{IVv2def}, \eqref{II def of mp3p} and \eqref{II Delta33 Delta11}, the matrices $\tilde{v}_{j}$ can be written as
\begin{subequations}\label{vtilde Sector I}
\begin{align}
& \tilde{v}_1 = \begin{pmatrix}
1 & 0 & 0 \\
0 & 1 & 0 \\
d_{1,1}^{-1}\lambda_{1}^{2}r_{2,a}(\frac{1}{\omega^{2} k})z_{1}^{-2i \nu_{3}} z_{1,(0)}^{i \nu_{1}} z_{1}^{i \nu_{1}}e^{\frac{iz_{1}^{2}}{2}} & 0 & 1
\end{pmatrix}, 
	\\
& \tilde{v}_4 = \begin{pmatrix}
1 & 0 & -d_{1,1}\lambda_{1}^{-2} \big(\frac{r_{1}(\frac{1}{\omega^{2}k})}{f(k)}\big)_{a} z_{1}^{2i \nu_{3}} z_{1,(0)}^{-i \nu_{1}} z_{1}^{-i \nu_{1}}e^{-\frac{iz_{1}^{2}}{2}} \\
0 & 1 & 0 \\
0 & 0 & 1
\end{pmatrix}, 
	\\
& \tilde{v}_6 = \begin{pmatrix}
1 & 0 & 0 \\
0 & 1 & 0 \\
-d_{1,1}^{-1}\lambda_{1}^{2}\big(\frac{r_{2}(\frac{1}{\omega^{2} k})}{f(k)}\big)_{a} z_{1}^{-2i \nu_{3}} z_{1,(0)}^{i \nu_{1}} z_{1}^{i \nu_{1}}e^{\frac{iz_{1}^{2}}{2}} & 0 & 1 
\end{pmatrix}, 
	\\
& \tilde{v}_3 = \begin{pmatrix}
1 & 0 & d_{1,1}\lambda_{1}^{-2} r_{1,a}(\frac{1}{\omega^{2}k}) z_{1}^{2i \nu_{3}} z_{1,(0)}^{-i \nu_{1}} z_{1}^{-i \nu_{1}}e^{-\frac{iz_{1}^{2}}{2}} \\
0 & 1 & 0 \\
0 & 0 & 1
\end{pmatrix}.
\end{align}
\end{subequations}
Let $\mathcal{X}_{1}^{\epsilon}:= \cup_{j=1}^{4}\mathcal{X}_{1,j}^{\epsilon}$, where
\begin{align*}
& \mathcal{X}_{1,1}^{\epsilon} = \Gamma_1^{(4)}\cap D_{\epsilon}(\omega k_{4}), & & \mathcal{X}_{1,2}^{\epsilon} = \Gamma_4^{(4)}\cap D_{\epsilon}(\omega k_{4}), & & \mathcal{X}_{1,3}^{\epsilon} = \Gamma_6^{(4)}\cap D_{\epsilon}(\omega k_{4}), & & \mathcal{X}_{1,4}^{\epsilon} = \Gamma_3^{(4)}\cap D_{\epsilon}(\omega k_{4}).
\end{align*}
Let $k_{\star}=\omega k_{4}$.
The above expressions for $\tilde{v}_j$ suggest that we approximate $\tilde{n}$ by $(1,1,1)Y_1 \tilde{m}^{\omega k_{4}}$, where $\tilde{m}^{\omega k_{4}}(x,t,k)$ is analytic for $k \in D_{\epsilon}(\omega k_{4}) \setminus \mathcal{X}_{1}$,
obeys the jump relation $\tilde{m}^{\omega k_{4}}_+ = \tilde{m}^{\omega k_{4}}_- \tilde{v}_{\mathcal{X}_{1}^{\epsilon}}^{\omega k_{4}}$ on $\mathcal{X}_{1}^{\epsilon}$ where
\begin{align*}
& \tilde{v}_{\mathcal{X}_{1,1}^{\epsilon}}^{\omega k_{4}}=\begin{pmatrix}
1 & 0 & 0 \\
0 & 1 & 0 \\
\frac{\mathcal{P}(\omega^{2}k_{4})}{\mathcal{P}(\omega k_{4})}\lambda_{1}^{2}r_{2}(\frac{1}{\omega^{2} k_{\star}})z_{1}^{-2i \nu_{3}} z_{1,(0)}^{i \nu_{1}} z_{1}^{i \nu_{1}}e^{\frac{iz_{1}^{2}}{2}} & 0 & 1
\end{pmatrix}, \\
& \tilde{v}_{\mathcal{X}_{1,2}^{\epsilon}}^{\omega k_{4}}=\begin{pmatrix}
1 & 0 & -\frac{\mathcal{P}(\omega k_{4})}{\mathcal{P}(\omega^{2} k_{4})}\lambda_{1}^{-2} \frac{r_{1}(\frac{1}{\omega^{2}k_{\star}})}{f(k_{\star})}z_{1}^{2i \nu_{3}} z_{1,(0)}^{-i \nu_{1}} z_{1}^{-i \nu_{1}}e^{-\frac{iz_{1}^{2}}{2}} \\
0 & 1 & 0 \\
0 & 0 & 1
\end{pmatrix}, \\
& \tilde{v}_{\mathcal{X}_{1,3}^{\epsilon}}^{\omega k_{4}} = \begin{pmatrix}
1 & 0 & 0 \\
0 & 1 & 0 \\
-\frac{\mathcal{P}(\omega^{2}k_{4})}{\mathcal{P}(\omega k_{4})}\lambda_{1}^{2}\frac{r_{2}(\frac{1}{\omega^{2} k_{\star}})}{f(k_{\star})}z_{1}^{-2i \nu_{3}} z_{1,(0)}^{i \nu_{1}} z_{1}^{i \nu_{1}}e^{\frac{iz_{1}^{2}}{2}} & 0 & 1 
\end{pmatrix}, \\
& \tilde{v}_{\mathcal{X}_{1,4}^{\epsilon}}^{\omega k_{4}} = \begin{pmatrix}
1 & 0 & \frac{\mathcal{P}(\omega k_{4})}{\mathcal{P}(\omega^{2} k_{4})}\lambda_{1}^{-2} r_{1}(\frac{1}{\omega^{2}k_{\star}}) z_{1}^{2i \nu_{3}} z_{1,(0)}^{-i \nu_{1}} z_{1}^{-i \nu_{1}}e^{-\frac{iz_{1}^{2}}{2}} \\
0 & 1 & 0 \\
0 & 0 & 1
\end{pmatrix},
\end{align*}
and satisfies $\tilde{m}^{\omega k_{4}} \to I$ as $t \to \infty$ uniformly for $k \in \partial D_{\epsilon}(\omega k_{4})$. 
We choose our free parameter $\lambda_{1}$ as follows:
\begin{align}\label{lambda1def}
\lambda_{1} := \big(\tfrac{\mathcal{P}(\omega k_{4})}{\mathcal{P}(\omega^{2} k_{4})}\big)^{\frac{1}{2}} |\tilde{r}(\tfrac{1}{\omega^{2}k_{\star}})|^{-\frac{1}{4}} = \big(\tfrac{\mathcal{P}(\omega k_{4})}{\mathcal{P}(\omega^{2} k_{4})}\big)^{\frac{1}{2}} |\tilde{r}(\tfrac{1}{k_{4}})|^{-\frac{1}{4}},
\end{align}
and define the local parametrix $m^{\omega k_{4}}$ by
\begin{align}\label{II m omegak4 def}
m^{\omega k_{4}}(x,t,k) = Y_{1}(\zeta,t)m^{X,(1)}(q_{1},q_{3},z_{1}(\zeta,t,k))Y_{1}(\zeta,t)^{-1}, \qquad k \in D_{\epsilon}(\omega k_{4}),
\end{align}
where $m^{X,(1)}(\cdot) = m^{X,(1)}(q_{1},q_{3},\cdot)$ is the solution of the model RH problem of Lemma \ref{II Xlemma 3} with 
\begin{align*}
& q_{1} = \tilde{r}(\omega k_{4})^{\frac{1}{2}}r_{1}(\omega k_{4}), & & \bar{q}_{1} = \tilde{r}(\omega k_{4})^{-\frac{1}{2}}r_{2}(\omega k_{4}) = \tilde{r}(\omega k_{4})^{\frac{1}{2}}\overline{r_{1}(\omega k_{4})}, \\
& q_{3} = |\tilde{r}(\tfrac{1}{k_{4}})|^{\frac{1}{2}}r_{1}(\tfrac{1}{k_{4}}), & & \bar{q}_{3} = -|\tilde{r}(\tfrac{1}{k_{4}})|^{-\frac{1}{2}}r_{2}(\tfrac{1}{k_{4}}) = |\tilde{r}(\tfrac{1}{k_{4}})|^{\frac{1}{2}}\overline{r_{1}(\tfrac{1}{k_{4}})}.
\end{align*}

\begin{lemma}\label{II lemma: bound on Y}
The function $Y_{1}(\zeta,t)$ is uniformly bounded:
\begin{align}\label{II Ybound}
\sup_{\zeta \in \mathcal{I}} \sup_{t \geq 2} | Y_{1}(\zeta,t)^{\pm 1}| < C.
\end{align}
Moreover, the functions $d_{1,0}(\zeta, t)$ and $d_{1,1}(\zeta, k)$ satisfy
\begin{align}\label{II d0estimate}
& |d_{1,0}(\zeta, t)| = e^{-\pi \nu_{1}}, & &  \zeta \in \mathcal{I}, \ t \geq 2,
	\\ \label{II d1estimate}
& |d_{1,1}(\zeta, k) - \tfrac{\mathcal{P}(\omega k_{4})}{\mathcal{P}(\omega^{2} k_{4})}| \leq C |k - \omega k_4| (1+ |\ln|k-\omega k_4||), & & \zeta \in \mathcal{I}, \ k \in \mathcal{X}_{1}^{\epsilon}.
\end{align}
\end{lemma}
\begin{proof}
Standard estimates give (\ref{II Ybound}) and (\ref{II d1estimate}). To prove (\ref{II d0estimate}), we first note that
\begin{align}\nonumber
 |d_{1,0}(\zeta,t)| := &\; e^{-\re \chi_{1}(\zeta,\omega k_{4}) - \re \tilde{\chi}_{2}(\zeta,\omega k_{4}) + 2 \re\tilde{\chi}_{3}(\zeta,\omega k_{4}) }  
	\\ \label{d10absolutevalue}
& \times e^{-(\nu_{2}-2\nu_{4})\tilde{\arg}_{\omega^{2}k_{2}}(\omega k_{4}-\omega^{2}k_{2})} 
e^{-2(\nu_{1}-\nu_{3})\arg z_{1,\star}} |\mathcal{D}_{1}(\omega k_{4})|.
\end{align}
We deduce from Lemma \ref{II deltalemma} that if $k \in \partial \D \setminus \{\arg k \in (\pi/2, 2\pi/3)\}$ then
\begin{align*}
& |\delta_{1}(\zeta,k)| = e^{\nu_{1} \frac{\arg_i(k) + \arg(\omega k_4) + \pi}{2} - \re \chi_{1}(\zeta,k) }, 
	\\ 
& |\delta_{2}(\zeta,k)| = e^{\nu_{1} \frac{\arg_i(k) + \arg(\omega k_4) + \pi}{2} - \nu_{2} \frac{\arg_i(k) + \arg(\omega^2 k_2) + \pi}{2}- \re \chi_{2}(\zeta,k) }, 
	\\
& |\delta_{3}(\zeta,k)| = e^{\nu_{3} \frac{\arg_i(k) + \arg(\omega k_4) + \pi}{2}  - \nu_{4} \frac{\arg_i(k) + \arg(\omega^2 k_2) + \pi}{2} - \re \chi_{3}(\zeta,k) }, 
	\\
& |\delta_{4}(\zeta,k)| = e^{\nu_{4} \frac{\arg_i(k) + \arg(\omega^2 k_2) + \pi}{2} - \re \chi_{4}(\zeta,k) }, 
	\\
& |\delta_{5}(\zeta,k)| = e^{\nu_{5} \frac{\arg_i(k) + \arg(\omega^2 k_2) + \pi}{2} - \re \chi_{5}(\zeta,k) },
\end{align*}
and
\begin{align*}
\re \chi_{1}(\zeta,k) =&\; \frac{1}{2\pi} \int_{\Gamma_2^{(2)}}  \frac{\arg{s}}{2}  d\ln(1+r_1(s)r_{2}(s))
+ \frac{\arg_i{k}+ \pi}{2} \nu_1, 
	\\
\re \chi_{2}(\zeta,k) = &\; \frac{1}{2\pi} \int_{\Gamma_5^{(2)}}  \frac{\arg{s}}{2}d\ln(1+r_1(s)r_{2}(s))
+  \frac{\arg_i{k} + \pi}{2}(\nu_1 - \nu_2), 
	\\
\re \chi_{3}(\zeta,k) =&\; \frac{1}{2\pi} \int_{\Gamma_5^{(2)}}  \frac{\arg{s}}{2}d\ln(f(s)) + \frac{\arg_i{k} + \pi}{2}(\nu_3 - \nu_4), 
	\\
\re \chi_{4}(\zeta,k) = &\; \frac{1}{2\pi}\lim_{\epsilon \to 0_{+}} \bigg(  \int_{\omega^{2}k_{2}}^{e^{i(\frac{2\pi}{3}-\epsilon)}} \hspace{-0.4cm} \frac{\arg{s}}{2}d\ln(f(s)) - \frac{\arg{\omega}}{2} \ln(f(e^{i(\frac{2\pi}{3}-\epsilon)})) \bigg)
 + \frac{\arg_i{k} + \pi}{2} \nu_4, 
	\\
\re \chi_{5}(\zeta,k) = &\; \frac{1}{2\pi}\lim_{\epsilon \to 0_{+}} \bigg( \int_{\omega^{2}k_{2}}^{e^{i(\frac{2\pi}{3}-\epsilon)}} \hspace{-0.4cm} \frac{\arg{s}}{2}d\ln(f(\omega^{2}s)) - \frac{\arg{\omega}}{2} \ln(f(e^{-i\epsilon})) \bigg)
 + \frac{\arg_i{k} + \pi}{2} \nu_5,
\end{align*}
where $\arg_i(k) \in (\pi/2, 5\pi/2)$.
Hence, for $k \in \partial \D \setminus \{\arg k \in (\pi/2, 2\pi/3)\}$,
\begin{align}\nonumber
& |\delta_{1}(\zeta,k)| = e^{\nu_{1} \frac{\arg(\omega k_4)}{2} - \frac{1}{2\pi} \int_{\Gamma_2^{(2)}}  \frac{\arg{s}}{2}  d\ln(1+r_1(s)r_{2}(s))}, 
	\\ \nonumber
& |\delta_{2}(\zeta,k)| = e^{\nu_{1} \frac{\arg(\omega k_4)}{2} - \nu_{2} \frac{\arg(\omega^2 k_2)}{2}- \frac{1}{2\pi} \int_{\Gamma_5^{(2)}}  \frac{\arg{s}}{2}d\ln(1+r_1(s)r_{2}(s))},
	\\\nonumber
& |\delta_{3}(\zeta,k)| = e^{\nu_{3} \frac{\arg(\omega k_4)}{2}  - \nu_{4} \frac{\arg(\omega^2 k_2)}{2} - \frac{1}{2\pi} \int_{\Gamma_5^{(2)}}  \frac{\arg{s}}{2}d\ln(f(s)) },
	\\ \nonumber
& |\delta_{4}(\zeta,k)| = e^{\nu_{4} \frac{\arg(\omega^2 k_2)}{2} - \frac{1}{2\pi}\lim_{\epsilon \to 0_{+}} \big(  \int_{\omega^{2}k_{2}}^{e^{i(\frac{2\pi}{3}-\epsilon)}} \hspace{-0.1cm} \frac{\arg{s}}{2}d\ln(f(s)) - \frac{\arg{\omega}}{2} \ln(f(e^{i(\frac{2\pi}{3}-\epsilon)})) \big) }.
	\\ \label{absolutevaluedeltaj}
& |\delta_{5}(\zeta,k)| = e^{\nu_{5} \frac{\arg(\omega^2 k_2)}{2} - \frac{1}{2\pi}\lim_{\epsilon \to 0_{+}} \big( \int_{\omega^{2}k_{2}}^{e^{i(\frac{2\pi}{3}-\epsilon)}} \hspace{-0.1cm} \frac{\arg{s}}{2}d\ln(f(\omega^{2}s)) - \frac{\arg{\omega}}{2} \ln(f(e^{-i\epsilon})) \big) }.
\end{align}
Note that the right-hand sides in (\ref{absolutevaluedeltaj}) are independent of $k$. 
Since the points $k_4, \omega^2 k_4$, $\frac{1}{k_4}$, $\frac{1}{\omega k_4}$, $\frac{1}{\omega^2 k_4}$ all lie in $\partial \D \setminus \{\arg k \in (\pi/2, 2\pi/3)\}$, we obtain
\begin{align*}
|\mathcal{D}_{1}(\omega k_4)| = &\; 
\frac{e^{ \frac{1}{2\pi} \int_{\Gamma_2^{(2)}}  \frac{\arg{s}}{2}  d\ln(1+r_1(s)r_{2}(s)) - \nu_{1} \frac{\arg(\omega k_4)}{2}} e^{2\nu_{3} \frac{\arg(\omega k_4)}{2}  - 2\nu_{4} \frac{\arg(\omega^2 k_2)}{2} - 2\frac{1}{2\pi} \int_{\Gamma_5^{(2)}}  \frac{\arg{s}}{2}d\ln(f(s)) }}{ e^{\nu_{1} \frac{\arg(\omega k_4)}{2} - \nu_{2} \frac{\arg(\omega^2 k_2)}{2}- \frac{1}{2\pi} \int_{\Gamma_5^{(2)}}  \frac{\arg{s}}{2}d\ln(1+r_1(s)r_{2}(s))}}.
\end{align*}
Similar arguments show that
\begin{align*}
\re \chi_{1}(\zeta,\omega k_{4}) & = \frac{1}{2\pi} \int_{\Gamma_2^{(2)}} \frac{\arg s}{2} d\ln(1+r_1(s)r_{2}(s)) + \frac{\arg(\omega k_{4}) + \pi}{2} \nu_1,
	\\
\re \tilde{\chi}_{2}(\zeta,\omega k_{4}) & = \frac{1}{2\pi} \int_{\Gamma_5^{(2)}} \frac{\arg s}{2} d\ln(1+r_1(s)r_{2}(s))  +  \frac{\arg(\omega k_{4}) - \pi}{2}(\nu_1 - \nu_2),
	\\
\re \tilde{\chi}_{3}(\zeta,\omega k_{4}) & = \frac{1}{2\pi} \int_{\Gamma_5^{(2)}} \frac{\arg s}{2} d\ln(f(s)) + \frac{\arg(\omega k_{4}) - \pi}{2}(\nu_3 - \nu_4).
\end{align*}
Plugging the above expressions into (\ref{d10absolutevalue}) and using that 
\begin{align*}
\tilde{\arg}_{\omega^{2}k_{2}}(\omega k_{4}-\omega^{2}k_{2})
= \frac{\arg(\omega k_4) + \arg(\omega^{2}k_{2}) - \pi}{2} \quad \mbox{ and } \quad \arg z_{1,\star} = \tfrac{\pi}{2}-\arg (\omega k_{4}),
\end{align*}
straightforward calculations show that $|d_{1,0}(\zeta, t)| = e^{-\pi \nu_1}$.
\end{proof}

\begin{lemma}\label{II k0lemma}
The function $m^{\omega k_4}(x,t,k)$ defined in \eqref{II m omegak4 def} is an analytic function of $k \in D_\epsilon(\omega k_4) \setminus \mathcal{X}_{1}^\epsilon$ which is uniformly bounded for $t \geq 2$, $\zeta \in \mathcal{I}$, and $k \in D_\epsilon(\omega k_4) \setminus \mathcal{X}_{1}^\epsilon$.
Across $\mathcal{X}_{1}^\epsilon$, $m^{\omega k_4}$ obeys the jump condition $m_+^{\omega k_4} =  m_-^{\omega k_4} v^{\omega k_4}$, where $v^{\omega k_4}$ satisfies
\begin{align}\label{II v3vk0estimate}
\begin{cases}
 \| v^{(4)} - v^{\omega k_4} \|_{L^1(\mathcal{X}_{1}^\epsilon)} \leq C t^{-1} \ln t,
	\\
\| v^{(4)} - v^{\omega k_4} \|_{L^\infty(\mathcal{X}_{1}^\epsilon)} \leq C t^{-1/2} \ln t,
\end{cases} \qquad \zeta \in \mathcal{I}, \ t \geq 2.
\end{align}
Furthermore, as $t \to \infty$,
\begin{align}\label{II mmodmuestimate2}
& \| m^{\omega k_4}(x,t,\cdot)^{-1} - I \|_{L^\infty(\partial D_\epsilon(\omega k_4))} = O(t^{-1/2}),
	\\ \label{II mmodmuestimate1}
& m^{\omega k_4} - I = \frac{Y_{1}(\zeta, t)m_1^{X,(1)} Y_{1}(\zeta, t)^{-1}}{z_{1,\star}\sqrt{t} (k-\omega k_4) \hat{z}_{1}(\zeta,k)} + O(t^{-1})
\end{align}
uniformly for $\zeta \in \mathcal{I}$, and $m_1^{X,(1)} = m_1^{X,(1)}(q_{1},q_{3})$ is given by \eqref{II mXasymptotics 3}.
\end{lemma}
\begin{proof}
By \eqref{lol4} and \eqref{II m omegak4 def}, we have
$$v^{(4)} - v^{\omega k_4} = Y_{1}(\zeta, t) (\tilde{v} - v^{X,(1)})Y_{1}(\zeta,t)^{-1},
$$
where we recall that $\tilde{v}$ is given by \eqref{vtilde Sector I}. 
Thus, recalling (\ref{II Ybound}), the bounds (\ref{II v3vk0estimate}) follow if we can show that
\begin{subequations}\label{II v3k0}
\begin{align}\label{II v3k0L1}
  &\| \tilde{v}(x,t,\cdot) - v^{X,(1)}(x,t,z_{1}(\zeta,t,\cdot)) \|_{L^1(\mathcal{X}_{1,j}^\epsilon)} \leq C t^{-1} \ln t,
	\\ \label{II v3k0Linfty}
&  \| \tilde{v}(x,t,\cdot) - v^{X,(1)}(x,t,z_{1}(\zeta,t,\cdot)) \|_{L^\infty(\mathcal{X}_{1,j}^\epsilon)} \leq C t^{-1/2} \ln t,
\end{align}
\end{subequations}
for $j = 1, \dots, 4$. We give the proof of (\ref{II v3k0}) for $j = 1$; similar arguments apply when $j = 2,3,4$.

For $k \in \mathcal{X}_{1,1}^\epsilon$, only the $(31)$ element of the matrix $\tilde{v} - v^{X,1}$ is nonzero. Since $-\bar{q}_{3}=\tfrac{\mathcal{P}(\omega^{2} k_{4})}{\mathcal{P}(\omega k_{4})}\lambda_{1}^{2}r_2(\frac{1}{k_{4}})$ and
$$
\big| e^{\frac{iz_{1}^2}{2}} \big| = e^{\re (\frac{iz_{1}^2}{2})} = e^{-\frac{|z_{1}|^2}{2}},
$$
we can estimate $|(\tilde{v} - v^{X,(1)})_{31}|$ as follows:
\begin{align*}
|(\tilde{v} - & v^{X,(1)})_{31}| =  \bigg|d_{1,1}^{-1}\lambda_{1}^{2}r_{2,a}(\tfrac{1}{\omega^{2} k})z_{1}^{-2i \nu_{3}} z_{1,(0)}^{i \nu_{1}} z_{1}^{i \nu_{1}}e^{\frac{iz_{1}^{2}}{2}}
- (-\bar{q}_{3} z_{1}^{-2i \nu_{3}} z_{1,(0)}^{i \nu_{1}} z_{1}^{i \nu_{1}}e^{\frac{iz_{1}^{2}}{2}})\bigg|
	\\
=& \;  \big| z_{1}^{-2i \nu_{3}} z_{1,(0)}^{i \nu_{1}} z_{1}^{i \nu_{1}} \big| \Big|  (d_{1,1}^{-1}-\tfrac{\mathcal{P}(\omega^{2} k_{4})}{\mathcal{P}(\omega k_{4})}) \lambda_{1}^{2}r_{2,a}(\tfrac{1}{\omega^{2} k}) +\big(\tfrac{\mathcal{P}(\omega^{2} k_{4})}{\mathcal{P}(\omega k_{4})}\lambda_{1}^{2}r_{2,a}(\tfrac{1}{\omega^{2} k})
- (-\bar{q}_{3})\big)\Big| \big| e^{\frac{iz_{1}^2}{2}} \big|
	\\
\leq &\; C \bigg(|d_{1,1}^{-1} - \tfrac{\mathcal{P}(\omega^{2} k_{4})}{\mathcal{P}(\omega k_{4})}||\lambda_{1}^{2}r_{2,a}(\tfrac{1}{\omega^{2} k})| + |\tfrac{\mathcal{P}(\omega^{2} k_{4})}{\mathcal{P}(\omega k_{4})}\lambda_{1}^{2}r_{2,a}(\tfrac{1}{\omega^{2} k})
- (-\bar{q}_{3}) | \bigg)e^{-\frac{|z_{1}|^2}{2}}
	\\
\leq &\; C \bigg(|d_{1,1}^{-1} - \tfrac{\mathcal{P}(\omega^{2} k_{4})}{\mathcal{P}(\omega k_{4})}| + |k-\omega k_4| \bigg) e^{\frac{t}{4}|\re \Phi_{31}(\zeta, k)|} e^{-\frac{|z_{1}|^2}{2}}
	\\
\leq &\; C \bigg(|d_{1,1}^{-1} - \tfrac{\mathcal{P}(\omega^{2} k_{4})}{\mathcal{P}(\omega k_{4})}| + |k-\omega k_4| \bigg)e^{-c t |k-\omega k_4|^2}, \qquad k \in \mathcal{X}_{1,1}^\epsilon.
\end{align*}
Using (\ref{II d1estimate}), this gives
\begin{align*}
|(\tilde{v} - v^{X,(1)})_{31}| \leq C |k-\omega k_4|(1+ |\ln|k-\omega k_4||)e^{-c t |k-\omega k_4|^2}, \qquad k \in \mathcal{X}_{1,1}^\epsilon.
\end{align*}
Hence
\begin{align*}
& \|(\tilde{v} - v^{X,(1)})_{31}\|_{L^1(\mathcal{X}_{1,1}^\epsilon)}
\leq C \int_0^\infty u(1+ |\ln u|) e^{-ctu^2} du \leq Ct^{-1}\ln t,
	\\
& \|(\tilde{v} - v^{X,(1)})_{31}\|_{L^{\infty}(\mathcal{X}_{1,1}^\epsilon)}
\leq C \sup_{u \geq 0} u(1+ |\ln u|) e^{-ctu^2} \leq Ct^{-1/2}\ln t.
\end{align*}

The variable $z_{1}$ tends to infinity as $t \to \infty$ if $k \in \partial D_\epsilon(\omega k_4)$, because
$$|z_{1}| = |z_{1,\star}|\sqrt{t} |k-\omega k_{4}| \, | \hat{z}_{1}| \geq c \sqrt{t}.$$
Thus equation (\ref{II mXasymptotics 3}) yields
\begin{align*}
& m^{X,(1)}(q_{1},q_{3}, z_{1}(\zeta,t, k)) = I + \frac{m_1^{X,(1)}(q_{1},q_{3})}{z_{1,\star}\sqrt{t} (k-\omega k_{4}) \hat{z}_{1}(\zeta,k)} + O(t^{-1}),
\qquad  t \to \infty,
\end{align*}
uniformly with respect to $k \in \partial D_\epsilon(\omega k_4)$ and $\zeta \in \mathcal{I}$. 
Recalling the definition (\ref{II m omegak4 def}) of $m^{\omega k_4}$, this gives
\begin{align}\label{II mmodmmuI}
 & m^{\omega k_4} - I = \frac{Y_{1}(\zeta, t)m_1^{X,(1)}(q_1, q_3) Y_{1}(\zeta, t)^{-1}}{z_{1,\star}\sqrt{t} (k-\omega k_{4}) \hat{z}_{1}(\zeta,k)} + O(t^{-1}), \qquad  t \to \infty,
\end{align}
uniformly for $k \in \partial D_\epsilon(\omega k_4)$ and $\zeta \in \mathcal{I}$. This proves (\ref{II mmodmuestimate2}) and (\ref{II mmodmuestimate1}).
%Since $\hat{z}_{1}(\zeta,\omega k_{4})=1$, equation (\ref{II mmodmuestimate1}) follows from (\ref{II mmodmmuI}) and Cauchy's formula (recall that $\partial D_{\epsilon}(k_{0})$ is oriented negatively). 
%(note that $\partial D_{\epsilon}(k_{0})$ is oriented positively, see Figure \ref{II Gammahat.pdf}).
\end{proof}

\section{Local parametrix near $\omega^{2}k_{2}$}\label{localparametrixsec2}

As $k \to \omega^{2} k_{2}$, we have
\begin{align*}
& -(\Phi_{32}(\zeta,k)-\Phi_{32}(\zeta,\omega^{2} k_{2})) =  \Phi_{32,\omega^{2} k_{2}}(k-\omega^{2} k_{2})^{2} + O((k-\omega^{2} k_{2})^{3}), \\
& \Phi_{32,\omega^{2} k_{2}} := -\omega^{2} \frac{4-3 k_{2} \zeta - k_{2}^{3} \zeta}{4 k_{2}^{4}}.
\end{align*}
We define $z_{2}=z_{2}(\zeta,t,k)$ by
\begin{align*}
z_{2}=z_{2,\star}\sqrt{t} (k-\omega^{2} k_{2}) \hat{z}_{2}, \qquad \hat{z}_{2} := \sqrt{\frac{2i(\Phi_{32}(\zeta,\omega^{2} k_{2})-\Phi_{32}(\zeta,k))}{z_{2,\star}^{2}(k-\omega^{2} k_{2})^{2}}},
\end{align*}
where the principal branch is chosen for $\hat{z}_{2}=\hat{z}_{2}(\zeta,k)$, and
\begin{align*}
z_{2,\star} := \sqrt{2} e^{\frac{\pi i}{4}} \sqrt{\Phi_{32,\omega^{2} k_{2}}}, \qquad -i\omega^{2} k_{2}z_{2,\star}>0.
\end{align*}
Note that $\hat{z}_{2}(\zeta,\omega^{2} k_{2})=1$, and $t(\Phi_{32}(\zeta,k) - \Phi_{32}(\zeta,\omega^{2} k_{2})) = \frac{iz_{2}^{2}}{2}$. Provided $\epsilon>0$ is chosen small enough, the map $z_{2}$ is conformal from $D_\epsilon(\omega^{2} k_2)$ to a neighborhood of $0$ and its expansion as $k \to \omega^{2} k_{2}$ is given by
\begin{align*}
z_{2} = z_{2,\star}\sqrt{t}(k-\omega^{2} k_{2})(1+O(k-\omega^{2} k_{2})).
\end{align*}
For all $k \in D_\epsilon(\omega^{2} k_2)$, we have
\begin{align*}
\ln_{\omega^{2} k_2}(k-\omega^{2} k_2) & = \ln_{0}[z_{2,\star}(k-\omega^{2} k_2)\hat{z}_{2}]- \ln \hat{z}_{2} -\ln z_{2,\star}, \\
\tilde{\ln}_{\omega^{2} k_2}(k-\omega^{2} k_2) & = \ln[z_{2,\star}(k-\omega^{2} k_2)\hat{z}_{2}]- \ln \hat{z}_{2} -\ln z_{2,\star},
\end{align*}
where $\ln_{0}(k):= \ln|k|+i\arg_{0}k$, $\arg_{0}(k)\in (0,2\pi)$, and $\ln$ is the principal logarithm. Shrinking $\epsilon > 0$ if necessary, $k \mapsto \ln \hat{z}_{2}$ is analytic in $D_{\epsilon}(\omega^{2} k_2)$, $\ln \hat{z}_{2} = O(k-\omega^{2} k_2)$ as $k \to \omega^{2} k_2$, and
\begin{align*}
\ln z_{2,\star} = \ln |z_{2,\star}| + i \arg z_{2,\star} = \ln |z_{2,\star}| + i \big( \tfrac{\pi}{2}-\arg (\omega^{2} k_2) \big),
\end{align*}
where $\arg (\omega^{2} k_2) \in (\tfrac{\pi}{2},\tfrac{2\pi}{3})$. By Lemma \ref{II deltalemma}, we have
\begin{align*}
%& \delta_{1}(\zeta,k) = e^{-i\nu_{1}\ln_{\omega k_{4}}(k-\omega k_{4})}e^{-\chi_{1}(\zeta,k)}, \\
& \delta_{2}(\zeta,k) = e^{-i\nu_{1}\ln_{\omega k_{4}}(k-\omega k_{4})}e^{-\chi_{2}(\zeta,k)}t^{-\frac{i\nu_{2}}{2}}z_{2,(0)}^{i\nu_{2}}\hat{z}_{2}^{-i\nu_{2}}z_{2,\star}^{-i\nu_{2}}, \\
& \delta_{3}(\zeta,k) = e^{-i\nu_{3}\ln_{\omega k_{4}}(k-\omega k_{4})}e^{-\chi_{3}(\zeta,k)}t^{-\frac{i\nu_{4}}{2}}z_{2,(0)}^{i\nu_{4}}\hat{z}_{2}^{-i\nu_{4}}z_{2,\star}^{-i\nu_{4}}, \\
& \delta_{4}(\zeta,k) = t^{\frac{i\nu_{4}}{2}} z_{2}^{-i\nu_{4}} \hat{z}_{2}^{i\nu_{4}} z_{2,\star}^{i\nu_{4}}e^{-\tilde{\chi}_{4}(\zeta,k)}, \\
& \delta_{5}(\zeta,k) = t^{\frac{i\nu_{5}}{2}}z_{2}^{-i\nu_{5}}\hat{z}_{2}^{i\nu_{5}}z_{2,\star}^{i\nu_{5}}e^{-\tilde{\chi}_{5}(\zeta,k)},
\end{align*}
where we have used the notation $z_{2,(0)}^{a} :=  e^{a \ln_{0} z_{2}}$, $\hat{z}_{2}^{a} :=  e^{a \ln \hat{z}_{2}}$, $z_{2,\star}^{a} :=  e^{a \ln z_{2,\star}}$,
and $z_{2}^{a} := e^{a \ln z_{2}}$ for $a \in \C$. We obtain
\begin{align}\label{II Delta33 Delta22}
& \frac{\Delta_{33}(\zeta,k)}{\Delta_{22}(\zeta,k)} = \frac{\delta_2(\zeta, k)^2 \delta_4(\zeta, k)}{\delta_3(\zeta, k)\delta_5(\zeta, k)^2 }\mathcal{D}_2(k)\frac{\mathcal{P}(k)}{\mathcal{P}(\omega^{2} k)}
= d_{2,0}(\zeta,t)d_{2,1}(\zeta,k)z_{2}^{i (2\nu_{5}-\nu_{4})} z_{2,(0)}^{i (2\nu_{2}-\nu_{4})},  
\end{align}
where $d_{2,0}(\zeta,t)$ is defined in (\ref{def of d20}) and
\begin{align}
& d_{2,1}(\zeta,k) := e^{-2\chi_{2}(\zeta,k)+2\chi_{2}(\zeta,\omega^{2} k_{2}) + \chi_{3}(\zeta,k)-\chi_{3}(\zeta,\omega^{2} k_{2}) - \tilde{\chi}_{4}(\zeta,k)+\tilde{\chi}_{4}(\zeta,\omega^{2} k_{2}) + 2\tilde{\chi}_{5}(\zeta,k)-2\tilde{\chi}_{5}(\zeta,\omega^{2} k_{2}) } 
	\nonumber \\
& \hspace{1.6cm} \times e^{i(\nu_{3}-2\nu_{1})\big[\ln_{\omega k_{4}}(k-\omega k_{4})-\ln_{\omega k_{4}}(\omega^{2} k_{2}-\omega k_{4})\big]} \hat{z}_{2}^{2i(\nu_{4}-\nu_{5}-\nu_{2})} \frac{\mathcal{D}_{2}(k)}{\mathcal{D}_{2}(\omega^{2} k_{2})}\frac{\mathcal{P}(k)}{\mathcal{P}(\omega^{2} k)}. 
\end{align}
We state in Lemma \ref{II lemma: bound on Y green} below some properties of $d_{2,0}(\zeta, t)$ and $d_{2,1}(\zeta, k)$. Define
\begin{align*}
Y_{2}(\zeta,t) = d_{2,0}(\zeta,t)^{\frac{ \tilde{\sigma}}{2}}e^{-\frac{t}{2}\Phi_{32}(\zeta,\omega^{2}k_{2})\tilde{\sigma}}\lambda_{2}^{\tilde{\sigma}},
\end{align*}
where $\tilde{\sigma} := \diag (0,1,-1)$ and $\lambda_{2} \in \mathbb{C}$ is a parameter that will be chosen below. For $k \in D_\epsilon(\omega^{2} k_2)$, define
\begin{align*}
\tilde{n}(x,t,k) = n^{(4)}(x,t,k)Y_{2}(\zeta,t).
\end{align*}
Let $\tilde{v}$ be the jump matrix of $\tilde{n}$, and let $\tilde{v}_{j}$ be $\tilde{v}$ on $\Gamma_{j}^{(4)}\cap D_{\epsilon}(\omega^{2}k_{2})$. By \eqref{IVv2def}, \eqref{II def of mp3p} and \eqref{II Delta33 Delta22}, the matrices $\tilde{v}_{j}$ can be written as
\begin{align*}
& \tilde{v}_7^{-1} = \begin{pmatrix}
1 & 0 & 0 \\
0 & 1 & 0 \\
0 & -d_{2,1}^{-1}\lambda_{2}^{2} \big( \frac{r_{1}(\frac{1}{\omega k})-r_{1}(k)r_{1}(\omega^{2}k)}{1+r_{1}(k)r_{2}(k)} \big)_{a} z_{2}^{-i (2\nu_{5}-\nu_{4})} z_{2,(0)}^{-i (2\nu_{2}-\nu_{4})} e^{\frac{iz_{2}^{2}}{2}} & 1
\end{pmatrix}, 
	\\
& \tilde{v}_9 = \begin{pmatrix}
1 & 0 & 0 \\
0 & 1 & d_{2,1}\lambda_{2}^{-2} \big( \frac{r_{2}(\frac{1}{\omega k})-r_{2}(k)r_{2}(\omega^{2}k)}{f(\omega^{2}k)} \big)_{a} z_{2}^{i (2\nu_{5}-\nu_{4})} z_{2,(0)}^{i (2\nu_{2}-\nu_{4})} e^{-\frac{iz_{2}^{2}}{2}} \\
0 & 0 & 1
\end{pmatrix}, 
	\\
& \tilde{v}_{11} = \begin{pmatrix}
1 & 0 & 0 \\
0 & 1 & 0 \\
0 & d_{2,1}^{-1}\lambda_{2}^{2} \big( \frac{r_{1}(\frac{1}{\omega k})-r_{1}(k)r_{1}(\omega^{2}k)}{f(\omega^{2}k)} \big)_{a} z_{2}^{-i (2\nu_{5}-\nu_{4})} z_{2,(0)}^{-i (2\nu_{2}-\nu_{4})} e^{\frac{iz_{2}^{2}}{2}} & 1 
\end{pmatrix}, 
	\\
& \tilde{v}_8^{-1} = \begin{pmatrix}
1 & 0 & 0 \\
0 & 1 & -d_{2,1}\lambda_{2}^{-2} \big( \frac{r_{2}(\frac{1}{\omega k})-r_{2}(k)r_{2}(\omega^{2}k)}{1+r_{1}(k)r_{2}(k)} \big)_{a} z_{2}^{i (2\nu_{5}-\nu_{4})} z_{2,(0)}^{i (2\nu_{2}-\nu_{4})} e^{-\frac{iz_{2}^{2}}{2}} \\
0 & 0 & 1
\end{pmatrix}.
\end{align*}
Let $\mathcal{X}_{2}^{\epsilon} := \cup_{j=1}^{4}\mathcal{X}_{2,j}^{\epsilon}$, where
\begin{align*}
& \mathcal{X}_{2,1}^{\epsilon} = \Gamma_7^{(4)}\cap D_{\epsilon}(\omega^{2} k_{2}), & & \mathcal{X}_{2,2}^{\epsilon} = \Gamma_9^{(4)}\cap D_{\epsilon}(\omega^{2} k_{2}), \\
& \mathcal{X}_{2,3}^{\epsilon} = \Gamma_{11}^{(4)}\cap D_{\epsilon}(\omega^{2} k_{2}), & & \mathcal{X}_{2,4}^{\epsilon} = \Gamma_8^{(4)}\cap D_{\epsilon}(\omega^{2} k_{2}),
\end{align*}
are oriented outwards from $\omega^{2}k_{2}$. Let $k_{\star}=\omega^{2} k_{2}$. The above expressions for $\tilde{v}_j$ suggest that we approximate $\tilde{n}$ by $(1,1,1)Y_2 \tilde{m}^{\omega^{2}k_{2}}$, where $\tilde{m}^{\omega^{2}k_{2}}(x,t,k)$ is analytic for $k \in D_{\epsilon}(\omega^{2} k_{2}) \setminus \mathcal{X}_{2}$,
obeys the jump relation $\tilde{m}^{\omega^{2}k_{2}}_+ = \tilde{m}^{\omega^{2}k_{2}}_- \tilde{v}_{\mathcal{X}_{2}^{\epsilon}}^{\omega^{2}k_{2}}$ on $\mathcal{X}_{2}^{\epsilon}$ where
\begin{align*}
& \tilde{v}_{\mathcal{X}_{2,1}^{\epsilon}}^{\omega^{2}k_{2}} = \begin{pmatrix}
1 & 0 & 0 \\
0 & 1 & 0 \\
0 & -\frac{\mathcal{P}(\omega^{2}k_{\star})}{\mathcal{P}( k_{\star})}\lambda_{2}^{2} \frac{r_{1}(\frac{1}{\omega k_{\star}})-r_{1}(k_{\star})r_{1}(\omega^{2}k_{\star})}{1+r_{1}(k_{\star})r_{2}(k_{\star})} z_{2}^{-i (2\nu_{5}-\nu_{4})} z_{2,(0)}^{-i (2\nu_{2}-\nu_{4})} e^{\frac{iz_{2}^{2}}{2}} & 1
\end{pmatrix}, \\
& \tilde{v}_{\mathcal{X}_{2,2}^{\epsilon}}^{\omega^{2}k_{2}} = \begin{pmatrix}
1 & 0 & 0 \\
0 & 1 & \frac{\mathcal{P}(k_{\star})}{\mathcal{P}(\omega^{2} k_{\star})} \lambda_{2}^{-2} \frac{r_{2}(\frac{1}{\omega k_{\star}})-r_{2}(k_{\star})r_{2}(\omega^{2}k_{\star})}{f(\omega^{2}k_{\star})}  z_{2}^{i (2\nu_{5}-\nu_{4})} z_{2,(0)}^{i (2\nu_{2}-\nu_{4})} e^{-\frac{iz_{2}^{2}}{2}} \\
0 & 0 & 1
\end{pmatrix}, \\
& \tilde{v}_{\mathcal{X}_{2,3}^{\epsilon}}^{\omega^{2}k_{2}} = \begin{pmatrix}
1 & 0 & 0 \\
0 & 1 & 0 \\
0 & \frac{\mathcal{P}(\omega^{2}k_{\star})}{\mathcal{P}( k_{\star})} \lambda_{2}^{2} \frac{r_{1}(\frac{1}{\omega k_{\star}})-r_{1}(k_{\star})r_{1}(\omega^{2}k_{\star})}{f(\omega^{2}k_{\star})} z_{2}^{-i (2\nu_{5}-\nu_{4})} z_{2,(0)}^{-i (2\nu_{2}-\nu_{4})} e^{\frac{iz_{2}^{2}}{2}} & 1 
\end{pmatrix}, \\
& \tilde{v}_{\mathcal{X}_{2,4}^{\epsilon}}^{\omega^{2}k_{2}} = \begin{pmatrix}
1 & 0 & 0 \\
0 & 1 & -\frac{\mathcal{P}(k_{\star})}{\mathcal{P}(\omega^{2} k_{\star})}\lambda_{2}^{-2} \frac{r_{2}(\frac{1}{\omega k_{\star}})-r_{2}(k_{\star})r_{2}(\omega^{2}k_{\star})}{1+r_{1}(k_{\star})r_{2}(k_{\star})}  z_{2}^{i (2\nu_{5}-\nu_{4})} z_{2,(0)}^{i (2\nu_{2}-\nu_{4})} e^{-\frac{iz_{2}^{2}}{2}} \\
0 & 0 & 1
\end{pmatrix},
\end{align*}
and satisfies $\tilde{m}^{\omega^{2}k_{2}} \to I$ as $k \to \infty$. 
We choose our free parameter $\lambda_{2}$ as follows:
\begin{align}\label{lambda2def}
\lambda_{2} = \big( \tfrac{\mathcal{P}(k_{\star})}{\mathcal{P}(\omega^{2} k_{\star})} \big)^{\frac{1}{2}} |\tilde{r}(\tfrac{1}{\omega k_{\star}})|^{\frac{1}{4}} = \big( \tfrac{\mathcal{P}(\omega^{2}k_{2})}{\mathcal{P}(\omega k_{2})} \big)^{\frac{1}{2}} |\tilde{r}(\tfrac{1}{k_{2}})|^{\frac{1}{4}},
\end{align}
and define the local parametrix $m^{\omega^{2} k_{2}}$ by
\begin{align}\label{II m omega2k2 def}
m^{\omega^{2} k_{2}}(x,t,k) = Y_{2}(\zeta,t)m^{X,(2)}(q_{2},q_{4},q_{5},q_{6},z_{2}(\zeta,t,k))Y_{2}(\zeta,t)^{-1}, \qquad k \in D_{\epsilon}(\omega^{2} k_{2}),
\end{align}
where $m^{X,(2)}(\cdot) = m^{X,(2)}(q_{2},q_{4},q_{5},q_{6},\cdot)$ is the solution of the model RH problem of Lemma \ref{II Xlemma 3 green} with 
\begin{align*}
& q_{2} = \tilde{r}(k_{\star})^{\frac{1}{2}}r_{1}(k_{\star}), & & \bar{q}_{2} = \tilde{r}(k_{\star})^{-\frac{1}{2}}r_{2}(k_{\star}) = \tilde{r}(k_{\star})^{\frac{1}{2}}\overline{r_{1}(k_{\star})}, \\
& q_{4} = |\tilde{r}(\tfrac{1}{\omega^{2}k_{\star}})|^{\frac{1}{2}}r_{1}(\tfrac{1}{\omega^{2}k_{\star}}), & & \bar{q}_{4} = -|\tilde{r}(\tfrac{1}{\omega^{2}k_{\star}})|^{-\frac{1}{2}}r_{2}(\tfrac{1}{\omega^{2}k_{\star}}) = |\tilde{r}(\tfrac{1}{\omega^{2}k_{\star}})|^{\frac{1}{2}}\overline{r_{1}(\tfrac{1}{\omega^{2}k_{\star}})}, \\
& q_{5} = |\tilde{r}(\omega^{2}k_{\star})|^{\frac{1}{2}}r_{1}(\omega^{2}k_{\star}), & & \bar{q}_{5} = -|\tilde{r}(\omega^{2}k_{\star})|^{-\frac{1}{2}}r_{2}(\omega^{2}k_{\star}) = |\tilde{r}(\omega^{2}k_{\star})|^{\frac{1}{2}}\overline{r_{1}(\omega^{2}k_{\star})}, \\
& q_{6} = |\tilde{r}(\tfrac{1}{\omega k_{\star}})|^{\frac{1}{2}}r_{1}(\tfrac{1}{\omega k_{\star}}), & & \bar{q}_{6} = -|\tilde{r}(\tfrac{1}{\omega k_{\star}})|^{-\frac{1}{2}}r_{2}(\tfrac{1}{\omega k_{\star}}) = |\tilde{r}(\tfrac{1}{\omega k_{\star}})|^{\frac{1}{2}}\overline{r_{1}(\tfrac{1}{\omega k_{\star}})}.
\end{align*}
Using the relation (\ref{r1r2 relation on the unit circle}) %evaluated at $\omega k$
%\begin{align*}
%r_{1}(\tfrac{1}{\omega^{2}k}) + r_{2}(\omega^{2}k) + r_{1}(k)r_{2}(\tfrac{1}{\omega k})=0,
%\end{align*}
together with the fact that
\begin{align*}
\big| \tilde{r}(\tfrac{1}{\omega^{2}k}) \big|^{-\frac{1}{2}} = \big| \tilde{r}(\omega^{2}k) \big|^{\frac{1}{2}} = |\tilde{r}(k)|^{-\frac{1}{2}}\big| \tilde{r}(\tfrac{1}{\omega k}) \big|^{\frac{1}{2}}, \qquad k \in \mathbb{C}\setminus\{-1,1\},
\end{align*}
we infer that $q_{4}-\bar{q}_{5} - q_{2}\bar{q}_6=0$, so that \eqref{II condition on q2q4q5q6} is fulfilled as it must. The proof of the following lemma is similar to the proof of Lemma \ref{II lemma: bound on Y}.

\begin{lemma}\label{II lemma: bound on Y green}
The function $Y_{2}(\zeta,t)$ is uniformly bounded, i.e., $\sup_{\zeta \in \mathcal{I}} \sup_{t \geq 2} | Y_{2}(\zeta,t)^{\pm 1}| < C$.
Moreover, the functions $d_{2,0}(\zeta, t)$ and $d_{2,1}(\zeta, k)$ satisfy
\begin{align}\label{II d0estimate green}
& |d_{2,0}(\zeta, t)| = e^{\pi (2\nu_{2}-\nu_{4})}, & & \zeta \in \mathcal{I}, \ t \geq 2,
	\\ \label{II d1estimate green}
& |d_{2,1}(\zeta, k) - \tfrac{\mathcal{P}(\omega^{2}k_{2})}{\mathcal{P}(\omega k_{2})}| \leq C |k - \omega^{2} k_2| (1+ |\ln|k-\omega^{2} k_2||), & & \zeta \in \mathcal{I}, \ k \in \mathcal{X}_{2}^{\epsilon}.
\end{align}
\end{lemma}

The next lemma follows from the same kind of arguments used to prove Lemma \ref{II k0lemma}.

\begin{lemma}\label{II k0lemma green}
The function $m^{\omega^{2} k_2}(x,t,k)$ defined in \eqref{II m omega2k2 def} is an analytic function of $k \in D_\epsilon(\omega^{2} k_2) \setminus \mathcal{X}_{2}^\epsilon$ which is uniformly bounded for $t \geq 2$, $\zeta \in \mathcal{I}$, and $k \in D_\epsilon(\omega^{2} k_2) \setminus \mathcal{X}_{2}^\epsilon$.
Across $\mathcal{X}_{2}^\epsilon$, $m^{\omega^{2} k_2}$ obeys the jump condition $m_+^{\omega^{2} k_2} =  m_-^{\omega^{2} k_2} v^{\omega^{2} k_2}$, where $v^{\omega^{2} k_2}$ satisfies
\begin{align}\label{II v3vk0estimate green}
\begin{cases}
 \| v^{(4)} - v^{\omega^{2} k_2} \|_{L^1(\mathcal{X}_{2}^\epsilon)} \leq C t^{-1} \ln t,
	\\
\| v^{(4)} - v^{\omega^{2} k_2} \|_{L^\infty(\mathcal{X}_{2}^\epsilon)} \leq C t^{-1/2} \ln t,
\end{cases} \qquad \zeta \in \mathcal{I}, \ t \geq 2.
\end{align}
Furthermore, as $t \to \infty$,
\begin{align}\label{II mmodmuestimate2 green}
& \| m^{\omega^{2} k_2}(x,t,\cdot)^{-1} - I \|_{L^\infty(\partial D_\epsilon(\omega^{2} k_2))} = O(t^{-1/2}),
	\\ \label{II mmodmuestimate1 green}
& m^{\omega^{2} k_2} - I = \frac{Y_{2}(\zeta, t)m_1^{X,(2)}(q_{2},q_{4},q_{5},q_{6}) Y_{2}(\zeta, t)^{-1}}{z_{2,\star}\sqrt{t} (k-\omega^{2} k_2) \hat{z}_{2}(\zeta,k)} + O(t^{-1})
\end{align}
uniformly for $\zeta \in \mathcal{I}$, where $\partial D_{\epsilon}(\omega^{2} k_2)$ is oriented clockwise, and $m_1^{X,(2)}$ is given by \eqref{II mXasymptotics 3 green}.
\end{lemma}

\section{The $n^{(4)}\to \hat{n}$ transformation}\label{n3tonhatsec}

We use the symmetries
\begin{align*}
& m^{\omega k_4}(x,t,k) = \mathcal{A} m^{\omega k_4}(x,t,\omega k)\mathcal{A}^{-1} = \mathcal{B} m^{\omega k_4}(x,t,k^{-1}) \mathcal{B}, \\
& m^{\omega^{2} k_2}(x,t,k) = \mathcal{A} m^{\omega^{2} k_2}(x,t,\omega k)\mathcal{A}^{-1} = \mathcal{B} m^{\omega^{2} k_2}(x,t,k^{-1}) \mathcal{B},
\end{align*}
to extend the domains of definition of $m^{\omega k_4}$ and $m^{\omega^{2} k_2}$ from $D_\epsilon(\omega k_4)$ and $D_\epsilon(\omega^{2} k_2)$ to $\mathcal{D}_{\omega k_4}$ and $\mathcal{D}_{\omega^{2} k_2}$, respectively, where 
\begin{align*}
& \mathcal{D}_{\omega k_4} := D_\epsilon(\omega k_4) \cup \omega D_\epsilon(\omega k_4) \cup \omega^2 D_\epsilon(\omega k_4) \cup D_\epsilon(\tfrac{1}{\omega k_4}) \cup \omega D_\epsilon(\tfrac{1}{\omega k_4}) \cup \omega^2 D_\epsilon(\tfrac{1}{\omega k_4}), \\
& \mathcal{D}_{\omega^{2} k_2} := D_\epsilon(\omega^{2} k_2) \cup \omega D_\epsilon(\omega^{2} k_2) \cup \omega^2 D_\epsilon(\omega^{2} k_2) \cup D_\epsilon(\tfrac{1}{\omega^{2} k_2}) \cup \omega D_\epsilon(\tfrac{1}{\omega^{2} k_2}) \cup \omega^2 D_\epsilon(\tfrac{1}{\omega^{2} k_2}).
\end{align*}
We will show that the solution $\hat{n}(x,t,k)$ defined by 
\begin{align}\label{Sector II final transfo}
\hat{n} =
\begin{cases}
n^{(4)} (m^{\omega k_4})^{-1}, & k \in \mathcal{D}_{\omega k_4}, \\
n^{(4)} (m^{\omega^{2} k_2})^{-1}, & k \in \mathcal{D}_{\omega^{2} k_2}, \\
n^{(4)}, & \text{elsewhere},
\end{cases}
\end{align}
is small for large $t$. 
Let $\hat{\Gamma} = \Gamma^{(4)} \cup \partial \mathcal{D}$ with $\mathcal{D} := \mathcal{D}_{\omega k_4}\cup \mathcal{D}_{\omega^{2} k_2}$ be the contour displayed in Figure \ref{II Gammahat.pdf}, where each circle that is part of $\partial \mathcal{D}$ is oriented negatively, and define the jump matrix $\hat{v}$ by
\begin{align}\label{def of vhat II}
\hat{v}= \begin{cases}
v^{(4)}, & k \in \hat{\Gamma} \setminus \bar{\mathcal{D}}, 	\\
m^{\omega k_4}, & k \in \partial \mathcal{D}_{\omega k_4}, \\
m^{\omega^{2} k_2}, & k \in \partial \mathcal{D}_{\omega^{2} k_2}, \\
m_-^{\omega k_4} v^{(4)}(m_+^{\omega k_4})^{-1}, & k \in \hat{\Gamma} \cap \mathcal{D}_{\omega k_4}, \\
m_-^{\omega^{2} k_2} v^{(4)}(m_+^{\omega^{2} k_2})^{-1}, & k \in \hat{\Gamma} \cap \mathcal{D}_{\omega^{2} k_2}.
\end{cases}
\end{align}
The function $\hat{n}$ is analytic on $\C \setminus \hat{\Gamma}$ and satisfies $\hat{n}_+ = \hat{n}_- \hat{v}$ for $k \in \hat{\Gamma} \setminus \hat{\Gamma}_\star$, where $\hat{\Gamma}_\star$ denotes the points of self-intersection of $\hat{\Gamma}$. As $k \to \infty$, $\hat{n}(x,t,k) = (1,1,1) + O(k^{-1})$.

\begin{figure}
\begin{center}
\begin{tikzpicture}[master, scale=0.9]
\node at (0,0) {};
%\draw[black,line width=0.55 mm,->-=0.25,->-=0.57,->-=0.71,->-=0.91] (0,0)--(30:6.5);
\draw[black,line width=0.55 mm] (0,0)--(30:6.5);
\draw[black,line width=0.55 mm,->-=0.35,->-=0.65,->-=0.82,->-=0.98] (0,0)--(90:6.5);
%\draw[black,line width=0.55 mm,->-=0.25,->-=0.57,->-=0.71,->-=0.91] (0,0)--(150:6.5);
\draw[black,line width=0.55 mm] (0,0)--(150:6.5);
%\draw[dashed,black,line width=0.15 mm] (0,0)--(60:6.5);
%\draw[dashed,black,line width=0.15 mm] (0,0)--(120:6.5);

%Boundary of sector S
\draw[dashed,black,line width=0.15 mm] (0,0)--(60:7.8);
\draw[dashed,black,line width=0.15 mm] (0,0)--(120:7.8);

\draw[black,line width=0.55 mm] ([shift=(30:3*1.5cm)]0,0) arc (30:150:3*1.5cm);
\draw[black,arrows={-Triangle[length=0.27cm,width=0.18cm]}]
($(70:3*1.5)$) --  ++(70-90:0.001);
\draw[black,arrows={-Triangle[length=0.27cm,width=0.18cm]}]
($(70:3.65)$) --  ++(70-90:0.001);
\draw[black,arrows={-Triangle[length=0.27cm,width=0.18cm]}]
($(70:5.8)$) --  ++(70-90:0.001);
%\draw[black,arrows={-Triangle[length=0.27cm,width=0.18cm]}]
%($(123:3*1.5)$) --  ++(120+90:0.001);
%\draw[black,arrows={-Triangle[length=0.24cm,width=0.16cm]}]
%($(122:3.3*1.5)$) --  ++(120+90:0.001);
%\draw[black,arrows={-Triangle[length=0.24cm,width=0.16cm]}]
%($(122:2.75*1.5)$) --  ++(120+90:0.001);

\draw[black,line width=0.55 mm] ([shift=(90:3.65)]0,0) arc (90:150:3.65);
\draw[black,line width=0.55 mm] ([shift=(90:5.8)]0,0) arc (90:150:5.8);

%\node at (60:3.15*1.5) {\small $9$};
%\node at (60:6.1) {\small $9_{R}$};
%\node at (60:3.9) {\small $9_{L}$};

%\node at (96:1.15*1.5) {\small $1$};
%
%\node at (86:2.6*1.5) {\small $1_{r}$};
%
%\node at (87:7.2) {\small $4'$};
%
%\node at (86.5:3.5*1.5) {\small $4_{r}'$};

\draw[black,line width=0.55 mm] (125:3*1.5) circle (0.3cm);
\draw[black,line width=0.55 mm] (-125+240:3*1.5) circle (0.3cm);
\draw[black,line width=0.55 mm] (100:3*1.5) circle (0.3cm);
\draw[black,line width=0.55 mm] (140:3*1.5) circle (0.3cm);

%Contours 6_s and 5_s
\draw[black,line width=0.55 mm,-<-=0.04,->-=0.80] (100:3.65)--(100:5.8);
%\draw[black,line width=0.55 mm,->-=0.25,->-=0.80] (100:3.65)--(100:5.8);

\draw[black,line width=0.55 mm,->-=0.9] (108:3.65)--(108:4.1);
\draw[black,line width=0.55 mm,->-=0.75] (108:4.9)--(108:5.8);

%Contours 8_s and 7_s
\draw[black,line width=0.55 mm,-<-=0.04,->-=0.80] (115:3.65)--(115:5.8);
%\draw[black,line width=0.55 mm,->-=0.25,->-=0.80] (115:3.65)--(115:5.8);

\draw[black,line width=0.55 mm,->-=0.9] (120:3.65)--(120:4.1);
\draw[black,line width=0.55 mm,->-=0.75] (120:4.9)--(120:5.8);
\draw[black,line width=0.55 mm] (125:3.65)--(125:5.8);
\draw[black,line width=0.55 mm] (132:3.65)--(132:4.1);
\draw[black,line width=0.55 mm] (132:4.95)--(132:5.8);
\draw[black,line width=0.55 mm] (140:3.65)--(140:5.8);

%(111.471:3) to [out=111.471+45, in=120-90] (120:3.5)

\draw[black,line width=0.55 mm] (100:3*1.5) to [out=100+45, in=108-90] (108:3.3*1.5) to [out=108+90, in=115-45] (115:3*1.5) to [out=115+45, in=120-90] (120:3.3*1.5) to [out=120+90, in=140-45]  (125:3*1.5) to [out=125+45, in=132-90] (132:3.3*1.5) to [out=132+90, in=140-45] (140:3*1.5);

\draw[black,line width=0.55 mm] (100:3*1.5) to [out=100+45+90, in=108-90] (108:2.75*1.5) to [out=108+90, in=115-45-90] (115:3*1.5) to [out=115+45+90, in=120-90] (120:2.75*1.5) to [out=120+90, in=140-45-90]  (125:3*1.5) to [out=125+45+90, in=132-90] (132:2.75*1.5) to [out=132+90, in=140-45-90] (140:3*1.5);

\draw[black,line width=0.55 mm,-<-=0.15,->-=0.80] (90:3.65)--(100:3*1.5)--(90:5.8);

%\draw[black,line width=0.55 mm,->-=0.30,->-=0.80] (150:3.65)--(140:3*1.5)--(150:5.8);
\draw[black,line width=0.55 mm] (150:3.65)--(140:3*1.5)--(150:5.8);

\draw[black,line width=0.55 mm] ([shift=(30:3.65cm)]0,0) arc (30:90:3.65cm);
\draw[black,line width=0.55 mm] ([shift=(30:5.8cm)]0,0) arc (30:90:5.8cm);

\draw[black,arrows={-Triangle[length=0.27cm,width=0.18cm]}]
($(91:3*1.5)$) --  ++(91-90:0.001);
\draw[black,arrows={-Triangle[length=0.27cm,width=0.18cm]}]
($(110:3*1.5)$) --  ++(110+90:0.001);
\draw[black,arrows={-Triangle[length=0.27cm,width=0.18cm]}]
($(110:3.3*1.5)$) --  ++(110+90:0.001);
\draw[black,arrows={-Triangle[length=0.27cm,width=0.18cm]}]
($(110:2.75*1.5)$) --  ++(110+90:0.001);

% The angles have been slightly modified to make the figure more visible
\draw[blue,fill] (125:3*1.5) circle (0.1cm);
\draw[green,fill] (-125+240:3*1.5) circle (0.1cm);
\draw[red,fill] (100:3*1.5) circle (0.1cm);
\draw[red,fill] (140:3*1.5) circle (0.1cm);

% Contour through omega
\draw[black,line width=0.4 mm] (120:2.75*1.5)--(120:3.3*1.5);
\draw[black,arrows={-Triangle[length=0.16cm,width=0.12cm]}]
($(120:2.92*1.5)$) --  ++(120:0.001);
\draw[black,arrows={-Triangle[length=0.16cm,width=0.12cm]}]
($(120:3.2*1.5)$) --  ++(120:0.001);

%Contour through -\omega^2
\draw[black,line width=0.55 mm,->-=0.3,->-=0.8] (60:2.45*1.5)--(60:3.85*1.5);

%Arrow for contour 7
\draw[black,arrows={-Triangle[length=0.16cm,width=0.13cm]}]
($(119.3:3.3*1.5)$) --  ++(120+50:0.001);

%Arrow for contour 8
\draw[black,arrows={-Triangle[length=0.16cm,width=0.13cm]}]
($(119.3:3*1.5)$) --  ++(120+90:0.001);

%Arrow for contour 9
\draw[black,arrows={-Triangle[length=0.16cm,width=0.13cm]}]
($(119.3:2.728*1.5)$) --  ++(120+130:0.001);

% ----------------------start the solitons---------------------
\draw[line width=0.45 mm] (120:6.9) circle (0.4);
\draw[black,arrows={-Triangle[length=0.2cm,width=0.18cm]}]
($(120:7.3)+(-0.866*0.15,-0.5*0.15)$) --  ++(120+90:0.001);

% This is the true inverse of the circle. DO NOT ERASE
%\draw[red,line width=0.15 mm] plot[samples=400,variable=\t,domain=-pi:pi,smooth] ({4.5^2*(6.9+0.4*cos(-\t r))/((6.9+0.4*cos(\t r))^2+(0.4*sin(\t r))^2)},{4.5^2*(0.4*sin(-\t r))/((6.9+0.4*cos(\t r))^2+(0.4*sin(\t r))^2 ) });
%\draw[line width=0.15 mm] (0:4.5^2/6.9+0.01) circle (0.18);

\draw[line width=0.45 mm] (120:4.5^2/6.9+0.01) circle (0.18);
\draw[black,arrows={-Triangle[length=0.2cm,width=0.18cm]}]
($(120:4.5^2/6.9+0.018-0.18)+(-0.866*0.11,-0.5*0.11)$) --  ++(120+98:0.001);

\draw[line width=0.45 mm] (120+21:6.5) circle (0.4);
\draw[black,arrows={-Triangle[length=0.2cm,width=0.18cm]}]
($(120+21:6.5)+(120:0.4)+(-0.866*0.15,-0.5*0.15)$) --  ++(120+90:0.001);
\draw[line width=0.45 mm] (120-21:6.5) circle (0.4);
\draw[black,arrows={-Triangle[length=0.2cm,width=0.18cm]}]
($(120-21:6.5)+(120:0.4)+(-0.866*0.15,-0.5*0.15)$) --  ++(120+90:0.001);

% This is the true inverse of the circle. DO NOT ERASE
%\draw[red,line width=0.15 mm] plot[samples=400,variable=\t,domain=-pi:pi,smooth] ({4.5^2*(6.5*cos(-21)+0.4*cos(-\t r))/((6.5*cos(21)+0.4*cos(\t r))^2+(6.5*sin(21)+0.4*sin(\t r))^2)},{4.5^2*(6.5*sin(-21)+0.4*sin(-\t r))/((6.5*cos(21)+0.4*cos(\t r))^2+(6.5*sin(21)+0.4*sin(\t r))^2) });
%\draw[line width=0.15 mm] (-21:3.125) circle (0.20);

\draw[line width=0.45 mm] (-21+120:3.125) circle (0.2);
\draw[black,arrows={-Triangle[length=0.2cm,width=0.18cm]}]
($(-21+120:3.125)-(120:0.20)+(-0.866*0.10,-0.5*0.10)$) --  ++(98+120:0.001);
\draw[line width=0.45 mm] (21+120:3.125) circle (0.2);
\draw[black,arrows={-Triangle[length=0.2cm,width=0.18cm]}]
($(21+120:3.125)-(120:0.20)+(-0.866*0.10,-0.5*0.10)$) --  ++(98+120:0.001);

% ----------------------end the solitons---------------------
\end{tikzpicture}
\end{center}
\begin{figuretext}
\label{II Gammahat.pdf}The contour $\hat{\Gamma} = \Gamma^{(4)} \cup \partial \mathcal{D}$ for $\arg k \in [\frac{\pi}{6},\frac{5\pi}{6}]$ (solid) and the boundary of $\mathsf{S}$ (dashed). %From right to left, the dots are $\omega k_{4}$ (red), $\omega^{2}k_{2}$ (green), $k_{1}$ (blue), and $\omega k_{3}$ (red). 
\end{figuretext}
\end{figure}

Let $\hat{\mathcal{X}}^\epsilon = \bigcup_{j=1}^{2}\big(\mathcal{X}_{j}^\epsilon \cup \omega \mathcal{X}_{j}^\epsilon \cup \omega^2 \mathcal{X}_{j}^\epsilon \cup (\mathcal{X}_{j}^\epsilon)^{-1} \cup (\omega \mathcal{X}_{j}^\epsilon)^{-1} \cup (\omega^2 \mathcal{X}_{j}^\epsilon)^{-1}\big)$.

\begin{lemma}\label{II whatlemma}
Let $\hat{w} = \hat{v}-I$. The following estimates hold uniformly for $t \geq 2$ and $\zeta \in \mathcal{I}$:
\begin{subequations}\label{II hatwestimate}
\begin{align}\label{II hatwestimate1}
& \| \hat{w}\|_{(L^1\cap L^\infty)(\hat\Gamma \setminus (\partial \mathcal{D} \cup \hat{\mathcal{X}}^\epsilon))} \leq C t^{-1},
	\\\label{II hatwestimate3}
& \| \hat{w}\|_{(L^1\cap L^{\infty})(\partial \mathcal{D})} \leq C t^{-1/2},	\\\label{II hatwestimate4}
& \| \hat{w}\|_{L^1(\hat{\mathcal{X}}^\epsilon)} \leq C t^{-1}\ln t,
	\\\label{II hatwestimate5}
& \| \hat{w}\|_{L^\infty(\hat{\mathcal{X}}^\epsilon)} \leq C t^{-1/2}\ln t.
\end{align}
\end{subequations}
\end{lemma}
\begin{proof}
On $\hat{\Gamma} \setminus \bar{\mathcal{D}}$, the estimate \eqref{II hatwestimate1} follows from \eqref{def of vhat II} and Lemmas \ref{II v3lemma} and \ref{circleslemma}; on $(\hat{\Gamma} \cap \mathcal{D})\setminus \hat{\mathcal{X}}^\epsilon$, it follows from \eqref{def of vhat II} and Lemmas \ref{symmetryjumpslemma}, \ref{DeltalemmaIV}, \ref{II k0lemma}, and \ref{II k0lemma green}. The estimate \eqref{II hatwestimate3} follows from \eqref{def of vhat II}, \eqref{II mmodmuestimate2}, and \eqref{II mmodmuestimate2 green}. Finally, \eqref{II hatwestimate4} and \eqref{II hatwestimate5} follow from \eqref{def of vhat II}, \eqref{II v3vk0estimate}, and \eqref{II v3vk0estimate green}.
\end{proof}

%For a function $h$ defined on $\hat{\Gamma}$, we define $\hat{\mathcal{C}}h$ by
%\begin{align*}
%(\hat{\mathcal{C}}h)(z) = \frac{1}{2\pi i} \int_{\hat{\Gamma}} \frac{h(z')dz'}{z' - z}, \qquad z \in \C \setminus \hat{\Gamma}.
%\end{align*}

The estimates in Lemma \ref{II whatlemma} show that
\begin{align}\label{II hatwLinfty}
\begin{cases}
\|\hat{w}\|_{L^1(\hat{\Gamma})}\leq C t^{-1/2},
	\\
\|\hat{w}\|_{L^\infty(\hat{\Gamma})}\leq C t^{-1/2}\ln t,
\end{cases}	 \qquad t \geq 2, \ \zeta \in \mathcal{I},
\end{align}
and hence, employing the general identity $\| f \|_{L^p} \leq \| f \|_{L^1}^{1/p}\|f \|_{L^{\infty}}^{(p-1)/p}$,
\begin{align}\label{II Lp norm of what}
& \| \hat{w}\|_{L^p(\hat{\Gamma})} 
\leq C t^{-\frac{1}{2}} (\ln t)^{\frac{p-1}{p}},  \qquad t \geq 2, \ \zeta \in \mathcal{I},
\end{align}
for each $1 \leq p \leq \infty$. The estimates \eqref{II Lp norm of what} imply that $\hat{w} \in (L^2 \cap L^{\infty})(\hat{\Gamma})$ and that there is a $T$ such that $I - \hat{\mathcal{C}}_{\hat{w}(x, t, \cdot)} \in \mathcal{B}(L^2(\hat{\Gamma}))$ is invertible whenever $\zeta \in \mathcal{I}$ and $t \geq T$. So standard RH theory gives
\begin{align}\label{II hatmrepresentation}
\hat{n}(x, t, k) = (1,1,1) + \hat{\mathcal{C}}(\hat{\mu}\hat{w}) = (1,1,1) + \frac{1}{2\pi i}\int_{\hat{\Gamma}} \hat{\mu}(x, t, s) \hat{w}(x, t, s) \frac{ds}{s - k}
\end{align}
for $\zeta \in \mathcal{I}$ and $t \geq T$, where 
\begin{align}\label{II hatmudef}
\hat{\mu} := (1,1,1) + (I - \hat{\mathcal{C}}_{\hat{w}})^{-1}\hat{\mathcal{C}}_{\hat{w}}(1,1,1) \in (1,1,1) + L^2(\hat{\Gamma}).
\end{align}
%$\hat{\mathcal{C}}_{\hat{w}}h := \hat{\mathcal{C}}_{-}(h \hat{w})$ 

\begin{lemma}\label{II lemma: estimate on mu}
Let $1 < p < \infty$. For all sufficiently large $t$, we have
\begin{align*}
& \|\hat{\mu} - (1,1,1)\|_{L^p(\hat{\Gamma})} \leq  C t^{-\frac{1}{2}}(\ln t)^{\frac{p-1}{p}}, \qquad \zeta \in \mathcal{I}.
\end{align*}
\end{lemma}
\begin{proof}
Suppose $t$ is sufficiently large so that $\|\hat{w}\|_{L^\infty(\hat{\Gamma})}<K_{p}^{-1}$, where $K_{p} := \|\hat{\mathcal{C}}_{-}\|_{\mathcal{B}(L^p(\hat{\Gamma}))}$ is finite because $p > 1$. 
The estimate \eqref{II Lp norm of what} implies that $\hat{w}\in L^{p}(\hat{\Gamma})$, and so $\hat{\mathcal{C}}_{\hat{w}}(1,1,1) \in L^{p}(\hat{\Gamma})$. Thus
\begin{align*}
\|\hat{\mu} - (1,1,1)\|_{L^p(\hat{\Gamma})} & \leq
\sum_{j=1}^{\infty} \|\hat{\mathcal{C}}_{\hat{w}}\|_{\mathcal{B}(L^p(\hat{\Gamma}))}^{j-1}\|\hat{\mathcal{C}}_{\hat{w}}(1,1,1)\|_{L^{p}(\hat{\Gamma})}  
\leq \sum_{j=1}^{\infty} K_{p}^{j}\|\hat{w}\|_{L^\infty(\hat{\Gamma})}^{j-1} \|\hat{w}\|_{L^p(\hat{\Gamma})} 
	\\
& = \frac{K_{p} \|\hat{w}\|_{L^p(\hat{\Gamma})}}{1-K_{p}\|\hat{w}\|_{L^\infty(\hat{\Gamma})}}.
\end{align*}
The claim is now a consequence of \eqref{II hatwLinfty} and \eqref{II Lp norm of what}. 
\end{proof}

We now turn to the asymptotics of $\hat{n}$. 
The following limit exists as $k \to \infty$:
\begin{align*}
& \hat{n}^{(1)}(x,t):=\lim_{k\to \infty} k(\hat{n}(x,t,k) - (1,1,1))
= - \frac{1}{2\pi i}\int_{\hat{\Gamma}} \hat{\mu}(x,t,k) \hat{w}(x,t,k) dk.
\end{align*}
\begin{lemma}
As $t \to \infty$, 
\begin{align}\label{II limlhatm}
& \hat{n}^{(1)}(x,t) = -\frac{(1,1,1)}{2\pi i}\int_{\partial \mathcal{D}}\hat{w}(x,t,k) dk + O(t^{-1}\ln t).
\end{align}
\end{lemma}
\begin{proof}
Since
\begin{align*}
\hat{n}^{(1)}(x,t) = & -\frac{(1,1,1)}{2\pi i}\int_{\partial \mathcal{D}} \hat{w}(x,t,k) dk  -\frac{(1,1,1)}{2\pi i}\int_{\hat{\Gamma}\setminus\partial \mathcal{D}} \hat{w}(x,t,k) dk
	\\
& -\frac{1}{2\pi i}\int_{\hat{\Gamma}} (\hat{\mu}(x,t,k)-(1,1,1))\hat{w}(x,t,k) dk,
\end{align*}
the lemma follows from Lemmas \ref{II whatlemma} and \ref{II lemma: estimate on mu} and straightforward estimates.
\end{proof}

We define the functions $\{F_{1}^{(l)},F_{2}^{(l)}\}_{l \in \mathbb{Z}}$ by
\begin{align*}
& F_{1}^{(l)}(\zeta,t) = - \frac{1}{2\pi i} \int_{\partial D_\epsilon(\omega k_4)} k^{l-1}\hat{w}(x,t,k) dk	
= - \frac{1}{2\pi i}\int_{\partial D_\epsilon(\omega k_4)}k^{l-1}(m^{\omega k_4} - I) dk, \\
& F_{2}^{(l)}(\zeta,t) = - \frac{1}{2\pi i} \int_{\partial D_\epsilon(\omega^{2} k_2)} k^{l-1}\hat{w}(x,t,k) dk	
= - \frac{1}{2\pi i}\int_{\partial D_\epsilon(\omega^{2} k_2)}k^{l-1}(m^{\omega^{2} k_2} - I) dk.
\end{align*}
We infer from \eqref{II mmodmuestimate1} and \eqref{II mmodmuestimate1 green} that $F_{1}^{(l)}$ and $F_{2}^{(l)}$ satisfy (recall that $\hat{z}_{1}(\zeta,\omega k_{4})=1=\hat{z}_{2}(\zeta,\omega^{2} k_{2})$ and that $\partial D_\epsilon(\omega k_{4})$ and $\partial D_\epsilon(\omega^{2} k_{2})$ are oriented negatively)
\begin{align}
F_{1}^{(l)}(\zeta, t) &  =  (\omega k_{4})^{l-1}\frac{Y_{1}(\zeta,t) m_1^{X,(1)}(q_1,q_3) Y_{1}(\zeta,t)^{-1}}{z_{1,\star}\sqrt{t}} + O(t^{-1}) \nonumber \\
& =  -i(\omega k_{4})^{l}Z_{1}(\zeta,t) + O(t^{-1}) \qquad \mbox{as } t \to \infty, \label{II asymptotics for F1l} \\
F_{2}^{(l)}(\zeta, t) &  =  (\omega^{2} k_{2})^{l-1}\frac{Y_{2}(\zeta,t) m_1^{X,(2)}(q_2,q_4,q_5,q_6) Y_{2}(\zeta,t)^{-1}}{z_{2,\star}\sqrt{t}} + O(t^{-1}) \nonumber \\
& =  -i(\omega^{2} k_{2})^{l}Z_{2}(\zeta,t) + O(t^{-1}) \qquad \mbox{as } t \to \infty, \label{II asymptotics for F4l}
\end{align}
where
\begin{subequations}\label{Z1Z2def}
\begin{align}
& Z_{1}(\zeta,t) := \frac{Y_{1} m_1^{X,(1)} Y_{1}^{-1}}{-i\omega k_{4}z_{1,\star}\sqrt{t}} = 
\frac{1}{-i\omega k_{4}z_{1,\star}\sqrt{t}} \begin{pmatrix}
0 & 0 & \frac{\beta_{12}^{(1)}d_{1,0}\lambda_{1}^{2}}{e^{t\Phi_{31}(\zeta,\omega k_{4})}} \\
0 & 0 & 0 \\
\frac{\beta_{21}^{(1)}e^{t\Phi_{31}(\zeta,\omega k_{4})}}{d_{1,0}\lambda_{1}^{2}} & 0 & 0
\end{pmatrix}, \\
& Z_{2}(\zeta,t) := \frac{Y_{2} m_1^{X,(2)} Y_{2}^{-1}}{-i\omega^{2} k_{2}z_{2,\star}\sqrt{t}} = 
\frac{1}{-i\omega^{2}k_{2}z_{2,\star}\sqrt{t}} \begin{pmatrix}
0 & 0 & 0 \\
0 & 0 & \frac{\beta_{12}^{(2)}d_{2,0}\lambda_{2}^{2}}{e^{t\Phi_{32}(\zeta,\omega^{2}k_{2})}} \\
0 & \frac{\beta_{21}^{(2)}e^{t\Phi_{32}(\zeta,\omega^{2}k_{2})}}{d_{2,0}\lambda_{2}^{2}} & 0
\end{pmatrix},
\end{align}
\end{subequations}
and
\begin{subequations}\label{betap1p and betap2p Sector II}
\begin{align}
& \beta_{12}^{(1)} = \frac{e^{\frac{3\pi i}{4}}e^{\frac{\pi \hat{\nu}_{1}}{2}}e^{2\pi \nu_{1}}\sqrt{2\pi}q_{3}}{(e^{\pi \hat{\nu}_{1}}-e^{-\pi \hat{\nu}_{1}})\Gamma(-i\hat{\nu}_{1})}, & & \beta_{21}^{(1)} = \frac{e^{-\frac{3\pi i}{4}}e^{\frac{\pi \hat{\nu}_{1}}{2}}\sqrt{2\pi}\bar{q}_{3}}{(e^{\pi \hat{\nu}_{1}}-e^{-\pi \hat{\nu}_{1}})\Gamma(i\hat{\nu}_{1})}, \\
&  \beta_{12}^{(2)} = \frac{e^{\frac{3\pi i}{4}}e^{\frac{\pi\hat{\nu}_{2}}{2}}e^{2\pi (\nu_{4}-\nu_{2})}\sqrt{2\pi}(\bar{q}_6-\bar{q}_2\bar{q}_5)}{(e^{\pi \hat{\nu}_{2}}-e^{-\pi \hat{\nu}_{2}})\Gamma(-i\hat{\nu}_{2})}, & & \beta_{21}^{(2)} = \frac{e^{-\frac{3\pi i}{4}}e^{\frac{\pi \hat{\nu}_{2}}{2}}e^{2\pi\nu_{2}}\sqrt{2\pi}(q_{6}-q_{2}q_{5})}{(e^{\pi \hat{\nu}_{2}}-e^{-\pi \hat{\nu}_{2}})\Gamma(i\hat{\nu}_{2})},
\end{align}
\end{subequations}
and $\nu_{1}$, $\hat{\nu}_{1}$, $\nu_{2}$, $\hat{\nu}_{2}$, $\nu_4$ have been defined in Lemmas \ref{II deltalemma} and \ref{lemma: nuhat lemma II}.

\begin{lemma}\label{II lemma: some integrals by symmetry}
For $l \in \mathbb{Z}$ and $j = 0,1,2$, we have
\begin{align}
& -\frac{1}{2\pi i}\int_{\omega^{j} \partial D_\epsilon(\omega k_4)} k^{l}\hat{w}(x,t,k)dk = \omega^{j(l+1)} \mathcal{A}^{-j}F_{1}^{(l+1)}(\zeta,t)\mathcal{A}^{j}, \label{II int1 F1} \\
& -\frac{1}{2\pi i}\int_{\omega^{j} \partial D_\epsilon((\omega k_4)^{-1})} k^{l}\hat{w}(x,t,k)dk = -\omega^{j(l+1)}\mathcal{A}^{-j}\mathcal{B} F_{1}^{(-l-1)}(\zeta,t) \mathcal{B}\mathcal{A}^{j}, \label{II int1 tilde F1} \\
& -\frac{1}{2\pi i}\int_{\omega^{j} \partial D_\epsilon(\omega^{2} k_2)} k^{l}\hat{w}(x,t,k)dk = \omega^{j(l+1)} \mathcal{A}^{-j}F_{2}^{(l+1)}(\zeta,t)\mathcal{A}^{j}, \label{II int1 F2} \\
& -\frac{1}{2\pi i}\int_{\omega^{j} \partial D_\epsilon((\omega^{2} k_2)^{-1})} k^{l}\hat{w}(x,t,k)dk = -\omega^{j(l+1)}\mathcal{A}^{-j}\mathcal{B} F_{2}^{(-l-1)}(\zeta,t) \mathcal{B}\mathcal{A}^{j}. \label{II int1 tilde F2}
\end{align}
\end{lemma}
\begin{proof}
The symmetry properties of $\hat{v}$ imply that 
$$\hat{w}(x, t, k) = \mathcal{A} \hat{w}(x, t, \omega k) \mathcal{A}^{-1} = \mathcal{B} \hat{w}(x, t, k^{-1}) \mathcal{B}, \qquad k \in \partial \mathcal{D}.$$
Therefore, for $j = 0,1,2$ we have 
\begin{align*}
& -\frac{1}{2\pi i}\int_{\omega^{j} \partial D_\epsilon(\omega k_4)} k^{l}\hat{w}(x,t,k)dk = -\frac{\omega^{j(l+1)}}{2\pi i}\int_{\partial D_\epsilon(\omega k_4)} k^{l}\hat{w}(x,t,\omega^{j}k)dk = \omega^{j(l+1)} \mathcal{A}^{-j}F_{1}^{(l+1)}(\zeta,t)\mathcal{A}^{j}, 
\end{align*}
and
\begin{align*}
& -\frac{1}{2\pi i}\int_{\omega^{j} \partial D_\epsilon((\omega k_4)^{-1})} k^{l}\hat{w}(x,t,k)dk = -\frac{\omega^{j(l+1)}}{2\pi i}\int_{\partial D_\epsilon((\omega k_4)^{-1})} k^{l}\hat{w}(x,t,\omega^{j}k)dk \\
& = -\frac{\omega^{j(l+1)}\mathcal{A}^{-j}}{2\pi i}\int_{\partial D_\epsilon((\omega k_4)^{-1})} k^{l}\hat{w}(x,t,k)dk \mathcal{A}^{j} = -\frac{\omega^{j(l+1)}\mathcal{A}^{-j}}{2\pi i}\int_{\partial D_\epsilon(\omega k_4)} k^{-l}\hat{w}(x,t,k^{-1}) \frac{dk}{-k^{2}} \mathcal{A}^{j} \\
& = \frac{\omega^{j(l+1)}\mathcal{A}^{-j}\mathcal{B}}{2\pi i}\int_{\partial D_\epsilon(\omega k_4)} k^{-l-2}\hat{w}(x,t,k) dk \mathcal{B}\mathcal{A}^{j} = -\omega^{j(l+1)}\mathcal{A}^{-j}\mathcal{B} F_{1}^{(-l-1)}(\zeta,t) \mathcal{B}\mathcal{A}^{j}.
\end{align*}
This proves \eqref{II int1 F1} and \eqref{II int1 tilde F1}. The proofs of \eqref{II int1 F2} and \eqref{II int1 tilde F2} are similar.
\end{proof}
Using Lemma \ref{II lemma: some integrals by symmetry}, we find
\begin{align*}
 - \frac{1}{2\pi i}\int_{\partial \mathcal{D}} \hat{w}(x,t,k) dk 
& = - \sum_{j=0}^{2}  \frac{1}{2\pi i}\int_{\omega^{j} \partial D_\epsilon(\omega k_4)} \hat{w}(x,t,k) dk - \sum_{j=0}^{2}  \frac{1}{2\pi i}\int_{\omega^{j} \partial D_\epsilon((\omega k_4)^{-1})}  \hat{w}(x,t,k) dk \\
& \quad - \sum_{j=0}^{2}  \frac{1}{2\pi i}\int_{\omega^{j} \partial D_\epsilon(\omega^{2} k_2)} \hat{w}(x,t,k) dk - \sum_{j=0}^{2}  \frac{1}{2\pi i}\int_{\omega^{j} \partial D_\epsilon((\omega^{2} k_2)^{-1})} \hat{w}(x,t,k) dk \\
& = \sum_{j=0}^{2} \omega^{j} \mathcal{A}^{-j}F_{1}^{(1)}(\zeta,t)\mathcal{A}^{j} - \sum_{j=0}^{2} \omega^{j}\mathcal{A}^{-j}\mathcal{B} F_{1}^{(-1)}(\zeta,t) \mathcal{B}\mathcal{A}^{j} \\
& \quad + \sum_{j=0}^{2} \omega^{j} \mathcal{A}^{-j}F_{2}^{(1)}(\zeta,t)\mathcal{A}^{j} - \sum_{j=0}^{2} \omega^{j}\mathcal{A}^{-j}\mathcal{B} F_{2}^{(-1)}(\zeta,t) \mathcal{B}\mathcal{A}^{j}.
\end{align*}
Therefore, \eqref{II limlhatm}, \eqref{II asymptotics for F1l}, and \eqref{II asymptotics for F4l} imply that $\hat{n}^{(1)} = (1,1,1)\hat{m}^{(1)}$, where
\begin{align}
& \hat{m}^{(1)}(x,t) =  \; \sum_{j=0}^{2} \omega^{j} \mathcal{A}^{-j}F_{1}^{(1)}(\zeta,t)\mathcal{A}^{j} - \sum_{j=0}^{2} \omega^{j}\mathcal{A}^{-j}\mathcal{B} F_{1}^{(-1)}(\zeta,t) \mathcal{B}\mathcal{A}^{j} \nonumber \\
& + \sum_{j=0}^{2} \omega^{j} \mathcal{A}^{-j}F_{2}^{(1)}(\zeta,t)\mathcal{A}^{j} - \sum_{j=0}^{2} \omega^{j}\mathcal{A}^{-j}\mathcal{B} F_{2}^{(-1)}(\zeta,t) \mathcal{B}\mathcal{A}^{j} + O(t^{-1}\ln t) \nonumber \\
& = -i(\omega k_{4})^{1}\sum_{j=0}^{2} \omega^{j} \mathcal{A}^{-j}Z_{1}(\zeta,t)\mathcal{A}^{j} + i(\omega k_{4})^{-1} \sum_{j=0}^{2} \omega^{j}\mathcal{A}^{-j}\mathcal{B} Z_{1}(\zeta,t) \mathcal{B}\mathcal{A}^{j} \nonumber \\
& -i(\omega^{2} k_{2})^{1}\sum_{j=0}^{2} \omega^{j} \mathcal{A}^{-j}Z_{2}(\zeta,t)\mathcal{A}^{j} + i(\omega^{2} k_{2})^{-1} \sum_{j=0}^{2} \omega^{j}\mathcal{A}^{-j}\mathcal{B} Z_{2}(\zeta,t) \mathcal{B}\mathcal{A}^{j} + O(t^{-1}\ln t) \label{II mhatplp asymptotics}
\end{align}
as $t \to \infty$ uniformly for $\zeta \in \mathcal{I}$.

\section{Asymptotics of $u$}\label{usec}

Recall from (\ref{recoveruvn}) that $u(x,t) = -i\sqrt{3}\frac{\partial}{\partial x}n_{3}^{(1)}(x,t)$, where $n_{3}(x,t,k) = 1+n_{3}^{(1)}(x,t)k^{-1}+O(k^{-2})$ as $k \to \infty$. 
Recall also that $n^{(4)}$ is defined to equal $n^{(3)}$ except in $\mathcal{D}_{\sol}$.
Taking also the transformations \eqref{Sector II first transfo}, \eqref{Sector II second transfo}, \eqref{II def of mp3p}, and \eqref{Sector II final transfo} into account, we obtain
\begin{align*}
n = \hat{n}\Delta^{-1}(G^{(2)})^{-1}(G^{(1)})^{-1}
\end{align*}
for $k \in \mathbb{C}\setminus (\hat{\Gamma}\cup \bar{\mathcal{D}} \cup \bar{\mathcal{D}}_{\sol})$, where $G^{(1)}$, $G^{(2)}$, $\Delta$ are defined in \eqref{II Gp1pdef}, \eqref{II Gp2pdef}, and \eqref{II def of Delta}, respectively.
Thus
\begin{align}
u(x,t) & = -i\sqrt{3}\frac{\partial}{\partial x}\bigg( \hat{n}_{3}^{(1)}(x,t) + \lim_{k\to \infty} k (\Delta_{33}(\zeta,k)^{-1}-1) \bigg) = -i\sqrt{3}\frac{\partial}{\partial x} \hat{n}_{3}^{(1)}(x,t) + O(t^{-1})
\end{align}
as $t \to \infty$. By \eqref{II mhatplp asymptotics}, as $t \to \infty$,
\begin{align*}
 \hat{n}_{3}^{(1)} = &\; Z_{1,13}(i (\omega k_{4})^{-1}-i (\omega k_{4})^{1})+Z_{1,31}(i \omega^{2}(\omega k_{4})^{-1}- i \omega^{1}(\omega k_{4})^{1})
 	\\
& + Z_{2,23}(i (\omega^{2} k_{2})^{-1}-i (\omega^{2} k_{2})^{1})+Z_{2,32}(i \omega^{1}(\omega^{2} k_{2})^{-1}- i \omega^{2}(\omega^{2} k_{2})^{1}) + O(t^{-1}\ln t).
\end{align*} 
Recall from Lemma \ref{lemma: nuhat lemma II} that $\hat{\nu}_{1}=\nu_{3}-\nu_{1}\geq 0$ and $\hat{\nu}_{2}=\nu_{2}+\nu_{5}-\nu_{4}\geq 0$. From (\ref{II d0estimate}), (\ref{II d0estimate green}), (\ref{Z1Z2def}), and (\ref{betap1p and betap2p Sector II}), we see that $\overline{Z_{1,13}} = \lambda_1^4 Z_{1,31}$ and $\overline{Z_{2,23}} = \lambda_2^4 Z_{2,32}$. Moreover, by (\ref{lambda1def}) and (\ref{lambda2def}), we have $\lambda_1^4 = -\tilde{r}(k_4^{-1})^{-1}\big(\frac{\mathcal{P}(\omega k_{4})}{\mathcal{P}(\omega^{2} k_{4})}\big)^{2}$ and $\lambda_2^4 = -\tilde{r}(k_2^{-1})\big(\frac{\mathcal{P}(\omega^{2} k_{2})}{\mathcal{P}(\omega k_{2})}\big)^{2}$. It follows that, as $t \to \infty$,
\begin{align*}
& \hat{n}_{3}^{(1)} = 2i \, \im \Big( Z_{1,13}(i (\omega k_{4})^{-1}-i (\omega k_{4})^{1}) + Z_{2,23}(i (\omega^{2} k_{2})^{-1}-i (\omega^{2} k_{2})^{1}) \Big) + O(t^{-1}\ln t) \\
& = 2 i \, \im \Bigg[ \frac{\beta_{12}^{(1)}d_{1,0}(i (\omega k_{4})^{-1}-i (\omega k_{4})^{1}) \mathcal{P}(\omega k_{4})}{-i \omega k_{4} z_{1,\star} \sqrt{t} e^{t \Phi_{31}(\zeta,\omega k_{4})}\big|\tilde{r}(\frac{1}{k_{4}})\big|^{\frac{1}{2}} \mathcal{P}(\omega^{2} k_{4})} \\
& + \frac{\beta_{12}^{(2)}d_{2,0}(i (\omega^{2} k_{2})^{-1}-i (\omega^{2} k_{2})^{1})\mathcal{P}(\omega^{2}k_{2})}{-i \omega^{2} k_{2} z_{2,\star} \sqrt{t} e^{t \Phi_{32}(\zeta,\omega^{2} k_{2})}\big|\tilde{r}(\frac{1}{k_{2}})\big|^{-\frac{1}{2}}\mathcal{P}(\omega k_{2})} \Bigg] + O(t^{-1}\ln t).
\end{align*}
The function $\mathcal{P}(k)$ obeys the symmetries $\mathcal{P}(k) = \mathcal{P}(k^{-1}) = \overline{\mathcal{P}(\bar{k})}^{-1}$. It follows that $|\mathcal{P}(k)| = 1$ for $k \in \partial \D$, and thus
\begin{align}\label{abscalP}
\big| \tfrac{\mathcal{P}(\omega k_{4})}{\mathcal{P}(\omega^{2} k_{4})} \big| = \big| \tfrac{\mathcal{P}(\omega^{2} k_{2})}{\mathcal{P}(\omega k_{2})} \big| = 1.
\end{align}
Using \eqref{betap1p and betap2p Sector II}, (\ref{abscalP}), and the identities
\begin{align*}
& d_{1,0}(\zeta,t) = e^{-\pi \nu_{1}}e^{i\arg d_{1,0}(\zeta,t)}, & & d_{2,0}(\zeta,t) = e^{\pi (2\nu_{2}-\nu_{4})}e^{i\arg d_{2,0}(\zeta,t)}, \\
& i (\omega k_{4})^{-1}-i (\omega k_{4})^{1} = 2 \sin(\arg(\omega k_{4}(\zeta))), & & i (\omega^{2} k_{2})^{-1}-i (\omega^{2} k_{2})^{1} = 2 \sin(\arg(\omega^{2} k_{2}(\zeta))), \\
& |\Gamma(i\nu)| = \frac{\sqrt{2\pi}}{\sqrt{\nu(e^{\pi \nu}-e^{-\pi \nu})}}, \quad \nu \in \mathbb{R}, & & q_{4}-\bar{q}_5-q_{2}\bar{q}_6=0, 
\end{align*}
we obtain
\begin{align*}
& \hat{n}_{3}^{(1)} = \frac{4i\sqrt{\hat{\nu}_{1}}}{-i\omega k_{4} z_{1,\star}\sqrt{t}|\tilde{r}(\frac{1}{k_{4}})|^{\frac{1}{2}}}\sin(\arg(\omega k_{4})) \\
& \times \sin(\tfrac{3\pi}{4}+\arg q_{3}+\arg\Gamma(i\hat{\nu}_{1})+\arg d_{1,0} + \arg \tfrac{\mathcal{P}(\omega k_{4})}{\mathcal{P}(\omega^{2} k_{4})}-t\im \Phi_{31}(\zeta,\omega k_{4}(\zeta)))  \\
& +\frac{4i\sqrt{\hat{\nu}_{2}}|\tilde{r}(\frac{1}{k_{2}})|^{\frac{1}{2}}}{-i\omega^{2} k_{2} z_{2,\star}\sqrt{t}}\sin(\arg(\omega^{2} k_{2})) \\
& \times \sin(\tfrac{3\pi}{4}-\arg (q_{6}-q_{2}q_{5})+\arg\Gamma(i\hat{\nu}_{2})+\arg d_{2,0}+ \arg \tfrac{\mathcal{P}(\omega^{2} k_{2})}{\mathcal{P}(\omega k_{2})}-t\im \Phi_{32}(\zeta,\omega^{2} k_{2}(\zeta)) ) \\
& + O(t^{-1}\ln t) \quad \mbox{as } t \to \infty.
\end{align*}
As in \cite{CLWasymptotics}, it is possible to show that the above asymptotics can be differentiated with respect to $x$ without increasing the error term. Therefore, as $t \to \infty$,
\begin{align*}
& u(x,t) = -i\sqrt{3}\frac{\partial}{\partial x}\hat{n}_{3}^{(1)}(x,t) + O(t^{-1}) = \frac{-4\sqrt{3}\sqrt{\hat{\nu}_{1}}\frac{d}{d\zeta}\im \Phi_{31}(\zeta,\omega k_{4}(\zeta))}{-i\omega k_{4} z_{1,\star}\sqrt{t}|\tilde{r}(\frac{1}{k_{4}})|^{\frac{1}{2}}}\sin(\arg(\omega k_{4})) \\
& \times \cos(\tfrac{3\pi}{4}+\arg q_{3}+\arg\Gamma(i\hat{\nu}_{1})+\arg d_{1,0} + \arg \tfrac{\mathcal{P}(\omega k_{4})}{\mathcal{P}(\omega^{2} k_{4})} -t\im \Phi_{31}(\zeta,\omega k_{4}(\zeta))) \\
&  +\frac{-4\sqrt{3}\sqrt{\hat{\nu}_{2}}|\tilde{r}(\frac{1}{k_{2}})|^{\frac{1}{2}}\frac{d}{d\zeta}\im \Phi_{32}(\zeta,\omega^{2} k_{2}(\zeta))}{-i\omega^{2} k_{2} z_{2,\star}\sqrt{t}}\sin(\arg(\omega^{2} k_{2}))\\
&\times \cos(\tfrac{3\pi}{4}-\arg (q_{6}-q_{2}q_{5})+\arg\Gamma(i\hat{\nu}_{2})+\arg d_{2,0} + \arg \tfrac{\mathcal{P}(\omega^{2} k_{2})}{\mathcal{P}(\omega k_{2})} -t\im \Phi_{32}(\zeta,\omega^{2} k_{2}(\zeta))) \\
& + O(t^{-1}\ln t).
\end{align*}
In view of the relations
$$\frac{d}{d\zeta}[\im \Phi_{31}(\zeta,\omega k_{4}(\zeta))] = -\im k_4(\zeta), \qquad \frac{d}{d\zeta}[\im \Phi_{32}(\zeta,\omega^{2} k_{2}(\zeta))]= \im k_2(\zeta),$$
this completes the proof of Theorem \ref{asymptoticsth}.

\appendix 

\section{Model RH problems}\label{appendix A}
Let $X \subset \C$ be the cross defined by $X = X_1 \cup \cdots \cup X_4$, where the rays
\begin{align} \nonumber
&X_1 := \bigl\{\rho e^{\frac{i\pi}{4}}\, \big| \, 0 \leq \rho  < \infty\bigr\}, && 
X_2 := \bigl\{\rho e^{\frac{3i\pi}{4}}\, \big| \, 0 \leq \rho  < \infty\bigr\},  &&
	\\ \label{Xdef}
&X_3 := \bigl\{\rho e^{-\frac{3i\pi}{4}}\, \big| \, 0 \leq \rho  < \infty\bigr\}, && 
X_4 := \bigl\{\rho e^{-\frac{i\pi}{4}}\, \big| \, 0 \leq \rho  < \infty\bigr\},
\end{align}
are oriented away from the origin. The proofs of the following two lemmas are similar to (but  more complicated than) the proof for the parabolic cylinder model problem that appears in \cite{DZ1993, I1981}. A detailed proof of Lemma \ref{II Xlemma 3 green} can be found in \cite{CLsectorV}.

\begin{lemma}[Model RH problem needed near $k= \omega k_{4}$ for Sector IV]\label{II Xlemma 3}
Let $q_{1},q_{3} \in \mathbb{C}$ be such that $1+|q_{1}|^{2}-|q_{3}|>0$, define
\begin{align*}
\nu_{1} = -\tfrac{1}{2\pi} \ln(1 + |q_{1}|^2), \qquad \nu_{3} = -\tfrac{1}{2\pi} \ln(1 + |q_{1}|^2 - |q_{3}|^{2}),
\end{align*}
and define the jump matrix $v^{X,(1)}(z)$ for $z \in X$ by
\begin{align}\label{II vXdef 3} 
v^{X,(1)}(z) = \begin{cases}
\begin{pmatrix} 1 & 0 & 0	\\
0 & 1 & 0 \\
 -\bar{q}_{3} z^{-2i \nu_{3}} z_{(0)}^{i \nu_{1}} z^{i \nu_{1}}e^{\frac{iz^{2}}{2}} & 0 & 1 \end{pmatrix}, &   z \in X_1, 
  	\\
\begin{pmatrix} 1 & 0 & \frac{-q_{3}}{1 + |q_{1}|^2 - |q_{3}|^{2}}z^{2i \nu_{3}} z_{(0)}^{-i \nu_{1}} z^{-i \nu_{1}}e^{-\frac{iz^{2}}{2}} 	\\
0 & 1 & 0 \\
0 & 0 & 1  \end{pmatrix}, &  z \in X_2, 
	\\
\begin{pmatrix} 1 & 0 & 0 \\
0 & 1 & 0 \\
\frac{\bar{q}_{3}}{1 + |q_{1}|^2 - |q_{3}|^{2}}z^{-2i \nu_{3}} z_{(0)}^{i \nu_{1}} z^{i \nu_{1}}e^{\frac{iz^{2}}{2}} & 0 	& 1 \end{pmatrix}, &  z \in X_3,
	\\
 \begin{pmatrix} 1 & 0	& q_{3} z^{2i \nu_{3}} z_{(0)}^{-i \nu_{1}} z^{-i \nu_{1}}e^{-\frac{iz^{2}}{2}}	\\
 0 & 1 & 0 \\
0 & 0	& 1 \end{pmatrix}, &  z \in X_4,
\end{cases}
\end{align}
where the branches for $z^{i\nu_{3}},z^{i\nu_{1}}$ lie on the negative real axis and the branch for $z_{(0)}^{i\nu_{1}}$ lies on the positive real axis, such that $z^{i\nu} = |z|^{i\nu}e^{-\nu  \arg(z)}$, $\arg(z) \in (-\pi,\pi)$ and $z_{(0)}^{i\nu} = |z|^{i\nu}e^{-\nu  \arg_{0}(z)}$, $\arg_{0}(z) \in (0,2\pi)$. Then the RH problem 
\begin{enumerate}[$(a)$]
\item $m^{X,(1)}(q_{1},q_{3}, \cdot) : \C \setminus X \to \mathbb{C}^{3 \times 3}$ is analytic;

\item on $X \setminus \{0\}$, the boundary values of $m^{X,(1)}$ exist, are continuous, and satisfy $m_+^{X,(1)} =  m_-^{X,(1)} v^{X,(1)}$;

\item $m^{X,(1)}(q_{1},q_{3},z) = I + O(z^{-1})$ as $z \to \infty$, and $m^{X,(1)}(q_{1},q_{3},z) = O(1)$ as $z \to 0$;
\end{enumerate}
has a unique solution $m^{X,(1)}(q_{1},q_{3}, z)$. This solution satisfies
\begin{align}\label{II mXasymptotics 3}
  m^{X,(1)}(z) = I + \frac{m^{X,(1)}_{1}}{z} + O\biggl(\frac{1}{z^2}\biggr), \; z \to \infty, \qquad m_{1}^{X,(1)} := \begin{pmatrix} 0 & 0 & \beta_{12}^{(1)} \\ 0 & 0 & 0 \\ \beta_{21}^{(1)} & 0 & 0 \end{pmatrix},
\end{align}  
where the error term is uniform with respect to $\arg z \in [0, 2\pi]$ and $q_{1},q_{3}$ in compact subsets of $\{q_{1},q_{3}\in \mathbb{C} \,|\, 1+|q_{1}|^{2}-|q_{3}|>0\}$, and $\beta_{12}^{(1)},\beta_{21}^{(1)}$ are defined by
\begin{align}\label{II betaXdef 3}
& \beta_{12}^{(1)} = \frac{e^{\frac{3\pi i}{4}}e^{\frac{3\pi\hat{\nu}_1}{2}}e^{2\pi \nu_{1}}\sqrt{2\pi}q_{3}}{(e^{2\pi \hat{\nu}_1}-1)\Gamma(-i\hat{\nu}_1)}, \qquad \beta_{21}^{(1)} = \frac{e^{-\frac{3\pi i}{4}}e^{\frac{3\pi \hat{\nu}_1}{2}}\sqrt{2\pi}\bar{q}_{3}}{(e^{2\pi \hat{\nu}_1}-1)\Gamma(i\hat{\nu}_1)},
\end{align}
where $\hat{\nu}_1 := \nu_{3}-\nu_{1} \geq 0$. 
(Note that $\beta_{12}^{(1)}\beta_{21}^{(1)} = \hat{\nu}_1$.)
\end{lemma}

\begin{lemma}[Model RH problem needed near $k=\omega^{2} k_{2}$]\label{II Xlemma 3 green}
Let $q_{2}, q_{4},q_{5},q_{6} \in \mathbb{C}$ be such that $1 + |q_{2}|^2 - |q_{4}|^{2}>0$, $1 - |q_{5}|^2 - |q_{6}|^{2}>0$, and 
\begin{align}\label{II condition on q2q4q5q6}
q_{4}-\bar{q}_5-q_{2}\bar{q}_6=0.
\end{align}
Define $\nu_2, \nu_4, \nu_5 \in \R$ by
\begin{align*}
\nu_{2} = -\tfrac{1}{2\pi} \ln(1 + |q_{2}|^2), \quad \nu_{4} = -\tfrac{1}{2\pi} \ln(1 + |q_{2}|^2 - |q_{4}|^{2}), \quad \nu_{5} = -\tfrac{1}{2\pi} \ln(1 - |q_{5}|^2 - |q_{6}|^{2}).
\end{align*}
Define the jump matrix $v^{X,(2)}(z)$ for $z \in X$ by
\begin{align}\label{II vXdef 3 green} 
v^{X,(2)}(z) = \begin{cases}
\begin{pmatrix} 
1 & 0 & 0 \\
0 & 1 & 0		\\
0 & -\frac{q_{6}-q_{2}q_{5}}{1+|q_{2}|^{2}} z^{-i (2\nu_{5}-\nu_{4})} z_{(0)}^{-i (2\nu_{2}-\nu_{4})} e^{\frac{iz^{2}}{2}} & 1 \end{pmatrix}, &   z \in X_1, 
  	\\
\begin{pmatrix}
1 & 0 & 0 \\
0 & 1 & -\frac{\bar{q}_6-\bar{q}_2\bar{q}_5}{1-|q_{5}|^{2}-|q_{6}|^{2}} z^{i (2\nu_{5}-\nu_{4})} z_{(0)}^{i (2\nu_{2}-\nu_{4})} e^{-\frac{iz^{2}}{2}} 	\\
0 & 0 & 1  \end{pmatrix}, &  z \in X_2, 
	\\
\begin{pmatrix} 
1 & 0 & 0 \\
0 & 1 & 0 \\
0 & \frac{q_{6}-q_{2}q_{5}}{1-|q_{5}|^{2}-|q_{6}|^{2}} z^{-i (2\nu_{5}-\nu_{4})} z_{(0)}^{-i (2\nu_{2}-\nu_{4})} e^{\frac{iz^{2}}{2}}	& 1 \end{pmatrix}, &  z \in X_3,
	\\
 \begin{pmatrix} 
1 & 0 & 0 \\ 
0 & 1	& \frac{\bar{q}_6-\bar{q}_2\bar{q}_5}{1+|q_{2}|^{2}} z^{i (2\nu_{5}-\nu_{4})} z_{(0)}^{i (2\nu_{2}-\nu_{4})} e^{-\frac{iz^{2}}{2}}	\\
0 & 0	& 1 \end{pmatrix}, &  z \in X_4,
\end{cases}
\end{align}
where $z^{i\nu}$ has a branch cut along $(-\infty,0]$ and $z_{(0)}^{i\nu}$ has a branch cut along $[0,+\infty)$ such that $z^{i\nu} = |z|^{i\nu}e^{-\nu  \arg(z)}$, $\arg(z) \in (-\pi,\pi)$, and $z_{(0)}^{i\nu} = |z|^{i\nu}e^{-\nu  \arg_{0}(z)}$, $\arg_{0}(z) \in (0,2\pi)$. Then the RH problem 
\begin{enumerate}[$(a)$]
\item $m^{X,(2)}(\cdot) = m^{X,(2)}(q_{2}, q_{4},q_{5},q_{6}, \cdot) : \C \setminus X \to \mathbb{C}^{3 \times 3}$ is analytic;

\item on $X \setminus \{0\}$, the boundary values of $m^{X,(2)}$ exist, are continuous, and satisfy $m_+^{X,(2)} =  m_-^{X,(2)} v^{X,(2)}$;

\item $m^{X,(2)}(z) = I + O(z^{-1})$ as $z \to \infty$, and $m^{X,(2)}(z) = O(1)$ as $z \to 0$;
\end{enumerate}
has a unique solution $m^{X,(2)}(z)$. This solution satisfies
\begin{align}\label{II mXasymptotics 3 green}
m^{X,(2)}(z) = I + \frac{m_{1}^{X,(2)}}{z} + O\biggl(\frac{1}{z^2}\biggr), \quad z \to \infty,  \quad m_{1}^{X,(2)} := \begin{pmatrix} 
0 & 0 & 0 \\
0 & 0 & \beta_{12}^{(2)} \\ 
0 & \beta_{21}^{(2)} & 0 \end{pmatrix}, 
\end{align}
where the error term is uniform with respect to $\arg z \in [-\pi, \pi]$ and $q_{2},q_{4},q_{5},q_{6}$ in compact subsets of $\{q_{2},q_{4},q_{5},q_{6}\in \mathbb{C} \,|\, 1 + |q_{2}|^2 - |q_{4}|^{2}>0, 1 - |q_{5}|^2 - |q_{6}|^{2}>0, q_{4}-\bar{q}_5-q_{2}\bar{q}_6=0\}$, and 
\begin{align}\label{II betaXdef 3 green}
& \beta_{12}^{(2)} := \frac{e^{\frac{3\pi i}{4}}e^{\frac{\pi\hat{\nu}_2}{2}}e^{2\pi (\nu_{4}-\nu_{2})}\sqrt{2\pi}(\bar{q}_6-\bar{q}_2\bar{q}_5)}{(e^{\pi \hat{\nu}_2}-e^{-\pi \hat{\nu}_2})\Gamma(-i\hat{\nu}_2)}, \qquad \beta_{21}^{(2)} := \frac{e^{-\frac{3\pi i}{4}}e^{\frac{\pi \hat{\nu}_2}{2}}e^{2\pi\nu_{2}}\sqrt{2\pi}(q_{6}-q_{2}q_{5})}{(e^{\pi \hat{\nu}_2}-e^{-\pi \hat{\nu}_2})\Gamma(i\hat{\nu}_2)},
\end{align}
where $\hat{\nu}_2 := \nu_2 + \nu_5 - \nu_4$. (Note that $\beta_{12}^{(2)}\beta_{21}^{(2)} = \hat{\nu}_{2}$.)
 
%We can verify using
%\begin{align*}
%|\Gamma(i\nu)| = \frac{\sqrt{2\pi}}{\sqrt{\nu(e^{\pi\nu}-e^{-\pi\nu})}}, \qquad \nu \in \mathbb{R}\setminus \{0\}
%\end{align*}
%that $\beta_{12}^{(2)}\beta_{21}^{(2)} = \nu = \nu_{2}+\nu_{5}-\nu_{4}$.
\end{lemma}

\subsection*{Acknowledgements}
Support is acknowledged from the Novo Nordisk Fonden Project, Grant 0064428, the European Research Council, Grant Agreement No. 682537, the Swedish Research Council, Grant No. 2015-05430, Grant No. 2021-04626, and Grant No. 2021-03877, and the Ruth and Nils-Erik Stenb\"ack Foundation.

\bibliographystyle{plain}
\bibliography{is}

\end{document}